\documentclass[12pt]{amsart}

\usepackage{pdfsync}       
\usepackage{tikz}
\usepackage{enumerate}
\usepackage{comment}
\usepackage{amssymb}
\usepackage{mathtools}

\usepackage[OT1]{fontenc}    %
\usepackage{graphicx}

\usepackage{palatino}
\usepackage[left=3cm,top=3.8cm,right=3cm]{geometry}        
\usepackage{xcolor}
\usepackage[pagebackref]{hyperref}
\hypersetup{
   colorlinks,
    linkcolor={red!60!black},
    citecolor={blue!60!black},
    urlcolor={blue!90!black}
}

\newcommand{\be}{\begin{equation}}
\newcommand{\ee}{\end{equation}}

\numberwithin{equation}{section}

\theoremstyle{plain}
\newtheorem{theorem}{Theorem}[section]
\newtheorem{corollary}[theorem]{Corollary}
\newtheorem{conjecture}[theorem]{Conjecture}
\newtheorem{prop}[theorem]{Proposition}
\newtheorem{lemma}[theorem]{Lemma}

\theoremstyle{remark}
\newtheorem{remark}[theorem]{Remark}

\theoremstyle{definition}
\newtheorem{definition}[theorem]{Definition}

\newcommand{\e}{\varepsilon}
\newcommand{\N}{\mathbb{N}}
\newcommand{\R}{\mathbb{R}}
\newcommand{\Z}{\mathbb{Z}}
\newcommand{\dist}{\mathrm{dist}}

\newcommand{\cP}{\mathcal{P}}
\newcommand{\cA}{\mathcal{A}}

\newcommand{\cM}{\mathcal{M}}

\newcommand{\cC}{\mathcal{C}}

\newcommand{\cD}{\mathcal{D}}

\newcommand{\cH}{\mathcal{H}}

\newcommand{\cE}{\mathcal{E}}

\newcommand{\eps}{\varepsilon}

\DeclareMathOperator{\dir}{dir}

\DeclareMathOperator{\adim}{dim_A}
\DeclareMathOperator{\hdim}{dim_H}
\DeclareMathOperator{\pdim}{dim_P}

\DeclareMathOperator{\lbdim}{\underline{\dim}_B}
\DeclareMathOperator{\ubdim}{\overline{\dim}_B}

\DeclareMathOperator{\supp}{supp}

\DeclareMathOperator{\bad}{\mathbf{Bad}}

\newcommand{\wh}{\widehat}
\newcommand{\wt}{\widetilde}

\title{On the distance sets spanned by sets of dimension $d/2$ in $\mathbb{R}^d$}

\author{Pablo Shmerkin and Hong Wang}

\address{Department of Mathematics, the University of British Columbia}
\email{pshmerkin@math.ubc.ca}
\urladdr{http://www.pabloshmerkin.org}
\thanks{PS was supported by an NSERC discovery grant}

\address{Department of Mathematics, UCLA}
\curraddr{Courant Institute of Mathematical Sciences, NYU}
\email{hw3639@nyu.edu}
\urladdr{https://sites.google.com/view/hongwang/home}
\thanks{HW was supported by the NSF grant DMS-20220032 and DMS-1926686}


\begin{document}

\begin{abstract}
	We establish the dimension version of Falconer's distance set conjecture for sets of equal Hausdorff and packing dimension (in particular, for Ahlfors-regular sets) in all ambient dimensions. In dimensions $d=2$ or $3$, we obtain the first explicit improvements over the classical $1/2$ bound for the dimensions of distance sets of general Borel sets of dimension $d/2$. For example, we show that the set of distances spanned by a planar Borel set of Hausdorff dimension $1$ has Hausdorff dimension at least $(\sqrt{5}-1)/2\approx 0.618$. In higher dimensions we obtain explicit estimates for the lower Minkowski dimension of the distance sets of sets of dimension $d/2$. These results rely on new estimates for the dimensions of radial projections that may have independent interest.
\end{abstract}


\maketitle

\tableofcontents

\section{Introduction and main results}

\subsection{Progress towards Falconer's distance set problem}
\label{subsec:intro-distance}

Given a set $X\subset\mathbb{R}^d$, its \emph{distance set} is
\[
	\Delta(X) = \{ |x-y|: x,y\in X\}.
\]
It is a classical problem to understand how the sizes of $X$ and $\Delta(X)$ are related. When size is measured by cardinality, this is the famous Erd\H{o}s distinct distances problem. In the plane, it was solved in a breakthrough article of L. Guth and N. Katz \cite{GuthKatz15}: a planar set of $N$ points determines at least $C_\e N^{1-\e}$ distances, and this is sharp. The problem remains open in all dimensions $\ge 3$.

In the 1980s, Falconer \cite{Falconer85} introduced a ``fractal'' variant of the problem, that has become one of the best known problems in geometric measure theory. All sets are assumed Borel. A conjecture that emerges from Falconer's work is that if $X\subset\R^d$, $d\ge 2$ is a set with $\hdim(X)\ge d/2$, then
\[
	\hdim(\Delta(X)) = 1.
\]
A closely related conjecture asserts that if $\hdim(X)>d/2$, then $\Delta(X)$ has positive Lebesgue measure. Both versions of the conjecture remain open in all dimensions. Falconer showed, via Cantor sets based on lattices, that the threshold $d/2$ would be sharp in both cases.

A stronger variant of the problem asks whether there is $y\in X$ such that the \emph{pinned distance set} $\Delta^y(X)$, where $\Delta^y(x)=|x-y|$, has full dimension or positive measure when $\hdim(X)=d/2$ or $\hdim(X)>d/2$, respectively.

There are three directions in which one can attempt to make progress on these problems: (i) find thresholds $s_d$ such that if $\hdim(X)>s_d$ then $|\Delta(X)|>0$; (ii) find lower bounds for $\hdim(\Delta(X))$ assuming only that $\hdim(X)=d/2$ or $\hdim(X)>d/2$; (iii) get stronger results for special classes of sets. Considerable efforts by many mathematicians have been devoted to direction (i). The current ``world records'' were obtained in \cite{DuZhang19, GIOW20, DIOWZ21}:
\[
	s_d = \left\{
	\begin{array}{lll}
		\frac{d}{2}+\frac{1}{4}                & \text { if } & \text{$d$ is even} \\
		\frac{d}{2}+\frac{1}{4}+\frac{1}{8d-4} & \text{ if }  & \text{$d$ is odd}
	\end{array}
	\right..
\]
In fact, even the pinned version holds: if $X\subset\R^d$ satisfies $\hdim(X)>s_d$, then there is $y\in X$ such that $|\Delta^y(X)|>0$. The insight to derive a pinned version in this context is due to B. Liu \cite{Liu19}. We remark that some of the methods leading to these estimates also provide information on the dimension of $\Delta(X)$ when $\hdim(X)>d/2$ (see \cite{Liu20}). However, the assumption $\hdim(X)>d/2$ is really crucial for these approaches and cannot be relaxed to $\hdim(X)=d/2$.

\subsubsection{Explicit bounds for general sets}
\label{subsub:general-sets}

In this work we do not have anything to say about direction (i), but make progress on directions (ii) and (iii), which have also attracted considerable attention. In his original paper, Falconer proved that if $\hdim(X)=d/2$, then $\hdim(\Delta(X))\ge 1/2$. There are various reasons why the threshold $1/2$ is a natural barrier to overcome. If $R$ was a Borel subring of the reals of dimension $1/2$, then $R^d$ would be a set of dimension $\ge d/2$ whose distance set has dimension $\le 1/2$. As it turns out, such subrings do not exist, but this was an open problem for decades (and the discretized sum-product problem, which can be seen as a quantitative version of non-existence of Borel subrings of the reals, continues to be a challenging open problem).  For polygonal norms the value $1/2$ is also sharp, so the curvature of the Euclidean norm must play a r\^{o}le in any attempt to improve upon it. Last, but not least, for what is perhaps the most natural discretization of the problem at a small scale $\delta$, the threshold $1/2$ is actually sharp, as witnessed by the ``train track'' example, see the discussion in \cite[\S 1.3]{KatzTao01}.

Despite these challenges, it follows from the combined efforts of N. Katz and T. Tao \cite{KatzTao01} and J. Bourgain \cite{Bourgain03} that there exists a universal $c>0$ such that if $X\subset\R^2$ is a Borel set of Hausdorff dimension $1$, then $\hdim(\Delta(X))\ge 1/2+c$. More precisely, Katz and Tao proved that a certain discretized sum-product expansion implies the above result on distance sets; Bourgain then proved the discretized sum-product conjecture. The value of $c$, although effective in principle, is not explicit; in any case, an inspection of the proofs suggests it would be tiny. In \cite{Shmerkin23}, the first author obtained a  pinned version (with $\Delta(X)$ replaced by $\Delta^y(X)$), an extension to $\R^3$ (for sets of dimension $3/2$), and a generalization to  $C^2$ norms of non-zero Gaussian curvature. These improvements ultimately rely on Bourgain's machinery and so the value of $c$ is also not explicit.

On the other hand, T. Keleti and the first author \cite{KeletiShmerkin19} proved that if $X$ is a planar Borel set with $\hdim(X)>1$, then there is $y\in X$ such that $\hdim(\Delta^y X)> 2/3$. This result was recovered by B. Liu \cite{Liu20} using the approach of \cite{GIOW20}, and then slightly improved by the first author \cite{Shmerkin21} by combining both approaches. The assumption $\hdim(X)>1$ is essential in these results; the methods do not work when $\hdim(X)=1$.

In this article, we obtain the first explicit estimate of the constant ``$c$'' in Bourgain's Theorem, and also its analog in $\R^3$. Given $u\in (0,1]$, let $\phi(u)$ be the positive root of $x^2+(2-u)x-u=0$. Explicitly,
\begin{equation} \label{eq:def-phi}
	\phi(u)=\frac{u}{2}+\frac{\sqrt{4+u^2}-2}{2}.
\end{equation}
In particular, $\phi(1)=0.618\ldots$ is the (reciprocal) golden mean. The $\frac{u}{2}$-term in \eqref{eq:def-phi} is the ``base'' estimate for the Falconer's distance set problem (see \cite{KatzTao01}) and also, more recently, the radial projection problem \cite{Orponen19}. 

\begin{theorem} \label{thm:planar-distance}
	Let $X\subset\R^2$ be a Borel set with $\hdim(X)= u\le 1$. Then
	\[
		\sup_{y\in X} \hdim(\Delta^y(X)) \ge \phi(u).
	\]
\end{theorem}

\begin{theorem} \label{thm:dim3}
	If $X\subset\R^3$ is a Borel set of Hausdorff dimension $\ge 3/2$, then
	\[
		\sup_{y\in X}\hdim(\Delta^y X) \ge \frac{1}{2}+\frac{\phi(1/2)}{4} > 0.57.
	\]
\end{theorem}

In higher dimensions, we are able to obtain explicit estimates for the \emph{lower Minkowski} (box counting) dimension (denoted $\lbdim$) of $\Delta^y(X)$, in place of the Hausdorff dimension:
\begin{theorem}  \label{thm:box-dim-high-dim}
	Let $d\ge 4$ and let $X\subset\R^d$ be a Borel set with $\hdim(X)\ge d/2$. Then
	\[
		\sup_{y\in X}\lbdim(\Delta^y(X))\ge \frac{1}{2}+c_d,
	\]
	where
	\[
		c_d = \left\{
		\begin{array}{ccc}
			\frac{\phi(1/2)}{d+1}   & \text{ if } & d \text{ is odd}  \\
			                        &             &                   \\
			\frac{\phi(1)-1/2}{d+1} & \text{ if } & d \text{ is even}
		\end{array}
		\right..
	\]
\end{theorem}
Our approach also shows that the same bound holds for the Hausdorff dimension under a mild ``non-concentration near hyperplanes'' assumption, see the discussion in \S\ref{subsubsec:remarks-box}.

\subsubsection{Falconer's conjecture for sets of equal Hausdorff and packing dimension}

There has also been much interest in finding improved distance set estimates for special classes of sets. T. Orponen \cite{Orponen17} proved that if $X\subset\R^2$ is an Ahlfors-regular set of dimension $\ge 1$, then the \emph{packing} dimension  (denoted $\pdim$) of $\Delta(X)$ is $1$. We refer to \cite[Chapter 5]{Mattila95} for the definition and main properties of packing dimension; here we only recall that it lies between Hausdorff and upper Minkowski dimension. The first author \cite{Shmerkin19} improved Orponen's result when $\hdim(X)>1$: assuming only that $\hdim(X)=\pdim(X)> 1$ (a class which is substantially larger than that of Ahlfors-regular sets), there exists $y\in X$ such that $\hdim(\Delta^y(X))=1$. This result was recovered and generalized in \cite{KeletiShmerkin19}.

On the other hand, L. Guth, N. Solomon and the second author \cite{GSW19} proved that Falconer's conjecture (essentially) holds at a small scale $\delta$ for sets which are maximally separated at that scale; this class of sets is in some sense at the opposite end of the spectrum from Ahlfors-regular sets.

In this article, we prove that the dimension version of Falconer's conjecture holds for sets of equal Hausdorff and packing dimension $d/2$, in all dimensions $d\ge 2$:
\begin{theorem} \label{thm:distance-conj-Ahlfors}
	Fix $d\ge 2$. Let $X\subset\R^d$ be a Borel set with $\hdim(X)=\pdim(X)=d/2$. Then for every $\eta>0$,
	\[
		\hdim\{ y\in X: \hdim(\Delta^y X)<1-\eta\} <\hdim(X).
	\]
	In particular,
	\[
		\sup_{y\in X}\hdim(\Delta^y X) =1,
	\]
	and if $\mathcal{H}^{d/2}(X)>0$, then the supremum is realized at $\mathcal{H}^{d/2}$-almost all $y\in X$.

	In the planar case $d=2$, we further have that if $\hdim(X)=\pdim(X)\le 1$, then
	\[
		\sup_{y\in X}\hdim(\Delta^y X) =\hdim(X).
	\]
\end{theorem}

All the results discussed so far hold also when distances are taken with respect to any $C^\infty$ norm of positive Gaussian curvature everywhere, see \S\ref{subsec:othernorms}.

\subsection{Methods and radial projections}

\subsubsection{From radial projections to distance sets}

In order to establish our results, we follow a multiscale linearization approach that goes back, in various forms, to \cite{HochmanShmerkin12, Orponen17, KeletiShmerkin19, Shmerkin23}. Roughly speaking, the idea is the following: let $\mu,\nu$ be Frostman measures on the set $X$ of interest with disjoint supports (note that we have already used the measurability of $X$ in assuming the existence of Frostman measures). We fix a small scale $\delta$ and are interested in obtaining lower bounds for the ``size'' of $\Delta^y(\mu)$, the pushforward measure of $\mu$ under $\Delta^y$,  at scale $\delta$ for a $\nu$-typical point. We use entropy (or a robust version of entropy) to quantify ``size at scale $\delta$''.

For $\theta\in \R^d\setminus\{0\}$, let $P_\theta:\R^d\to\R$ be projection in direction $\theta$. One key insight from the aforementioned articles is that the entropy of $\Delta^y(\mu)$ at scale $\delta$ can be estimated from an average of entropies of $P_{\tfrac{d\Delta^y(x)}{dx}}\mu^Q$, where $x$ is chosen according to $\mu$ and $Q$ are dyadic cubes containing $x$ and $\mu^Q= \tfrac{1}{\mu(Q)} T_Q(\mu|Q)$, with $T_Q$ the homothethy that maps $Q$ to the unit cube. The scales of the cubes $Q$ can be chosen freely, subject to certain bounds. See Propositions \ref{prop:entropy-of-image-measure} and \ref{prop:entropy-of-image-measure-robust} for precise statements.

Note that for any $x\neq y$,
\[
	\frac{d\Delta^y(x)}{dx} = \frac{x-y}{|x-y|} =: \pi_y(x).
\]
Here $\pi_y: \mathbb{R}^d\setminus \{y\} \rightarrow \mathbb{S}^{d-1}$ is called the \emph{radial projection} and denote $\pi_y(X):= \pi_y(X\setminus \{y\})$.
In other words, because the derivative of the distance map is $\pi_y$, the problem of estimating the size of the single \emph{non-linear} projection $\Delta^y\mu$ is reduced to that of studying a family of \emph{linear} projections in directions given by the \emph{radial} projections with center $y$. Crucially, we have the freedom to pick the point $y\in\supp\nu$.

As this rough description suggests, in order to obtain effective estimates with this approach one needs to tackle three different problems:
\begin{enumerate}[(a)]
	\item  \label{it:a:radial} Estimate the size of \emph{radial} projections; one needs to know that radial projections are ``well spread out'' in order to improve the estimates in the next step;
	\item  \label{it:b:linear}  Estimate the size of \emph{linear} projections, where the set of directions is possibly quite small (it arises from the size of the radial projections from the previous item);
	\item  \label{it:c:combinatorial} For each measure $\mu$, optimize the choice of scales of the cubes $Q$ arising in the multiscale decomposition, subject to the constraints of Propositions \ref{prop:entropy-of-image-measure} and \ref{prop:entropy-of-image-measure-robust}.
\end{enumerate}
We remark that part \eqref{it:c:combinatorial} only really arises if one seeks estimates for general sets; for Ahlfors-regular sets, for example, all the measures $\mu^Q$ essentially have the same size so the choice of scales is far less important.

Radial projections are also crucial in the recent harmonic-analytic approaches to the Falconer problem: see \cite{GIOW20, DIOWZ21}. They are also a natural object of study for their own sake. Two important results on radial projections were obtained by T.~Orponen in \cite{Orponen19}:
\begin{itemize}
	\item
	      Theorem 1.13 in \cite{Orponen19} (or rather its proof) says that if $\mu,\nu$ are Frostman measures of exponent $>d-1$ in $\R^d$, then for $\nu$-almost all $y$, the radial projection $\pi_y\mu$ has an $L^p$ density, for some fixed $p>1$. This theorem plays a key role in \cite{GIOW20, DIOWZ21} and also in \cite{KeletiShmerkin19}, which was the first paper in which the full strategy involving \eqref{it:a:radial}--\eqref{it:c:combinatorial} was implemented. In Orponen's theorem, the assumption that the Frostman exponent is $>d-1$ is necessary, as the example of $\mu,\nu$ supported on the same hyperplane shows. This is the reason why in \cite{KeletiShmerkin19} the assumption $\hdim(X)>1$ is essential, and why also the methods of \cite{GIOW20, DIOWZ21} do not say anything for sets of dimension $\le d/2$.
	\item
	      The second radial projection theorem from \cite{Orponen19}, Theorem 1.5, roughly says that if $\mu,\nu$ are planar measures of Frostman exponent $s>0$ and are not both supported on the same line, then $\pi_y(\mu)$ is at least $s/2$ dimensional for $\mu$-positively many $y$. This theorem (or rather its proof) plays a key role in the implementation of the strategy \eqref{it:a:radial}--\eqref{it:c:combinatorial} in \cite{Shmerkin19}, which as described above leads to the bound $\hdim(\Delta^y(X))\ge 1/2+c$ for $X\subset\R^2$ with $\hdim(X)=1$. However, by itself it is not enough to derive explicit estimates for the value of $c$.
\end{itemize}

\subsubsection{New results on radial projections}

The main new ideas in this article concern the dimension of radial projections, that is, part \eqref{it:a:radial} of the strategy above. We also have to grapple with new combinatorial problems for part \eqref{it:c:combinatorial}. For part \eqref{it:b:linear}, we appeal to a variety of existing, older and newer, linear projection results.

Let $X\subset\R^d$ be a Borel set. If $\hdim(X)>d-1$, it follows from Orponen's results that $\pi_y(X)$ has positive measure (in $S^{d-1}$) for many $y\in X$. If $\hdim(X)=d-1$, then this can no longer be true, since $X$ could be contained in a hyperplane. Nevertheless, the following conjecture seems plausible (the planar case was stated in \cite{Orponen19}):
\begin{conjecture} \label{conj:radial}
	Let $X\subset\R^d$ be a Borel set with $\hdim(X)\le d-1$ which is not contained in a hyperplane. Then
	\[
		\sup_{y\in X}\hdim(\pi_yX))=\hdim(X).
	\]
\end{conjecture}
This conjecture remains wide open; it is probably of a similar level of difficulty to the Falconer distance set problem. However, in this article we make some progress towards it, both for general sets and for sets of equal Hausdorff and packing dimension. We start with the latter:

\begin{theorem} \label{thm:radial-main-AD}
	Fix $k\in\{1,\ldots,d-1\}$. Let $X\subset \R^d$ be a Borel set which is not contained in a $k$-plane, and such that $\hdim(X)=\pdim(X)\in (k-1/k-\eta(k,d),k]$, where $\eta(k,d)>0$ is a small constant if $k\ge 2$, and $\eta(1,d)=0$. Then,
	\[
		\sup_{y\in X} \hdim(\pi_y X) = \hdim(X).
	\]
\end{theorem}

Note that in the plane this establishes Conjecture \ref{conj:radial} for sets of equal Hausdorff and packing dimension (while previous approaches to this problem do not seem to distinguish even Ahlfors-regular sets from general sets). Note also that Theorem \ref{thm:radial-main-AD} also yields the optimal estimate for sets of equal Hausdorff and packing dimension $d-1$ which are not contained in a hyperplane.

For general sets, we are able to prove:
\begin{theorem} \label{thm:radial-main-gral}
	Fix $k\in\{1,\ldots,d-1\}$. Let $X\subset\R^d$ be a Borel set with
	\[
		\hdim(X) =s \in (k-1/k-\eta(k,d),k],
	\]
	where $\eta(k,d)>0$ is a small constant if $k\ge 2$, and $\eta(1,d)=0$. Assume that $X$ is not contained in any $k$-plane. Then, for the function $\phi$ defined in \eqref{eq:def-phi},
	\[
		\sup_{y\in X} \hdim(\pi_y X) \ge  k-1+\phi(s-k+1).
	\]
\end{theorem}
It is not hard to check that this result improves, in many cases substantially, upon \cite[Theorem 6.15]{Shmerkin23} in the case $\alpha=\kappa\in (d-2,d)$. Also significantly, it removes the restriction that $\hdim(X)>d-2$. We remark that, on the other hand, we rely on (a quantitative version of)  \cite[Theorem 6.15]{Shmerkin23} in our proof of Theorem \ref{thm:radial-main-gral}. Note that in the plane, for sets not contained in a line, the numerology matches exactly that of Theorem \ref{thm:planar-distance}.

We remark that, as a by-product of the proofs, we are also able to obtain discretized versions (that is, involving a fixed small scale $\delta$ rather than actual fractals) for the planar versions of all the results discussed above: see Corollary \ref{cor:discretized}. In higher dimensions our methods are less explicit and so we currently do not have a satisfactory discretized formulation.

\subsubsection{Some ingredients in the proofs}

We briefly and informally describe some of the new insights that go into the proofs of our main results. More quantitative versions of Theorems \ref{thm:radial-main-AD} and \ref{thm:radial-main-gral} are the ingredients required for part \eqref{it:a:radial} of the general strategy to prove the distance set estimates stated in \S\ref{subsec:intro-distance}. As explained above, these radial projections estimates are also the main innovations in this article.

Our approach to the size of radial projections is also based on the steps \eqref{it:a:radial}--\eqref{it:c:combinatorial}. Indeed, the process we described to estimate from below the entropy of $\Delta^y(\mu)$ also works if $\Delta^y$ is replaced by any family of $C^2$ functions $(F_y)_{y\in Y}$; the only difference is that of course $\tfrac{d\Delta^y}{dx}$ has to be replaced by $\tfrac{d F_y}{dx}$. In the case $F_y= \pi_y$, we have $\tfrac{d \pi_y}{dx}=\pi_y(x)^\perp$ (in $\R^d$ this corresponds to a projection onto $\R^{d-1}$).

Let us consider the planar case first. A priori the discussion above could suggest that one would end in a circular argument, since estimates on radial projections go into estimates on radial projections. We show that, on the other hand, it is possible to bootstrap the steps \eqref{it:a:radial}--\eqref{it:c:combinatorial}, achieving a small gain at each step of the iteration. The small (and non-quantitative) gain is made possible by a very recent (linear) projection estimate of T.~Orponen and the first author \cite{OrponenShmerkin23} (which enters in part \eqref{it:b:linear} of the argument). As the starting point for the bootstrapping we appeal to (the proof of) \cite[theorem 1.5]{Orponen19}. For sets of equal Hausdorff and packing dimension, we are able to bootstrap all the way up to the dimension of the original set or measure, yielding the planar case of Theorems \ref{thm:radial-main-AD} and \ref{thm:distance-conj-Ahlfors}. For general sets, the bootsrapping argument has to be combined with the combinatorial part \eqref{it:c:combinatorial} of the scheme; as a result, the combinatorial problem we need to tackle is somewhat more involved than similar ones in \cite{KeletiShmerkin19, Shmerkin23}. This strategy is carried out in Section \ref{sec:plane}.

Now let us discuss the case of dimension $d\ge 3$. Let $\mu,\nu$ be compactly supported measures on $\R^d$ with disjoint supports. A simple but important observation is that the dimension of $\pi_y\mu$ for $y\in\supp(\nu)$ is bounded below by the dimension of $\pi_{Py}(P\mu)$ for any linear projection $P:\R^d\to\R^k$. This is simply because the projection of a tube in $\R^d$ is contained in a tube in $\R^k$, and the separation between $\supp(\mu)$ and $\supp(\nu)$ allows us to replace  cones by comparable tubes. This observation was first applied to Falconer's problem in \cite{DIOWZ21}. One of the main contributions of this article is a significant refinement of this idea: if, in addition to a lower bound on the size of $\pi_{Py}(P\mu)$, one has non-trivial bounds on radial projections on typical slices of $\mu$ with $(d-k+1)$-dimensional planes ``orthogonal'' to $P$, then both these bounds can be combined to obtain an improved estimate for $\pi_y\mu$ (throughout this discussion, it should be understood that $y$ is a $\nu$-typical point). Although the geometric picture behind this idea is not difficult, carrying it to fruition involves a fair amount of work using tools from geometric measure theory such as the Marstrand-Mattila slicing theorem and a variant of it for cylindrical projections. The outcome is Theorem \ref{thm:radial-inductive}, providing a mechanism to obtain radial projection estimates by inducing on the dimension.

Now, for Theorem \ref{thm:radial-inductive} to be of any use, one needs to establish non-trivial radial projection estimates on typical slices of a measure with a lower dimensional subspace. For simplicity (and also because it is the case that ends up being most relevant to our distance set estimates) suppose $\mu$ is a Borel measure on $\R^3$ with Frostman exponent $s\in (1,2)$. Then the sliced measures with a typical plane (in a suitable sense) have dimension $s-1$, so in particular they are not trivial. But because $s-1<1$, a priori it could happen that the sliced measures are almost all supported on lines, in which case all radial projections would be trivial. We suspect that this actually cannot happen as soon as $s>1$, but can only prove it (by appealing to a result from \cite{Shmerkin23} that was already mentioned before) if $s>3/2-\eta$, where $\eta>0$ is a small absolute constant. Once we know that the sliced measures are not supported on lines, we can combine the planar results discussed above with Theorem \ref{thm:radial-inductive} to deduce Theorem \ref{thm:dim3}, and the case $d=3$ of Theorems \ref{thm:distance-conj-Ahlfors} and \ref{thm:radial-main-gral}. Similar considerations apply in higher dimensions, and help explain why in Theorem \ref{thm:radial-main-gral} we get information only if $\hdim(X)$ is in a certain interval.

To conclude the introduction, we comment on how our approach overcomes the difficulties in bypassing the $1/2$ threshold that we discussed in \S\ref{subsub:general-sets}. The non-existence of Borel subrings of the reals, in the quantitative form expressed by Bourgain's discretized projection and sum-product theorems \cite{Bourgain03, Bourgain10}, is behind some of the main tools we use, including the ``$\e$-improvements'' to Kaufman's projection theorem from \cite{OrponenShmerkin23} and the radial projection theorem from \cite{Shmerkin23}. The curvature of the sphere shows up in the radial projection estimates (``radial'' projections for polyhedral norms are concentrated on the finite set of vectors normal to the unit ball; while the arguments extend without much trouble to smooth curved norms). Finally, the reason why train-tracks aren't a serious issue in our framework is the following: let $\mu$ be a Frostman measure on the target compact set $X\subset\R^2$. If $\mu$ gives positive mass to a line, the distance set problem becomes trivial. Otherwise, by compactness, there is a scale $\delta=\delta(\mu)$ such that any $\delta$-tubes has $\mu$-mass $\le c$, where $c$ is a small constant depending only on $\hdim(X)$. This means that $X$ may look like a train track at scale $\delta$, but not at substantially smaller scales. Indeed, a result of T.~Orponen \cite{Orponen19} (that we have already discussed several times, and is stated in the form we need as Proposition \ref{prop:Orponen-radial}) implies that if $c$ is taken small enough, then $X$ has ``positive dimensional'' radial projections at scales $\le \delta'(\mu)\ll\delta(\mu)$. As explained earlier, this result of Orponen is the initial step for our bootstrapping argument in the plane, which is in turn used (by inducing on the dimension) also for our higher-dimensional results. In all our discretizations of a measure $\mu$, we only consider scales $\le \delta'(\mu)$, ensuring that train-track-like configurations do not arise.

\subsection{Organization of the paper}

In Section \ref{sec:notation-and-preliminary} we set up notation to be used throughout the paper and recall a number of known preliminaries that will feature in many of the later arguments. Section \ref{sec:projection-thms} collects a number of known (linear) projection theorems in the form that we need later in the paper; put together, these theorems are the heart of step \eqref{it:b:linear} in the strategy outlined above. In Section \ref{sec:abstract-nonlinear} we implement the steps \eqref{it:a:radial}--\eqref{it:c:combinatorial} at an abstract level, for general parametrized families of $C^2$ maps (satisfying some regularity assumptions). The ideas here can be traced back to \cite{KeletiShmerkin19, Shmerkin23} but we believe that the abstract formulation, in addition to avoiding having to repeat many similar arguments several times, will help clarify the main features of the method. We also pay special attention to some quantitative dependencies among the many parameters involved, with a view on future applications.

In Section \ref{sec:plane}, we establish the planar case of our main results on distance sets and radial projections, by implementing the bootstrapping argument we outlined above. In Section \ref{sec:high-dim-radial}, we discuss radial projections in higher dimensions, including our main scheme to induce on the dimension (the already mentioned Theorem \ref{thm:radial-inductive}), and deduce the case $d\ge 3$ of Theorems \ref{thm:radial-main-AD} and \ref{thm:radial-main-gral}. In Section \ref{sec:high-dim-dist}, all the pieces are put together to conclude the proofs of our estimates for Falconer's problem (Theorems \ref{thm:dim3}, \ref{thm:box-dim-high-dim} and \ref{thm:distance-conj-Ahlfors}). We conclude the paper in Section \ref{sec:extensions} with some further remarks and generalizations.

\subsection{On the proof of Theorem \ref{thm:distance-conj-Ahlfors}}

The proof of Theorem \ref{thm:distance-conj-Ahlfors} on the distance sets of sets of equal Hausdorff and packing dimension highlights many of the main ideas in this paper, while avoiding technical combinatorial difficulties that arise when dealing with general sets. The reader may want to focus on this proof, or even just the planar case, in a first read.

Theorem \ref{thm:distance-conj-Ahlfors} is proved in \S\ref{subsec:high-dim-dist-Ahlfors}, while its planar case is established in  \S\ref{subsec:proof-planar-AD}. The following sections are only required for the general case, and can be skipped in a first reading:
\begin{itemize}
	\item \S\ref{subsec:Kaufman}, \S\ref{subsec:Bourgain},
	\item All of Section \ref{sec:abstract-nonlinear} from \S\ref{subsec:def-comb} to the end, except for the case $\pdim(X)\le \hdim(X)+\e$ of Theorem \ref{thm:abstract-proj-Hausdorff},
	\item All of Section \ref{sec:plane} from \S\ref{subsec:radial-general} to the end, except for Corollary \ref{cor:planar-thin-tubes}.
\end{itemize}

\subsection*{Acknowledgement}

The authors thank the anonymous referees for their careful reading of the manuscript and numerous helpful suggestions that improved the presentation of the results.

\section{Preliminaries}
\label{sec:notation-and-preliminary}

\subsection{Notation}
\label{subsec:notation}

We use Landau's $O(\cdot)$ notation: given $X>0$, $O(X)$ denotes a positive quantity bounded above by $C X$ for some constant $C>0$. If $C$ is allowed to depend on some other parameters, these are denoted by subscripts. We sometimes write $X\lesssim Y$ in place of $X=O(Y)$ and likewise with subscripts. We write $X\gtrsim Y$,  $X\approx Y$ to denote $Y\lesssim X$,  $X\lesssim Y\lesssim X$ respectively.

The family of compactly supported, non-trivial, finite Borel measures on a metric space $X$ is denoted by $\cM(X)$, and the union of $\cM(X)$ and the trivial measure is denoted by $\cM_0(X)$. We let $|\mu|=\mu(X)\in [0,\infty)$ denote the total mass of a measure $\mu\in\cM_0(X)$.

We let $\cP(X)=\{\mu\in\cM(X): |\mu|=1\}$ be the family of compactly supported probability measures on $X$, and $\cP_0(X)=\{\mu\in\cM_0(X): |\mu|\le 1\}$ be the family of compactly supported sub-probability measures on $X$ (including the trivial measures).


If $\mu\in\cM(X)$ and $\mu(Y)>0$, then $\mu_Y = \tfrac{1}{\mu(Y)}\mu|_Y\in \cP(X)$ is the normalized restriction of $\mu$ to $Y$.

If a measure $\mu\in\cP(\R^d)$ has a density in $L^p$, then its density is sometimes also denoted by $\mu$, and in particular  $\|\mu\|_p$ stands for the $L^p$ norm of its density.

We use $|\cdot|$ to denote both Lebesgue measure and cardinality; the meaning will be clear from context.

We denote the open ball of center $x$ and radius $r$ by $B(x,r)$ (in any metric space, usually $\R^d$). An arbitrary ball of radius $r$ is denoted by $B_r$. An $r$-tube is an open tube of radius $r$ centered at a line. The set we consider is usually compact, so the underlying $r$-tubes have finite length.

The family of all compact subsets of $\R^d$, endowed with the Hausdorff metric, will be denoted by $\cC(\R^d)$.

We let $\cD_j$ be the family of half-open $2^{-j}$-dyadic cubes in $\R^d$ (where $d$ is understood from context), and given $X\subset\R^d$, we denote the set of cubes in $\cD_j$ intersecting $X$ by $\cD_j(X)$. Slightly abusing notation, $\cD_j(x)$ denotes the only cube in $\cD_j$ containing $x$, and $\cD_j(\mu)$ stands for the family of cubes in $\cD_j$ with positive $\mu$-measure.

Given $X\subset \R^d$, let $|A|_{\delta}$ be the number of $\delta$-mesh cubes that intersect $A$. In particular, if $\delta=2^{-j}$, then $|A|_{\delta}=|\cD_j(A)|$.

The $r$-neighborhood of a set $A$ is denoted by $A^{(r)}$.

Let $V$ be a vector space. The Grassmanian of $k$-dimensional subspaces of $V$ is denoted by $\mathbb{G}(V,k)$. If $\dim V=d$ and $k=1$, we often identify $\mathbb{G}(V,1)$ with $S^{d-1}$; the fact that the identification is two-to-one does not cause any issues in practice. We denote the manifold of \emph{affine} $k$-planes in $V$ by $\mathbb{A}(V,k)$.

The orthogonal projection onto $V\in\mathbb{G}(\R^d,k)$ is denoted by $P_V$. In the special case $k=1$, we identify non-zero vectors $v$ with the line they span, so that $P_v = P_{\langle v\rangle}$.

Given $\mu\in\cP([0,1)^d)$ and a cube $Q$ with $\mu(Q)>0$, we denote
\begin{equation}
	\label{eq:def-mu-Q}
	\mu^Q = \text{Hom}_Q\mu_Q,
\end{equation}
where  $\text{Hom}_Q$ is the homothety renormalizing $Q$ to $[0,1)^d$. Thus, $\mu^Q$ is a magnified and renormalized copy of the restriction of $\mu$ to $Q$.

Logarithms are always to base $2$.

\subsection{Energy}

Recall that the Riesz $s$-energy of $\mu\in\cM(\R^d)$ is
\[
	\cE_s(\mu) = \iint \frac{d\mu(x)d\mu(y)}{|x-y|^s}.
\]

We will often appeal to the following close connection between finiteness of the $s$-energy and a $t$-Frostman condition for $t<s$. See e.g. \cite[Chapter 8]{Mattila95} for the simple proof.
\begin{lemma} \label{lem:energy-Frostman}
	Let $\mu\in\cM(\R^d)$.
	\begin{enumerate}
		\item If $\cE_s(\mu)<\infty$, then for all $\e>0$ there are $C>0$ and a compact set $X$ such that $\mu(X)\ge (1-\e)|\mu|$ and
		      \[
			      \mu_X(B_r) \le C\, r^{s-\e}\quad\text{for all } r>0.
		      \]
		\item Conversely, if
		      \[
			      \mu(B_r) \le C \,r^s\quad\text{for all }r>0,
		      \]
		      then $\cE_t(\mu)\lesssim_{C,t} 1$ for all $t<s$.
	\end{enumerate}
\end{lemma}

\subsection{Robust measures and robust entropy}

\begin{definition} \label{def:robust}
	A measure $\mu\in\cP(\R^d)$ is $(\delta,s,r)$-robust if $\mu(X) > r \Longrightarrow |X|_\delta > \delta^{-s}$.
\end{definition}

While a robust measure needs not satisfy a Frostman condition with exponent equal (or even close) to $s$, this is the case after passing to a subset of large measure.
\begin{lemma} \label{lem:robust-to-Frostman}
	Let $(a_m)_{m=1}^\infty$ be a sequence of positive numbers such that $\sum_{m=1}^\infty a_m<\infty$. Suppose $\mu\in\cP(\R^d)$ is $(2^{-m},s,a_m)$-robust for $m\in [m_0,m_1]$, where $m_0<m_1\in\N\cup\{+\infty\}$. If $\e:=\sum_{m=m_0}^{m_1} a_m<1$, then there is a set $X$ with $\mu(X)>1-\e$ such that
	\[
		\mu_X(B_r) \lesssim r^s\quad \text{for all } r\in [2^{-m_1},2^{-m_0}].
	\]

	In particular, if $\mu\in\cP(\R^d)$ is $(2^{-m},s,a_m)$-robust for all $m\ge m_0$, then any set $X$ with $\mu(X)>0$ satisfies $\mathcal{H}^s(X)>0$ and hence $\hdim(X)\ge s$.
\end{lemma}
\begin{proof}
	This is very standard but we provide the details for completeness. For each finite $m\in [m_0,m_1]$, let $E_m$ be the union of the cubes $Q\in\mathcal{D}_m$ with $\mu(Q)> 2^{-sm}$. By definition of robustness, $\mu(E_m)\le a_m$ and hence $X=\cap_{m=m_0}^{m_1} \R^d\setminus E_m$ satisfies $\mu(X)\ge 1-\e$, and $\mu_X(B_r)\lesssim r^s$ for $r\in [2^{-m_1}, 2^{-m_0}]$ by construction.

	The second claim is consequence of the first applied with $m_1=\infty$ and the mass distribution principle.
\end{proof}

Recall that the entropy of $\mu\in\cP(X)$ with respect to a finite partition $\cA$ of $X$ (or of a set of full $\mu$-measure in $X$) is defined by
\[
	H(\mu,\cA) = \sum_{A\in\cA} \mu(A)\log(1/\mu(A)),
\]
with the usual interpretation $0\log(1/0)=0$. We will sometimes write $H_m(\mu)$ in place of $H(\mu,\cD_m(\mu))$. Other than in the proof of Lemma \ref{lem:entropy-proj-elementary}, the only property of entropy that we will use directly is that $H(\mu,\cA)\le \log|\cA|$, which is a consequence of Jensen's inequality. In particular, $|\supp(\mu)|_{2^{-m}}\ge 2^{H_m(\mu)}$.

Given two Borel measures $\mu,\nu$ on the same metric space and a number $\Theta>0$, we write $\nu\le\Theta\mu$ if $\nu(X)\le \Theta\mu(X)$ for all Borel sets $X$. In particular, if $\mu(Y)>0$ then $\mu_Y\le \mu(Y)^{-1}\mu$.
\begin{definition}
	Let $\mu\in\cP(\R^d)$, fix $\Theta\ge 1$, and let $\mathcal{A}$ be a finite partition of some set of full $\mu$-measure. We define the $\Theta$-robust entropy $H^{\Theta}(\mu,\mathcal{A})$ as
	\[
		\inf \{H(\nu,\mathcal{A}): \nu\in\cP(\R^d), \nu\le\Theta \mu\}.
	\]
	In the case $\mathcal{A}=\mathcal{D}_m(\mu)$, we write $H_m^\Theta(\mu)$ in place of $H^\Theta(\mu,\mathcal{D}_m(\mu))$.
\end{definition}

The next lemma asserts that robust measures have large robust entropy; see \cite[Lemma 2.8]{Shmerkin23} for the short proof.
\begin{lemma} \label{lem:robust-to-entropy}
	Given $s,\eta,\e>0$, the following holds for all $m\ge m_0(s,\eta,\e)$: if $\mu\in\cP(\R^d)$ is $(2^{-m},s,2^{-\eta m})$-robust, then
	\[
		H^{2^{\eta m/2}}_m(\mu) \ge  (s-\e)m.
	\]
\end{lemma}

\subsection{Entropy of smooth projections}
\label{subsec:entropy-of-smooth-projections}

In order to bound the dimension of smooth images (such as pinned distance sets and radial projections) from below, we appeal to a multi-scale formula providing a lower bound for the (robust) entropy of smooth projections. It goes back, in various forms, to \cite{HochmanShmerkin12, Orponen17, Shmerkin19, KeletiShmerkin19}. The versions we present are taken from \cite[Appendix A]{Shmerkin23}.

Given a $C^1$ map $F:U\supset [0,1]^d\to\R^k$, let
\[
	V_F(x) = \left(\text{ker}(DF(x))\right)^\perp.
\]
Recall the definition \eqref{eq:def-mu-Q} of the renormalizations $\mu^Q$.

\begin{prop} [{\cite[Proposition A.1]{Shmerkin23}}] \label{prop:entropy-of-image-measure}
	Let $\mu\in\cP([0,1)^d)$, let $[A_j,B_j)_{j=1}^J$ be a sequence of disjoint sub-intervals of $[0,M]$ with $B_j\le 2A_j$. Let $F:U\supset [0,1)^d\to\R^k$ be a $C^2$ map such that $DF(x)$ has full rank $k$ for all $ \in \supp (\mu)$. Then
	\begin{equation} \label{eq:lower-bound-entropy}
		H_M(F\mu) \ge - O_{F,d}(J)   + \int \sum_{j=1}^{J}  H_{B_j-A_j}\left(P_{V_F(x)}\mu^{\cD_{A_j}(x)}\right) \,d\mu(x).
	\end{equation}
\end{prop}

\begin{prop}[{\cite[Proposition A.3]{Shmerkin23}}] \label{prop:entropy-of-image-measure-robust}
	Let $\mu\in\cP([0,1)^d)$, let $[A_j,B_j)_{j=0}^J$ be a sequence of disjoint sub-intervals of $[0,M]$ with $B_j\le 2A_j$. Let $F:U\supset [0,1)^d\to\R^k$ be a $C^2$ map such that $DF(x)$ has full rank $k$ for all $ \in \supp (\mu)$.  If $\mu'\in\cP([0,1)^d)$ satisfies $\mu'\le \Theta\mu$ then
	\begin{equation} \label{eq:lower-bound-entropy-robust}
		H_M(F\mu') \ge - O_{F,d}(J)   + \int \sum_{j=1}^{J}  H_{B_j-A_j}^{M \Theta}\left(P_{V_F(x)}\mu^{\cD_{A_j}(x)}\right) \,d\mu'(x).
	\end{equation}
\end{prop}

A key element of Propositions \ref{prop:entropy-of-image-measure}  and \ref{prop:entropy-of-image-measure-robust} is \emph{linearization}: at each scale $2^{-A_j}$, the nonlinear map $F$ is replaced by its linear approximation $P_{V_F(x)}$ around $\mu$-typical points $x$. Because the map $F$ is $C^2$, this approximation holds down to scales $2^{-2A_j}$; this explains the assumption $B_j \le 2A_j$ in the propositions. See Figure \ref{fig: linearization}.

\begin{figure}
	\centering
	\begin{tikzpicture}[scale=0.5]
		\draw[black] (0,0)--(0,10)--(10,10)--(10,0)--cycle;
		\draw[red, thick] (5,10) arc (30:-30:10);
		\draw[red, thick] (7,10) arc (30:-30:10);
		\draw[<-, red, thick] (5.5, 10.5)--(7, 12);
		\draw[red, thick] (7.8,12.8) node{Level sets of $F$};
		\draw[black, fill=black] (7.2, 5) circle[radius=0.1] node[right]{$x$};
		\draw[blue, dashed] (5,10)--(5,0);
		\draw[blue, dashed] (7,10)--(7,0);
		\draw[->, blue, thick] (4, 7)--(4, 3);
		\draw[blue, thick] (4, 5) node[left]{$V_F(x)^{\perp}$};
		\draw[<->] (5, -0.5)--(7,-0.5);
		\draw (6, -1) node{$\rho$};
		\draw[<-] (-0.8, 0)--(-0.8, 4);
		\draw[->] (-0.8, 6)--(-0.8, 10);
		\draw (-0.8, 5) node{$\rho^{1/2}$};
	\end{tikzpicture}
	\caption{Linearization}
	\label{fig: linearization}
\end{figure}
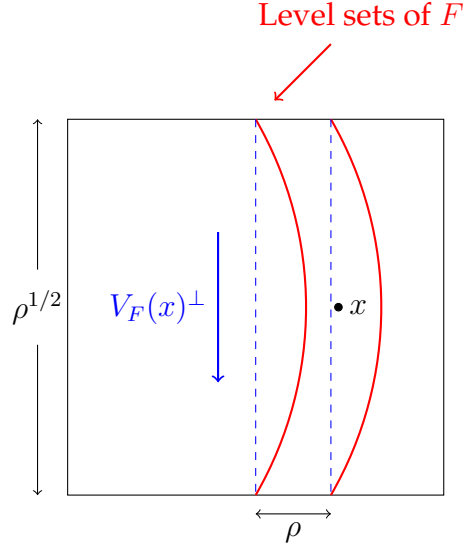

Proposition \ref{prop:entropy-of-image-measure-robust} is a robust version of Proposition \ref{prop:entropy-of-image-measure}. In general, to prove Hausdorff dimension one needs to rely on Proposition \ref{prop:entropy-of-image-measure-robust} while for lower box dimension estimates the simpler Proposition \ref{prop:entropy-of-image-measure} is enough.

\begin{remark} \label{rem:unif-F}
	In both Propositions \ref{prop:entropy-of-image-measure} and \ref{prop:entropy-of-image-measure-robust}, the implicit constant in $O_{F,d}(q)$ can be taken uniform in a $C^2$-neighborhood of $F$, see \cite{Shmerkin23}.
\end{remark}

\section{Quantitative projection theorems}
\label{sec:projection-thms}

\subsection{Adapted measures}

In this mostly expository section we present quantitative versions of older and newer (linear) projection theorems that will play a key r\^{o}le in the proofs of our main results.
\begin{definition}
	Let $X$ be a metric space. We say that $\mu\in\cP(X)$ is a $(\delta,s,C)$-measure if
	\[
		\mu(B_r) \le C\cdot r^s \quad\text{for all } r\in [\delta,1].
	\]
\end{definition}

Fix $1\le k<d$ and let $\rho\in\cP(\mathbb{G}(\R^d,k))$. The following definition says roughly that for $\rho$-nearly all $V\in \mathbb{G}(\R^d, k)$, the projections $P_V$ map measures of dimension $t$ at scale $\delta$ to measures of dimension $\ge s$ at scale $\delta$, robustly.

\begin{definition} \label{def:adapted-1}
	Given $t\in [0,d]$, $s\in [0,k]$ and $C\ge 1$, we say that a measure $\rho\in\cP(\mathbb{G}(\R^d,k))$ is \emph{$(t\to s;C)$-adapted at scale $\delta$} if whenever $\e>0$ and $\mu$ is a $(\delta,t,\delta^{-\e})$-measure on $[0,1]^d$, then
	\begin{equation} \label{eq:def-adapted-AR}
		\rho \{ V\in\mathbb{G}(\R^d,k): P_V\mu \text{ is not } (\delta,s-C\e,\delta^{\e})\text{-robust }\} \le \delta^{\e/C}.
	\end{equation}
\end{definition}
We will often omit the constant $C$ from the notation, but we always assume that it is independent of $V$, $\rho$, $\mu$, $\e$ and the scale $\delta$ under consideration (it will often be universal, but could depend on $t,s$). We note the trivial property that if $\rho$ is $(t\to s)$-adapted, then it is $(t'\to s')$-adapted for $t'\ge t$ and $s'\le s$.

We also note that if $\rho'$ and $\rho$ satisfy $\rho(B_\delta)\approx \rho'(B_\delta)$, then $\rho$ is adapted at scale $\delta$ if and only if $\rho'$ is adapted at scale $\delta$, since $|P_V X|_\delta \approx |P_{V'} X|_\delta$ when $d(V,V')\le\delta$. This allows us to replace $\rho$ by a smooth approximation at scale $\delta$, without loss of generality. Similarly, if $\mu(B_\delta)\approx \mu'(B_\delta)$, then $\mu(X^{(\delta)}) \approx \mu'(X^{(\delta)})$ for any set $X$. Since $|P_V X|_\delta \approx |P_V X^{(\delta)}|_\delta$, there is no loss of generality in replacing $\mu$ by $\mu'$. By taking a suitable smooth approximation, there is no loss of generality in replacing the assumption that $\mu$ is a $(\delta,t,\delta^{-\e})$ measure by
\begin{equation} \label{eq:delta-measure}
	\mu(B_r) \lesssim \delta^{-\e}\, r^t, \quad r\in (0,1].
\end{equation}
In the sequel we will use these simplifications without further reference.

In order to show robustness of (projected) measures, the next well known fact is very useful.
\begin{lemma} \label{lem:robust-l2-energy}
	Let $\mu\in\cP([0,1])$. For any $\eta>0$, when $\delta\in (0,1)$,
	\begin{enumerate}[(i)]
		\item If $\|\mu\|_2^2 \le \delta^{-u}$, then $\mu$ is $(\delta,1-u-2\eta,\delta^\eta)$-robust.
		\item If $\cE_s(\mu) \le \delta^{-u}$, then $\mu$ is $(\delta,s-u-2\eta,C\delta^\eta)$-robust for some constant $C>0$.
	\end{enumerate}
\end{lemma}
\begin{proof}
	Let $X$ be a set with $\mu(X)\ge \rho$. Then $\|\mu_X\|_2^2 \le \rho^{-2}\delta^{-u}$ in the first case and $\cE_s(\mu_X)\le \rho^{-2}\delta^{-u}$ in the second case. In the first case we use Cauchy-Schwarz to estimate the Lebesgue measure $|X|\ge \rho^2\delta^u$ which leads to $|X|_\delta \ge \rho^{2}\delta^{-(1-u)}$. In the second case, note that if $Q$ is a $\delta$-interval, then
	\[
		\mu_X(Q)^2 \le  \delta^{s}\int_{(x,y)\in Q^2} \frac{d\mu_X(x)d\mu_X(y)}{|x-y|^s}.
	\]
	Adding up over all the $\delta$-mesh cubes $Q_j$, we get
	\[ \sum_j \mu_X(Q_j)^2\le \delta^s \cE_s(\mu_X)  \le \rho^{-2}\delta^{s-u},
	\]
	from where the conclusion again follows by Cauchy-Schwarz:
	\[
		1= \sum_j \mu_X(Q_j)  \le  |X|_{\delta}^{1/2} \left(\sum_j \mu_X(Q_j)^2\right)^{1/2}.
	\]
\end{proof}

\subsection{Kaufman: classical and modern}
\label{subsec:Kaufman}

The classical projection theorem of Kaufman \cite{Kaufman68} says that if $0<s\le t\le 1$ and $A\subset\R^2$ is a Borel set with $\hdim(A)\ge t$, then
\[
	\hdim\{ \theta\in S^1: \hdim(P_\theta A)<s \}\le s.
\]
This is sharp if $t=s$. Very recently, T. Orponen and the first author \cite{OrponenShmerkin23} obtained an ``$\e$-improvement'' over Kaufman's Theorem when $t>s$: given $0<s<t\le 1$, there is $\e=\e(s,t)>0$ such that
\[
	\hdim\{ \theta\in S^1: \hdim(P_\theta A)<s \}\le s-\e.
\]
In this paper a quantitative version of this result plays a key r\^{o}le. We first state the corresponding quantitative version of Kaufman's Theorem, both as motivation and because we will need it later.
\begin{lemma} \label{lem:quantitative-Kaufman}
	Fix $0<t<s<1$, and let $\rho$ be a $(\delta,s,K)$-measure on $S^1$. If $\delta<\delta_0(s,t,K)$, then $\rho$ is $(t\to t)$-adapted at scale $\delta$.
\end{lemma}
\begin{proof}
	Let $\mu$ be a $(\delta,t,\delta^{-\e}$)-measure and, without loss of generality, assume $\mu$ is smooth at scale $\delta$. Then by decomposing the energy integral into the regions $\{ |x-y|<\delta\}$ and $\{ 2^{-j}\delta <|x-y|\le 2^{1-j}\delta\}$ (or see \cite[Lemma 3.1]{KeletiShmerkin19}), we get
	\[
		\cE_t(\mu) \lesssim_{t} \log(1/\delta)\delta^{-\e} \lesssim \delta^{-2\e}.
	\]
	It follows from the standard proof of Kaufman's Theorem, see \cite[Theorem 3.2]{Shmerkin23}, that
	\[
		\int_{S^1} \cE_t(P_\theta\mu)\,d\rho(\theta) \lesssim_{s,t} K  \cE_t(\mu) \lesssim_t K \delta^{-2\e}.
	\]
	More precisely, this holds after replacing $\rho$ by a smooth approximation at scale $\delta$ (so that $\rho(B_r)\lesssim K\,r^s$ for all $r\in (0,1]$).

		Let $E=\{ \theta\in S^1: \cE_t(P_\theta\mu)\ge \delta^{-4\e}\}$. Then $\rho(E)< \delta^{\e}$ assuming $\delta<\delta_0(s,t,K)$. The claim follows from Lemma \ref{lem:robust-l2-energy}.
\end{proof}

\begin{theorem} \label{thm:kaufman-improvement}
	For every $s\in (0,1)$ and $t\in (s,2)$ there is $\eta=\eta(s,t)>0$ such that the following holds if $\delta \leq \delta_0=\delta_0(s,t)$:

	Let $\mu$ be a $(\delta,t,\delta^{-\eta})$-measure on $[0,1)^2$, and let $\rho$ be a $(\delta,s,\delta^{-\eta})$-measure on $S^1$. Then there is a set $E$ with $\rho(E)\le \delta^\eta$ such that if $\theta\in S^1\setminus E$, then $P_\theta\mu$ is $(\delta,s+\eta,\delta^\eta)$-robust.
\end{theorem}
This is a rather simple consequence of the main result in \cite{OrponenShmerkin23}. It implies in particular that $\rho$ is $(t\to s+\eta)$-adapted at scale $\delta$. Before proceeding to the proof, we recall a discrete version of Frostman's Lemma. We say that $A\subset [-2,2]^d$ is a $(\delta,s,C$)-set if for every $r\in [\delta,1]$ it holds that
\[
	|A\cap B_r|_\delta \le  C \cdot |A|_\delta\cdot r^s.
\]
We denote $s$-dimensional Hausdorff content by $\mathcal{H}^{s}_{\infty}$.
\begin{lemma}\label{lem:discrete-Frostman} Let $\delta \in 2^{-\N}$, and let $B \subset [-2,2]^{d}$ be a set with $\mathcal{H}^{s}_{\infty}(B) =: \kappa > 0$. Then, there exists a $\delta$-separated $(\delta,s,C/\kappa)$-set $P \subset B$, where $C \geq 1$ is an absolute constant. Moreover, one can choose $P$ so that $|P| \leq \delta^{-s}$.
\end{lemma}
See \cite[Appendix B]{FasslerOrponen14} for the proof. We also recall the following form of the mass distribution principle.
\begin{lemma} \label{lem:Frostman-to-content}
	Let $\mu$ be a $(\delta,s,C)$-measure. Then $\mathcal{H}_\infty^s(\supp\mu^{(\delta)})\gtrsim C^{-1}$.
\end{lemma}
\begin{proof}
	If $\tilde{\mu}$ is a smooth approximation to $\mu$ at scale $\delta$, then $\wt{\mu}(B_r) \lesssim C r^s$ for all $r>0$ and therefore $\mathcal{H}_\infty^s(\supp\tilde{\mu})\gtrsim C^{-1}$.
\end{proof}

\begin{proof}[Proof of Theorem \ref{thm:kaufman-improvement}]
	We argue by contradiction: suppose $\rho(E_0)\ge \delta^\eta$, where
	\[
		E_0 = \{ \theta\in S^1: P_\theta\mu \text{ is not } (\delta,s+\eta,\delta^\eta)\text{-robust}\}.
	\]
	We will reach a contradiction if $\eta$ is small enough in terms of $s,t$ only.

	After pigeonholing a suitable $\pi/2$-arc and rotating, we may assume that $E_0$ is contained in the first quadrant. By Lemmas \ref{lem:discrete-Frostman} and \ref{lem:Frostman-to-content}, there is a $\delta$-separated $(\delta,s,O(\delta^{-2\eta}))$-set $E\subset E_0$ such that $|E|\le \delta^{-s}$.

	By assumption, for each $\theta\in E$ there is a set $X_\theta$ with $\mu(X_\theta)\ge \delta^\eta$ and $|P_\theta X_\theta|_{\delta} \le \delta^{-s-\eta}$. By Fubini, there is a set $Y_0$ with $\mu(Y_0) \gtrsim \delta^\eta$ such that if $x\in Y_0$, then
	\[
		|E_\theta|\gtrsim \delta^\eta |E|, \quad\text{ where } E_\theta=\{ \theta\in E: x\in X_\theta\}.
	\]
	Applying Lemmas \ref{lem:discrete-Frostman} and \ref{lem:Frostman-to-content} again, we obtain a $\delta$-separated $(\delta,t,O(\delta^{-2\eta}))$ set $Y\subset Y_0$.

	We recall some notation from \cite[Definition 2.10]{OrponenShmerkin23}. Given a dyadic square $p=[a,a+\delta]\times [b,b+\delta]$, we let $T^p$ be the union of all lines $y=a'x+b'$ over $(a',b')\in p$. We call the sets $T^p$ \emph{dyadic tubes}. We also refer to $[a,a+\delta]$ as the \emph{slope interval} of $T^p$. When intersected with a large ball, dyadic tubes contain and are contained in ordinary tubes of with $c\delta$ and $C\delta$ respectively, for some universal $0<c<C$. Thus, they are a convenient discretization of the family of tubes.

	Now for each dyadic square $p$ of side length $\delta$  intersecting $Y$, let $\mathcal{T}_p$ denote the set of dyadic tubes $T$ intersecting $p$ and such that the slope interval of $T$ intersects $E_\theta$, and let $\mathcal{T}=\cup \mathcal{T}_p$. By \cite[Corollary 2.12]{OrponenShmerkin23}, the family $\mathcal{T}_p$ is a $(\delta,s,O(\delta^{-\eta}))$-set of tubes for all $p$. Since $|E|\le\delta^{-s}$ and $E$ is contained in the first quadrant, the number of slopes $\delta i$ of tubes in $\mathcal{T}$ is $\lesssim \delta^{-s}$. Moreover, for each such slope $\delta i$, if we let $\theta\in E$ be a direction with slope within distance $C\delta$ of $\delta i$, then
	\[
		|\{T\in\mathcal{T}: \text{slope}(T)=\delta i \}| \lesssim |P_\theta X_\theta^{(2\delta)}|_\delta \lesssim  |P_\theta X_\theta|_\delta  \le \delta^{-s-\eta}.
	\]
	We conclude that $|\mathcal{T}|\lesssim \delta^{-2s-\eta}$. This contradicts \cite[Theorem 1.3]{OrponenShmerkin23} if $\eta$ is small enough in terms of $s,t$ and $\delta$ is small enough in terms of $\eta,s,t$.
\end{proof}

There is a version of Kaufman's theorem also for projections from $\R^d$ to $\R$, see \cite[\S 5.3]{Mattila15}, but we won't have occasion to use it. The improved version from Theorem \ref{thm:kaufman-improvement} is so far restricted to projections from $\R^2$ to $\R$.

\begin{remark} \label{rem:eta-continuous}
	By \cite[Remark 1.4]{OrponenShmerkin23}, we can take $\eta$ bounded away from zero in a neighborhood of any $(s,t)$ with $s\in (0,1)$, $t\in (s,2)$, and therefore we can take the function $\eta$ to be continuous in $(s,t)$ in this domain.
\end{remark}

\subsection{A robust Falconer's exceptional set estimate}

Our next estimate for adapted measures arises from another classical projection theorem due to Falconer \cite{Falconer85}, concerning projections to lines.
\begin{lemma} \label{lem:adapted-Falconer}
	Fix $0<s\le d-1$ and $t\ge d-1-s$. Let $\rho$ be a $(\delta,s,K)$-measure on $S^{d-1}$ for some $K\ge 1$. If $t>d-s$, then $\rho$ is $(t\to 1)$-adapted at scale $\delta$ and if $t<d-s$ then $\rho$ is $(t\to s+t+1-d)$-adapted at scale $\delta$, for all $\delta\le \delta_0(s,t,K)$.
\end{lemma}
\begin{proof}
	Again we may assume that $\rho(B_r)\lesssim K\,r^s$ for $r\in (0,1]$. Then Falconer's exceptional set estimate in the form presented in \cite[Theorem 3.3]{Shmerkin23} yields
	\[
		\int_{S^{d-1}} \|P_\theta\mu\|_2^2 \,d\rho(\theta) \lesssim_d K\, \cE_{d-s}(\mu)
	\]
	for any measure $\mu$ on $\R^d$ with $ \cE_{d-s}(\mu) <\infty$. If $\mu\in\cP([0,1]^d)$ satisfies $\mu(B_r)\lesssim \delta^{-\e} r^t$ for $r\in (0,1]$, then a calculation (or see \cite[Lemma 3.1]{KeletiShmerkin19}) shows that
	\[
		\cE_{d-s}(\mu) \lesssim_{s,t} \left\{
		\begin{array}{lll}
			\delta^{-\e}               & \text{ if } & t>d-s \\
			\log(1/\delta)\delta^{-\e} & \text{ if } & t=d-s \\
			\delta^{s+t-d-\e}          & \text{ if } & t<d-s
		\end{array}
		\right..
	\]
	The claim follows from Markov's inequality and Lemma \ref{lem:robust-l2-energy}.
\end{proof}

\subsection{A variant of Bourgain's projection theorem}
\label{subsec:Bourgain}

To conclude this section, we reinterpret Bourgain's projection theorem \cite{Bourgain10} in the language of adapted measures.

\begin{lemma} \label{lem:adapted-Bourgain}
	Given $0<t<d$ and $s>0$, there is $\eta=\eta(s, t, d)>0$ such that the following holds for $\delta\le  \delta_0= \delta_0(s,t,d)$.

	Let $\rho\in\cP(S^{d-1})$, $\mu\in \cP([0,1]^d)$ satisfy the decay bounds
	\begin{align*}
		\mu(B_r)      & \le \delta^{-\eta}\, r^t, \quad r\in[\delta,1],                            \\
		\rho(H^{(r)}) & \le \delta^{-\eta}\, r^s, \quad r\in [\delta,1], H\in\mathbb{G}(\R^d,d-1).
	\end{align*}
	Then there is a set $E\subset S^{d-1}$ with $\rho(E)\le \delta^\eta$ such that $P_\theta\mu$ is $(\delta,t/d+\eta,\delta^{-\eta})$-robust for all $\theta\in S^{d-1}\setminus E$.

	In particular, $\rho$ is $(t\to t/d+\eta)$-adapted.
\end{lemma}

The case in which $\mu=\mathbf{1}_X/|X|$ is a normalized Lebesgue measure on a set is a special case of Bourgain's projection theorem, as presented in \cite[Theorem 1]{He20}. The general case can be reduced to this one by splitting the measure into pieces on which $\mu$ is approximately constant. See \cite[Proof of Theorem 3.4, first case]{Shmerkin23} for details of the (standard) argument.

A version of Lemma \ref{lem:adapted-Bourgain} for entropy is elementary and holds under weaker assumptions; in particular, it involves control on the measures of hyperplane neighborhoods at only one scale. It will only be used in the proof of Theorem \ref{thm:box-dim-high-dim}. The proof uses some further elementary properties of (conditional) Shannon entropy, see e.g. \cite[Chapter 4]{Walters82}.
\begin{lemma} \label{lem:entropy-proj-elementary}
	Let $0<a,b\leq 1$ and let $\rho$ be a probability measure on $S^{d-1}$ such that $\rho(H^{(a)})\le b$ for all $H\in \mathbb{G}(\R^d,d-1)$. Let $\mu$ be a measure on $[0,1]^d$. Then
	\[
		\rho\left\{\theta\in S^{d-1}: H_m(P_\theta\mu)< \frac{H_m(\mu)}{d}-\log(1/a)-O_d(1)\right\} \le b .
	\]
\end{lemma}
\begin{proof}
	Let $\Theta\subset S^{d-1}$ satisfy $\rho(\Theta)>b$. Then $\rho_\Theta(H^{(a)})<1$ for all $H\in \mathbb{G}(\R^d,d-1)$ and thus we can find $\theta_1,\ldots,\theta_d\in\Theta$ such that $d(\theta_{i+1},\langle \theta_1,\ldots,\theta_i\rangle)\ge a$ for all $i$. Let $\cD_m^i=P_{\theta_i}^{-1}(\cD_m)$, where $\cD_m$ is the dyadic partition on $\R$, and let
	\[
		\cD'_m = \cD_m^1\vee\cdots\vee \cD_m^d
	\]
	be the common refinement of $\cD_m^1,\ldots, \cD_m^d$ (see  \cite[Chapter 4]{Walters82}). The elements of $\cD'_m$ are parallelepipeds containing a ball of radius $\gtrsim_d a 2^{-m}$; hence each cube in $\cD_m$ meets $\lesssim_d a^d$ atoms of $\cD'_m$. It follows that
	\[
		H(\mu,\cD_m) \le H(\mu,\cD'_m) +H(\mu,\cD_m|\cD'_m) \le  H(\mu,\cD'_m) +d\log(1/a)+O_d(1).
	\]
	On the other hand
	\begin{align*}
		H(\mu,\cD'_m) \le \sum_{i=1}^d H(\mu,\cD_m^i) = \sum_{i=1}^d H_m(P_{\theta_i}\mu).
	\end{align*}
	We conclude that there is $i$ such that
	\[
		H_m(P_{\theta_i}\mu) \ge \tfrac{1}{d}H(\mu,\cD'_m) \ge \tfrac{1}{d}H(\mu,\cD_m) - \log(1/a)-O_d(1).
	\]
\end{proof}

\begin{remark} \label{rem:adapted-continuous}
	In all the results of this section, the scale threshold $\delta_0$ depends on the parameters $s,t$, but it does so continuously in the range of allowed parameters. This will be important later when we want to know that a measure is adapted for many different pairs $(s\to t)$ simultaneously at a fixed scale $\delta$.
\end{remark}

\section{Abstract nonlinear projection theorems}
\label{sec:abstract-nonlinear}

\subsection{Setup}
\label{subsec:setup}

Fix integers $1\le k<d$. Let $(Y,\nu)$ be a compact metric measure space; it will act as our parameter space. Let $U\subset\R^d$ be a bounded domain, and let $F(x,y):U\times Y\to\R^k$ be a Borel map such that $F_y(x)=F(x,y)$ is a $C^2$ map without singular points for each $y\in Y$. Some examples include $Y=S^{d-1}$ and $F_y(x) = x\cdot y$ (linear projections), $Y\subset \R^d$ and $F_y(x)=|x-y|$ (distance projection) and $Y\subset \R^d$ and $F_y(x)=\pi_y(x)$ (radial projection). We always assume that $y\mapsto F_y$ is continuous in the $C^2$ norm
\[
	\|F\|_{C^2}= \|F\|_{\infty}+\sum_i
	\left\|\frac{\partial F}{\partial i}\right\|_{\infty}+\sum_{i,j}\left\|\frac{\partial^2 F}{\partial i \partial j}\right\|_{\infty}.
\]
In particular, the maps $(F_y)_{y\in Y}$ form a $C^2$-compact set.

\begin{definition} \label{def:regular-family}
	We call a tuple $\mathcal{F} = (Y,\nu,U,F)$ as above a \emph{regular family of projections}.
\end{definition}

Our goal is to bound the size of $F_y\mu$ with $\supp \mu \subset U$ in terms of the size of $\mu$, through the use of Proposition \ref{prop:entropy-of-image-measure-robust}.  Let
\begin{equation} \label{eq:def-Vx}
	V_x(y) = \left(\text{ker}D(F_y)(x)\right)^\perp \in \mathbb{G}(\R^d,k).
\end{equation}
In the case $k=1$, we identify $V_x(y)$ with $\dir \nabla F_y(x)=\nabla F_y(x)/\|\nabla F_y(x)\|$. We will linearize the map $F_y$ through the maps $P_{V_x(y)}$ applied at different scales, as in the statement of Proposition \ref{prop:entropy-of-image-measure-robust}.

We first consider the case in which $\mu$ is roughly Ahlfors regular. This case avoids a combinatorial problem that involves optimizing the choice of scales in a multiscale decomposition of $\mu$; we will later use it to obtain estimates for the Hausdorff dimension of radial projections and distance sets for sets of equal Hausdorff and packing dimension (including Ahlfors regular sets). Afterwards, we will deal with general measures $\mu$. The arguments in this section borrow heavily from \cite{KeletiShmerkin19, Shmerkin23}. In fact, our abstract theorems below can be used to recover some of the main results from \cite{KeletiShmerkin19, Shmerkin23}.

\subsection{The roughly Ahlfors-regular case}

The following proposition says, roughly, that if $\mu$ is not too far from being $t$-Ahlfors regular down to some scale $\delta$ and the measures $V_x\nu_Y$ are $(t\to s$)-adapted, then $F_y\mu$ is $s$-robust at scales $>\delta$ for many $y\in Y$. The actual statement is rather technical; control of the various quantitative parameters is important in our later applications. The reader may want to glance at the less technical Theorem \ref{thm:abstract-proj-Hausdorff} below to get a sense of what the proposition can accomplish.

\begin{prop}\label{prop: nonlinear-abstract-AD}
	Fix $1\le M_1 < \tilde{M} \in\N\cup\{+\infty\}$, $c\in (0,1)$, $\e>0$, $C\ge 2$, $1\le k<d$ and $t\in(0,d]$, $s\in (0,k]$.
	Suppose that
	$\mu \in \mathcal{P}([0,1]^d)$ satisfies
	\begin{equation}\label{eq: nearAD}
		2^{-m(t+\e) } \leq \mu(Q) \leq  2^{-m(t-\e)}
	\end{equation}
	for all finite $m \in [M_1,\tilde{M}]$ and all $Q\in \mathcal{D}_m(\mu)$.

	Let $\mathcal{F}=(Y,\nu,U,F)$ be a regular family of projections. Suppose that for all $x$ in a set $X'$ with $\mu(X')> 1-c$, there is a set $Y_x$ with $\nu  (Y_x) >1-c$  such that $V_x\nu_{Y_x}$ is $(t\to s;C)$-adapted at all (non-zero) scales in $[2^{-\tilde{M}/2},2^{-M_1}]$. Further, assume that $\{ (x,y):x\in X', y\in Y_x\}$ is compact.

	Then there are $c' =1-(1-c)^3/2$,   $M_0 =  \max \{  O_{F,s, \e,c} (1), M_1\e^{-2},  \frac{-C\log c}{\e^3} \} \in\N $ ,  a set $Y'\subset Y$ such that $\nu(Y')\geq 1-c'$, and  for any $y\in Y'$
	a set $X_y\subset X$ with $\mu(X_y) \geq 1-c'$ and such that $F_y(\mu_{X_y})$ is $(2^{-M}, s-2C\e, 2^{-M\e^2/4})$-robust for all (finite) $M\in [M_0,\tilde{M}]$. Moreover, $\{ (x,y):y\in Y', x\in X_y\}$ is compact.

\end{prop}

\begin{remark}
	It might seem strange that \eqref{eq: nearAD} is stronger when $\e$ is smaller but $M_0$ needs to be larger. But the actual constraint on $M_0$ comes from  the $\e$ that appears in $s-2C\e$.
\end{remark}

\begin{proof}

	Fix $M_0\le M\le \tilde{M}$ for the time being (where $M_0$ is sufficiently large, to be determined along the proof). Let $m_0 = \e M$,  $m_j = 2^j m_0$.  By perturbing $\e$, we may assume that there is an integer $J$ such that $m_J=M/2$. We only consider values of $j$ such that $0\le j\le J$.

	By translating the dyadic grid by a small random vector, we may assume that $\mu$ gives zero mass to dyadic hyperplanes. Thus, by replacing $X'$ by a slightly smaller compact set, we may assume every $x\in X'$ is at distance at least $\delta_j$ from all $2^{-{m_j}}$-dyadic hyperplanes for some sequence $\delta_j>0$.

	Given $x\in Q\in \mathcal{D}_{m_j}(\mu)$, we get from \eqref{eq: nearAD} that
	\begin{equation}
		\mu^Q(B_r) \lesssim_{M_1} 2^{2 m_j\e}\cdot r^{t} \le 2^{3 m_j \e}\cdot r^t \text{ for all } r\in (2^{-m_j}, 1],
	\end{equation}
	provided $M_0\ge M_1\e^{-2} $   (which also makes $m_j\ge m_0 \gg_{M_1} 1$). In other words, $\mu^Q$ is a $(2^{-m_j}, t, 2^{-3m_j \e})$-measure. Note that $M_1 \le m_j \le \tilde{M}/2$ if $M_0\ge M_1\e^{-2} $.

	Let us write $\mu_{x,j}$ for $\mu^Q$, where $Q\in\cD_{m_j}(x)$.  By the hypothesis that $\rho_x = V_x\nu_{Y_x}$ is $(t\to s$)-adapted at scales in $[2^{-\tilde{M}/2}, 2^{-M_1}]$, for each $j$ there is a set $\bad_M(x,j)$ with
	\[
		\nu_{Y_x}(\bad_M(x,j)) \leq 2^{-C^{-1} \e m_j},
	\]
	such that if $y\in Y_x\setminus \bad_M(x,j)$, then
	\begin{equation} \label{eq:def-bad-AD}
		P_{V_x(y)}\mu_{x,j} \text{ is } (2^{-m_j},s-C\e,2^{-\e m_j})\text{-robust}.
	\end{equation}
	Since $y\mapsto F_y$ is $C^2$-continuous and $\mu_{x,j}$ depends only on $\cD_{m_j}(x)$, we can perturb $y$ within $Y_x$ and $x$ within $\cD_{m_j}(x)$ while preserving the robustness of $ P_{V_x(y)}\mu_{x,j}$. More precisely, by our assumption on separation between $X'$ and dyadic hyperplanes, we may assume that \eqref{eq:def-bad-AD} continues to hold in some small closed dyadic cubes containing $x$ and $y$, of side length bounded below (depending on $j$).  Hence  there is no loss of generality in assuming that $\{ (x,y): y\in Y_x\setminus \bad_M(x,j) \}$ is compact. Let $\bad_M(x)= \cup_{j=1}^J \bad_M(x,j)$, and $\bad(x)=\cup_{M=M_0}^{\tilde{M}} \bad_M(x)$. Then
	\begin{align*}
		\nu_{Y_x}(\bad_M(x)) & \leq \sum_{j=1}^J 2^{-C^{-1} \e m_j}  \lesssim  2^{-C^{-1}\e m_0} =  2^{- C^{-1} \e^2 M}, \\
		\nu_{Y_x}(\bad(x))   & \lesssim \sum_{M=M_0}^\infty  2^{-C^{-1} \e^2 M} \lesssim_\e 2^{- C^{-1} \e^2 M_0}.
	\end{align*}
	Hence, taking $M_0 \ge \frac{-C\log c}{\e^3}$, we may assume that
	\[
		\nu_{Y_x}(\bad(x)) \le c.
	\]
	By the assumption and Fubini,
	\[
		(\mu_{X'}\times\nu)\{ (x,y): y\in Y_x\setminus\bad(x)\} \ge (1-c)^2 > \max(1- 2c,0).
	\]
	Note that $\{ (x,y): y\in Y_x\setminus\bad(x)\}$ is compact. Applying Fubini again, we see that there is a compact set $Y'$ with $\nu(Y')> (1-c)^2/2$ such that if $y\in Y'$ then, letting
	\[
		X_y= \{ x \in X':y\in Y_x\setminus\bad(x)\},
	\]
	we have
	\[
		\mu(X_y) = \mu(X')\mu_{X'}(X_y) \ge  (1-c)(1-c)^2/2 = \frac{(1-c)^3}{2},
	\]
	and $\{ (x,y): y\in Y', x\in X_y\}$ is compact.

	Fix $y\in Y'$ for the rest of the proof. In order to finish the proof, it is enough to show that if $M\in [M_0,\tilde{M}]$ and $Z\subset X_y$ satisfies $\mu_{X_y}(Z)\ge 2^{-\e^2 M/4}$, then
	\[
		|F_y Z|_{2^{-M}} \ge 2^{(s-2C\e)M}.
	\]
	We apply Proposition \ref{prop:entropy-of-image-measure-robust} to $\mu_Z$ and $\mu_{X'}$. Let $\e' = \e^2/4$. Since $\mu_{Z} \leq O_c(1) 2^{\e' M}\, \mu_{X'}$ and $\{[m_j, m_{j+1}]\}_{j=-1}^J$ is a partition of $[0,M]$ with $m_{J+1}:=M$ and $m_{-1}:=0$,
	\begin{equation}\label{eq: Prop4.2eq1}
		H_M(F_y \mu_{Z}) \geq -O_{F}(J)+  \sum_{j=0}^{J} \int_{Z}  H_{m_j}^{O_c(1) M2^{\e' M+1}} (P_{V_x}(y) \mu_{x,j}) \,d\mu_Z(x).
	\end{equation}
	By Remark \ref{rem:unif-F}, the implicit constant in $O_{F}(J)$ indeed depends only on the family $(F_y)_{y\in Y}$ and not on the particular $y$ chosen. As above,  $m_0= \e M$, $m_j = 2^j m_0$.  Assuming $M_0$ is large enough in terms of $c$ and $\e$,
	\[
		2^{\e m_j/2} = 2^{2^{j-1}\eps^2 M} \geq O_c(1) M 2^{\e' M}.
	\]
	Recall that $y\in Y_x\setminus\bad(x)$ for all $x\in Z\subset X_y$. By the definition of $\bad(x)$, in particular \eqref{eq:def-bad-AD}, we see that $P_{V_x(y)}\mu_{x,j}$ is $(2^{-m_j}, s-C\e, 2^{-m_j \e})$-robust.  Applying Lemma~\ref{lem:robust-to-entropy} and taking $M_0$ large enough in terms of $s,\e$, we get
	\begin{equation}\label{eq: prop4.2eq2}
		H^{M2^{\e' M+1}}_{m_j} (    P_{V_x(y)}\mu_{x,j}  )\geq (s- 1.5 C\e) m_j.
	\end{equation}
	Taking $M_0$ large enough in terms of $F$ and plugging \eqref{eq: prop4.2eq2} in \eqref{eq: Prop4.2eq1}, we conclude that   \[
		H_M(F_y\mu_Z)  \geq   - O_F(J) + \sum_{j=0} ^{J} (s-1.5C\eps)m_j \geq  (s-2C\e)M,\] and therefore $|F_y Z|_{2^{-M}} \ge 2^{(s-2 C\e)M}$, as we wanted to see.
\end{proof}

\subsection{General case: definition of the combinatorial parameter}
\label{subsec:def-comb}

Now we turn our attention to projection theorems for general measures $\mu$. It turns out that we need to consider adapted measures for many pairs $(s,t)$ simultaneously.
\begin{definition}\label{def: Dadapted}
	Let $D:[0,d]\to [0,k]$ be a function. We say that a measure $\rho\in\cP(\mathbb{G}(\R^d,k))$ is \emph{$(D;C)$-adapted at scale $\delta$} if it is $(s\to D(s);C)$-adapted at scale $\delta$ for all $s\in [0,d]$ (see Definition~\ref{def:adapted-1}).
\end{definition}
As before, we will often omit $C$ from the notation and assume it is universal.

The lower bound on the dimension of $F_y(\mu)$ will be expressed in terms of a certain combinatorial optimization problem involving the function $D(t)$.
\begin{definition} \label{eq:def-L-du}
	Given $u\in [0,d]$, we let $\mathcal{L}_{d,u}$ denote the class of all non-decreasing Lipschitz functions $f:[0,1]\to [0,d]$ with Lipschitz constant (at most) $d$ that are piecewise linear, and such that $f(0)=0$ and $f(x)\ge ux$ for all $x\in [0,1]$.

	For $f\in\mathcal{L}_{d,u}$, we will denote the right-derivative by $f'$ (which is well defined since $f$ is piecewise linear).
\end{definition}

\begin{definition} \label{eq:def-superlinear}
	We say that a function $f:[0,1]\to\R$ is \emph{$\sigma$-superlinear} on $[a,b]\subset [0,1]$, or that $(f,a,b)$ is $\sigma$-superlinear, if
	\[
		f(x) \ge f(a)+ \sigma\cdot (x-a), \quad x\in [a,b].
	\]
\end{definition}

We say that an interval $[a,b]\subset [0,1]$ is \emph{$\tau$-allowable} (or simply allowable if $\tau$ is clear from context) if
\[
	\tau \le b-a\le a.
\]
Likewise, given $[a,b]\subset [0,1]$, we say that a collection of pairwise non-overlapping sub-intervals $[a_j,b_j]$  of $[a,b]$ is $\tau$-allowable if each $[a_j,b_j]$ is allowable. Given $[a,b]\subset [0,1]$, let
\[
	\Sigma_{\tau}(D;f;[a,b]) = \sup\left\{ \sum_{j=1}^J (b_j-a_j)D(\sigma_j)  \right\},
\]
where the supremum is taken over all $\tau$-allowable collections $\{[a_j,b_j]\}$ of sub-intervals of $[a,b]$ and numbers $\sigma_j\in [0,d]$ such that $f$ is $\sigma_j$-superlinear on $[a_j,b_j]$. In the case $[a,b]=[0,1]$, we omit it from the notation.

We finally define
\[
	\Sigma_{\tau}(D;t) = \inf\{ \Sigma_{\tau}(D;f;[0,1]) : f\in\mathcal{L}_{d,t} \}.
\]
When the function $D$ is clear from context we omit it from the notation.

Since $\mathcal{L}_{d,t}\subset \mathcal{L}_{d,s}$ for $t>s$, $\Sigma_{\tau}(D, t)$ is an increasing function of $t$.  It is also continuous whenever $D$ is:
\begin{lemma} \label{lem:Sigma-Lipschitz}
	Fix $\tau, D$ as above. If $D$ is (Lipschitz) continuous, then $s\mapsto \Sigma_{\tau}(D;s)$ is (locally Lipschitz) continuous on $(0,d)$.
\end{lemma}
\begin{proof}
	Fix $s<d $ and $f\in\mathcal{L}_{d,s}$. Given $t>s$ sufficiently close to $s$, we will construct $\tilde{f}\in\mathcal{L}_{d,t}$ such that
	\[
		\tilde{f}'(x) \le f'(x)+ \eta_{t,s},\quad\text{where }\eta_{t,s}= \frac{d+t}{d-t}(t-s)
	\]
	for all $x\in [0,1)$. Then $f$ is $(\sigma-\eta_{t,s})$-superlinear on $[a,b]$ whenever $\tilde{f}$ is $\sigma$-superlinear on $[a,b]$, and this gives the claim.

				We define $\tilde{f}$ by modifying the slopes on the intervals on which $f$ is linear. When the slope is $\ge (t+d)/2$, we keep the slope equal. When the slope is $< (t+d)/2$, we increase it by $\eta_{t,s}$. Then $f'(x)\le \tilde{f}'(x) \le f'(x)+\eta_{t,s}$ for all $x$. If $t$ is close to $s$, then $\tilde{f}$ is still $d$-Lipschitz. If $a\in (0,1]$ is such that $f(a)< t a$, then
	\[
		|\{ x\in [0,a] : f'(x)\ge \tfrac{t+d}{2}\}| \le \tfrac{2t}{t+d} a.
	\]
	Then
	\[
		\tilde{f}(a) \ge f(a) + \eta_{t,s}|\{x\in [0,a]: f'(x)<\tfrac{t+d}{2}\}| \ge f(a)+a(t-s) \ge t a,
	\]
	by the construction of $\tilde{f}$, the choice of $\eta_{t,s}$, and since $f\in\mathcal{L}_{d,s}$. Hence $\tilde{f}\in\mathcal{L}_{d,t}$.
\end{proof}

\subsection{Statement of the projection theorem}

We can now state our general projection theorem for the family $(F_y\mu)_{y\in Y}$. In very rough terms (and up to suitable error terms and restrictions to appropriate sets), our projection theorem says that if $\mu$ is a $(\delta,s)$-measure and $V_x\nu$ is $D$-adapted at scale $\delta$ for $\mu$-many $x$, then there are many $y\in Y$ such that $F_y\mu$ is $(\delta,\Sigma_{\tau}(D;s))$-robust. See Theorem \ref{thm:abstract-proj-Hausdorff} below for a (simpler) version for Hausdorff dimension.

\begin{theorem} \label{thm:nonlinear-abstract}
	Let $D:[0,d]\to [0,k]$ be a continuous function. Fix $c\in(0,1)$, $C\ge 1$, $0\le \delta<\delta_1\le 1$. Let $\mathcal{F}=(Y,\nu,U,F)$ be a regular family of projections.

	Let $\mu\in\cP(U)$ satisfy
	\begin{equation} \label{eq:nonlinear-Frostman-assumption}
		\mu(B_r) \le r^s, \quad \delta \le r \le \delta_1.
	\end{equation}
	Suppose that there is a set $X'$ with $\mu(X')> 1-c$ such that the following holds for each $x\in X'$: there is a set $Y_x$ with $\nu(Y_x)>1- c$, such that the measure
	\[
		\rho_x = V_x \nu_{Y_x} \in\cP(\mathbb{G}(\R^d,k))
	\]
	is $(D;C)$-adapted at all scales $\tilde{\delta}\in (\delta^{1/2},\delta_1)$. We also assume that $\{ (x,y):x\in X',y\in Y_x\}$ is compact.

	Fix $\tau,\e>0$. If  $\delta\le\delta_0=\min( \delta_0( F, s,c,C,\e,\tau), \delta_1 2^{-\e^{-2}})$, then there are a set $Y_0\subset Y$ with $\nu(Y_0)>\max(1-2c^{1/2},0)$, and for each $y\in Y_0$ there is a set $X_y\subset X'$ with $\mu(X_y)>\max(1-2c^{1/2},0)$, such that
	\[
		F_y(\mu_{X_y}) \text{ is  } (\tilde{\delta},\Sigma_{\tau}(D;s)-O(\e^{1/2}),\tilde{\delta}^{\eta})\text{-robust}
	\]
	for all $\tilde{\delta}\in (\delta,\delta_0)$, where $\eta=\eta(\e,\tau)>0$. Moreover, $\{ (x,y): y\in Y_0, x\in X_y\}$ is compact.
\end{theorem}

\subsection{Review of uniform measures and uniformization}

\label{subsec:uniform-measures}

A $2^{-m}$-measure is a measure $\mu$ in $\cP([0,1)^d)$ such that $\mu_Q$ is a multiple of Lebesgue measure on $Q$ for $Q\in\cD_m$. Hence, $2^{-m}$-measures are defined down to resolution $2^{-m}$.  If $\mu$ is any measure in $\cP([0,1)^d)$, we denote by $\mu^{(m)}$ the associated $2^{-m}$-measure, that is, $\mu^{(m)}|_Q$ is a multiple of Lebesgue measure and $\mu^{(m)}(Q)=\mu(Q)$ for all $Q\in\cD_m$.

As in the papers \cite{KeletiShmerkin19,Shmerkin23,OrponenShmerkin23}, a crucial role will be played by ``uniform measures'' (often also called ``regular measures''), that is, measures with a uniform tree structure when represented in base $2^T$:

\begin{definition} \label{def:uniform}
	Given a sequence $\beta=(\beta_1,\ldots,\beta_{\ell})\in [0,d)^\ell$ and $T\in\N$, we say that a $2^{-T\ell}$-measure $\mu$ is \emph{$(\beta;T)$-uniform} if for any $Q\in \cD_{jT}(\mu)$, $1\le j\le\ell$, we have
	\[
		\mu(Q) \le 2^{-T \beta_j} \mu(\wh{Q}) \le 2\mu(Q),
	\]
	where $\wh{Q}$ is the only cube in $\cD_{(j-1)T}$ containing $Q$. When $T$ is understood from context we will simply write that $\mu$ is $\beta$-uniform.
\end{definition}

We note the following direct estimates for a $((\beta_1,\ldots,\beta_\ell);T)$-uniform measure $\mu$:
\begin{equation} \label{eq:reg-measure-cubes}
	2^{-j} 2^{-T(\beta_1+\ldots+\beta_j)} \le \mu(Q) \le 2^{-T(\beta_1+\ldots+\beta_j)}\quad (Q\in\cD_{Tj}(\mu)),
\end{equation}
which implies
\begin{equation} \label{eq:reg-box-counting}
	2^{T(\beta_1+\ldots+\beta_j)} \le |\cD_{Tj}(\mu)| \le 2^j 2^{T(\beta_1+\ldots+\beta_j)}.
\end{equation}

Due to our use of dyadic cubes,  sometimes we will need to deal with supports in the dyadic metric, i.e. given $\mu\in\cP([0,1)^d)$ we let
\[
	\supp_{\mathsf{d}}(\mu) = \{ x: \mu(\cD_j(x))>0 \text{ for all } j\in\N\}.
\]
In the sequel we do not distinguish between Euclidean and dyadic supports.

The following ``uniformization lemma'' will allow us to work with uniform measures rather than arbitrary ones.
\begin{lemma} \label{lem:subset-uniform}
	Fix $T,\ell\ge 1$ and $\e>0$. Write $m=T\ell$, and let $\mu$ be a $2^{-m}$-measure on $\R^d$. Then there is a set $X\subset\supp(\mu)$ with $\mu(X)\ge (2dT+2)^{-\ell}$ such that $\mu_X$ is $(\beta;T)$-uniform for some $\beta\in [0,d]^\ell)$.
\end{lemma}
See \cite[Lemma 3.4]{KeletiShmerkin19} for the proof. Iterating the above lemma, we can decompose an arbitrary measure into a combination of uniform measures, plus a negligible error. See \cite[Corollary 3.5]{KeletiShmerkin19} for the proof.
\begin{lemma} \label{lem:decomposition-uniform}
	Fix $T,\ell\ge 1$ and $\e>0$. Write $m=T\ell$, and let $\mu$ be a $2^{-m}$-measure on $\R^d$.  There exists a family of pairwise disjoint sets $X_1,\ldots, X_N$, each a union of cubes in $\cD_m$, with $X_i\subset\supp(\mu)$, and such that:
	\begin{enumerate}[(\rm i)]
		\item $\mu\left(\bigcup_{i=1}^N X_i\right) \ge 1-2^{-\e m}$.
		\item $\mu(X_i) \ge 2^{-\e' m}$, where $\e'=\e+\log(2d T+2)/T$,
		\item Each $\mu_{X_i}$ is $(\beta^i;T)$-uniform for some $\beta^i\in  [0,d]^\ell$.
	\end{enumerate}

\end{lemma}

\subsection{Proof of Theorem \ref{thm:nonlinear-abstract}}

\subsubsection{Setup}
\label{subsubsec:setup-abstract-proj-thm}
As in the proof of Proposition \ref{prop: nonlinear-abstract-AD}, we may assume that $X'$ is at positive distance from all dyadic hyperplanes.

Fix $\e>0$. Without loss of generality, $\delta=2^{-M}$ (where we allow for $M=+\infty$). Pick an integer $T=T(\e)$ such that
\begin{equation} \label{eq:choice-T}
	\frac{\log(2dT+2)}{T} \le \e.
\end{equation}

Fix a large integer $m=T\ell \le M$ for the time being. Let $(X_i)_i$ be the decomposition provided by Lemma \ref{lem:decomposition-uniform} applied to $\mu^{(m)}$; let $\beta^i$ be  the regularity sequence associated to $\mu_{X_i}$. Let $f_i$ be the function such that
\[
	f_i(\tfrac{j}{\ell}) = \tfrac{1}{\ell}(\beta^i_1+\cdots+\beta^i_j),
\]
and interpolates linearly between $j/\ell$ and $(j+1)/\ell$. Then $f_i$ is a piecewise-linear, $d$-Lipschitz function.

Take $m$ large enough that
\begin{equation} \label{eq:m-large}
	2^{3\e m}\ge \delta_1^{-s}, \quad 2^{\e m}\ge \sqrt{d}.
\end{equation}
By our choice of $T$ in \eqref{eq:choice-T} and Lemma \ref{lem:decomposition-uniform}(ii), for each $Q\in\cD_j$, $j\in [0,m]$, we have
\[
	\mu_{X_i}(Q) \le 2^{2\e m} \mu(Q)   \le 2^{3\e m}\cdot 2^{-jTs},
\]
using the left-hand inequality in \eqref{eq:m-large} and the trivial bound $\mu(Q)\le 1$ whenever $2^{-j}> \delta_1$, and the right-hand inequality in \eqref{eq:m-large} and the Frostman assumption \eqref{eq:nonlinear-Frostman-assumption} whenever $2^{-j}\le \delta_1$.

In light of \eqref{eq:reg-measure-cubes},
\[
	2^{-j}2^{-T(\beta_1+\ldots+\beta_j)} \le \mu(Q) \le 2^{3\e m}2^{jTs}.
\]
We deduce that
\[
	f_i(x) \ge s x - 3\e -\frac{1}{T} \overset{\eqref{eq:choice-T}}{\ge} s-4\e \ge (s-\e^{1/2})x, \quad x\in [4\e^{1/2},1].
\]
(This holds initially for $x=j/\ell$; it extends to all $x$ by piecewise linearity.)

Let $\tilde{f}_i$ be the function that agrees with $f_i$ on $[4\e^{1/2},1]$ and interpolates linearly between $f_i(0)=0$ and $f_i(4\e^{1/2})$. Then, writing $s'=s-\e^{1/2}$, we have $\tilde{f}_i\in\mathcal{L}_{d,s'}$.

Fix $i$ and let $([a_j,b_j])_{j=1}^J$ be a $\tau$-allowable collection such that $(f_i,a_j,b_j)$ is $\sigma_j$-superlinear for some $(\sigma_j)_{j=1}^J$, and
\[
	\sum_{j=1}^J (b_j-a_j) D(\sigma_j) \ge \Sigma_{\tau}(s')-\e.
\]
By perturbing slightly the intervals $[a_j,b_j]$, as well as the values of $\tau, \sigma_j$, we may assume that $a_j,b_j\in \tfrac{1}{\ell}\Z$. By removing those intervals with $a_j \le 4\e^{1/2}$, we may also assume that $a_j > 4\e^{1/2}$ at the price of weakening the above to
\begin{equation} \label{eq:optimal-sum-adapted}
	\sum_{j=1}^J (b_j-a_j) D(\sigma_j) \ge \Sigma_{\tau}(s')- O(\e^{1/2}).
\end{equation}

Write $A_j=m a_j, B_j=m b_j$, so that $A_j, B_j$ are multiples of $T$ in $[0,m]$. We also set
\begin{equation} \label{eq:def-m-j}
	m_j = B_j-A_j = m(b_j-a_j).
\end{equation}
Now it follows from the definition of regularity that, for any $Q\in\cD_{A_j}$ such that $\mu_{X_i}(Q)>0$, the measure
\[
	\mu_i^Q:= \left(\mu^Q_{X_i}\right)^{m_j}
\]
is $(\beta^i_{[A_j/T,B_j/T]};T)$-uniform (here we use the shortcut $\beta_{[a,b]}(x) = \beta(x-a)$, $x\in [a,b]$). The fact that $\tilde{f}$ is $\sigma_j$-superlinear on $[a_j,b_j]$ together with \eqref{eq:reg-measure-cubes} yields (for $m_j \ge \tau m$ and $T \ge 2 (\e \tau)^{-1}$)
\begin{equation} \label{eq:cond-Frostman}
	\mu_i^Q(B_r) \le 2^{\e m_j} r^{\sigma_j},  \quad r\in [2^{-m_j},1].
\end{equation}
That is, $\mu_i^Q$ is a $(2^{-m_j},\sigma_j,2^{\e m_j})$-measure.

\subsubsection{Application of adapted hypothesis and sets of bad points}

At this point we use the hypothesis that $\rho_x = V_x \nu_{Y_x}$ is $(D;C)$-adapted for each $x\in X'$. Suppose $x\in X'\cap X_i \cap Q$ for $Q\in\cD_{A_j}(\mu_{X_i})$. By the definition of adaptedness, there is a set $\bad(x,m,j)\subset Y_x$ such that
\[
	\nu_{Y_x}(\bad(x,m,j)) \le  2^{-m_j\e/C} \le 2^{-(\tau\e/C) m}
\]
and if $y\in Y_x\setminus \bad(x,m,j)$, then
\begin{equation} \label{eq:def-bad-set}
	P_{V_x(y)}\mu_i^Q \text{ is  } (2^{-m_j},D(\sigma_j)-C\e,2^{-\e m_j})\text{-robust}.
\end{equation}
For completeness, if $x\notin \cup_i X_i$, then we set $\bad(x,m,i)=\varnothing$. As in the proof of Proposition \ref{prop: nonlinear-abstract-AD}, we may assume that $Y_x\setminus \bad(x,m,i)$ is compact for all $m,i$.

Let $M_0$ be sufficiently large, to be fixed in the course of the proof. We set
\begin{align*}
	\bad(x,m)= \bigcup_{j=1}^{J} \bad(x,m,j), \\
	\bad(x) =\bigcup_{M_0\le m\le M, T\mid m} \bad(x,m).
\end{align*}
Then $\nu_{Y_x}(\bad(x,m))\le  J 2^{-(\tau\e/C) m}$ and hence, since $J\le \tau^{-1}$, we can ensure that
\[
	\nu_{Y_x}(\bad(x)) \le c
\]
by taking $M_0\gg_{c,C,\tau,\e} 1$. Recalling that  $\nu(Y_x)>1-c$, Fubini yields
\[
	(\mu_{X'}\times\nu)\{ (x,y): y\in Y_x\setminus \bad(x)\} > (1-c)^2 > \min(1-2c,0).
\]
Applying Fubini once more, we see that there is a compact set $Y_0$ with $\nu(Y_0)> \min(1- 2 c^{1/2},0)$ such that if $y\in Y_0$ then, letting
\[
	X_y= \{ x:y\in Y_x\setminus\bad(x)\},
\]
we have that $\{ (x,y):y\in Y_0, x\in X_y\}$ is compact, and
\[
	\mu(X_y) \ge \mu(X')\mu_{X'}(X_y) \ge  (1-c)(1-c^{1/2}) \ge 1-2c^{1/2}.
\]

\subsubsection{Application of multiscale entropy formula and conclusion of proof}

Fix $y\in Y_0$ for the rest of the proof and write $t= \Sigma_{\tau}(s-O(\e^{1/2}))-O(\e^{1/2})$. In light of the continuity of $\Sigma$ (Lemma \ref{lem:Sigma-Lipschitz}), it is enough to show that if $m=T\ell\in [M_0,M]$, then  $F_y(\mu_{X_y})$  is  $(2^{-m},t, 2^{-\eta m/2})$-robust, provided $\eta$ is small enough in terms of $t,\e$; indeed, since $T=T(\e)$, the claim for a general $\tilde{\delta}\in (2^{-M},2^{-M_0})$ follows by applying it to $2^{-T\ell}\approx_T \tilde{\delta}$.

Fix, then, $Z\subset X_y$ with $\mu_{X_y}(Z) \ge 2^{-\eta m/2}$, so that $\mu(Z)\gtrsim 2^{-\eta m/2}$. To conclude the proof, it is enough to show that
\begin{equation} \label{eq:delta-number-to-prove}
	\log |F_y Z|_{2^{-m}} \ge \left(\Sigma_{\tau}(s-O(\e^{1/2})) -O(\e^{1/2})\right) m.
\end{equation}

Let $(\mu_{X_i})_i$ be the decomposition of $\mu^{(m)}$ into uniform measures provided by Lemma \ref{lem:decomposition-uniform}. If $\eta<\e/2$, then we can find $i$ such that $\mu_{X_i}(Z)\gtrsim 2^{-\eta m}$. In \S\ref{subsubsec:setup-abstract-proj-thm} we have defined allowable intervals $([a_j,b_j])_{j=1}^J$ and $A_j=m a_j$, $B_j=m b_j$. From the definition of allowable, $J\le 1/\tau$ and $B_j\le 2 A_j$.

Recall the definition of the sets $\bad(x,m,j)$ from \eqref{eq:def-bad-set}. Since $y\in Y_x\setminus\bad(x)$ for all $x\in Z\subset X_y$,
\[
	P_{V_x(y)}\mu_{X_i}^{\cD_{A_j}(x)} \text{ is } (2^{-m_j},D(\sigma_j)-C\e,2^{-\e m_j})\text{-robust}.
\]
In light of Lemma \ref{lem:robust-to-entropy}, if $M_0 \gg_{\tau,\e} 1$, then
\begin{equation} \label{eq:robust-entropy-linearized}
	H_{m_j}^{2^{\e m_j/2}}\left(P_{V_x(y)}\mu_{X_i}^{\cD_{A_j}(x)}\right) \ge (D(\sigma_j)-O(\e))m_j.
\end{equation}

We apply Proposition \ref{prop:entropy-of-image-measure-robust} to $\mu_{X_i}$, $\mu_{X_i\cap Z}$ and the intervals $[A_j,B_j]$. Choosing $\eta$ as a small multiple of $\tau\e$, we conclude
\begin{align*}
	H_m(F_y \mu_{Z\cap X_i}) & \ge - O(J) + \int \sum_{j=1}^{J}  H_{m_j}^{m 2^{\eta m/2}}\left(P_{V_x(y)}\mu_{X_i}^{\cD_{A_j}(x)}\right) \,d\mu_{Z\cap X_i}(x)           \\
	                         & \overset{\eqref{eq:robust-entropy-linearized}}{\ge} - O(\tau^{-1}) +  \sum_{j=1}^{J}  (D(\sigma_j)-O(\e))m_j                              \\
	                         & \overset{\eqref{eq:optimal-sum-adapted},\eqref{eq:def-m-j}}{\ge}   - O(\e) m  + \left(\Sigma_{\tau}(s-O(\e^{1/2})) - O(\e^{1/2})\right)m,
\end{align*}
taking $m$ large enough that $\tau^{-1}\le \e m$. Recall that, for any measure $\rho$ supported on $[0,1)^d$, we have $H_m(\rho)\le \log|\supp\mu|_{2^{-m}}$. We have thus shown that \eqref{eq:delta-number-to-prove} holds, completing the proof of Theorem \ref{thm:nonlinear-abstract}.

\subsection{Hausdorff dimension versions}

Following standard arguments, Proposition \ref{prop: nonlinear-abstract-AD} and Theorem \ref{thm:nonlinear-abstract} have interpretations in terms of Hausdorff dimension.

\begin{theorem} \label{thm:abstract-proj-Hausdorff}
	Let $D:[0,d]\to [0,k]$ be a continuous function. Let $\mathcal{F}=(Y,\nu,U,F)$ be a regular family of projections. Let $X\subset U$ be a set such that for each $x\in X$ there is a set $Y_x$ with $\nu(Y_x)>0$, such that $\{ (x,y):x\in X,y\in Y_x\}$ is Borel and the measure $V_x \nu_{Y_x}$ is $D$-adapted at all sufficiently small scales.

	Then
	\[
		\sup_{y\in Y} \hdim(F_y X) \ge \sup_{\tau>0} \Sigma_{\tau}(D;\hdim(X)).
	\]

	If $\hdim(X)=\pdim(X)$, this can be improved to
	\[
		\sup_{y\in Y} \hdim(F_y X) \ge D(\hdim(X)).
	\]
	More generally, if $\pdim(X)\le \hdim(X)+\e$, then
	\[
		\sup_{y\in Y} \hdim(F_y X) \ge D(\hdim(X)) - O(\e).
	\]
\end{theorem}
\begin{proof}
	Consider first the case of a general set $X$ with $\hdim(X)=t$. Fix $\e,\tau>0$. By Lemma \ref{lem:Sigma-Lipschitz}, it is enough to show that there is $y\in Y$ such that
	\begin{equation} \label{eq:to-prove-Hdim-proj}
		\hdim(F_y X)  \ge \Sigma_{\tau}(t-\e)-\e.
	\end{equation}

	Let $\mu$ be a $(t-\e)$-Frostman measure on $X$. Without loss of generality, $\mu$ is supported on $[0,1)^d$. By restricting $\mu$ and $\nu$, we may assume that there are $c<1$, $\delta_1>0$ and sets $X'$, $(Y_x)_{x\in X'}$ with $\mu(X')> 1-c$, $\nu(Y_x)> 1-c$ and $V_x \nu_{Y_x}$ is $D$-adapted at all scales $\le \delta_1$. Moreover, we may assume that $\{ (x,y):x\in X', y\in Y_x\}$ is compact.

	We can then apply Theorem \ref{thm:nonlinear-abstract} (with $\delta=0$) to obtain a point $y\in Y$, a set $X_y\subset X$ with $\mu(X_y)>0$ and a number $\eta=\eta(\e,\tau)>0$, such that
	\[
		F_y(\mu_{X_y}) \text{ is  } (\delta,\Sigma_{\tau}(t-\e)-\e,\delta^{\eta})\text{-robust}
	\]
	for all sufficiently small $\delta$. Since $X_y\subset X$, the desired claim \eqref{eq:to-prove-Hdim-proj} follows from Lemma \ref{lem:robust-to-Frostman}.

	Suppose now that $\hdim(X)=t$ and $\pdim(X)\le t+\e$. Let $\mu$ be a $(t-\e)$-Frostman measure on $X$. Using (see e.g. \cite[\S 5.9]{Mattila95})  that
	\[
		\pdim(X) = \inf \left\{ \sup_i\ubdim(X_i): X\subset \cup_i X_i\right\},
	\]
	we may find a subset $X'\subset X$ with $\mu(X')>0$ and $\ubdim(X')<t+2\e$. By restricting $\mu$ to $X'$ we assume $\ubdim(\supp(\mu))< t+2\e$. Then
	\[
		\sum_{Q\in\cD_m:  \mu(Q)< 2^{(-t-3\e)m}} \mu(Q)   \lesssim 2^{(t+5\e/2)m}\cdot 2^{(-t-3\e)m} = 2^{-\e tm/2}.
	\]
	Thus, by discarding the cubes in $\cD_m$ with $\mu(Q)< 2^{(-t-3\e)m}$ for $m$ sufficiently large and restricting $\mu$ to the remaining set, we may assume that if $m\ge m_0$, then
	\[
		2^{(-t-3\e)m} \le \mu(Q) \le 2^{(-t+2\e)m} \quad\text{for all }Q\in\cD_m, \mu(Q)>0.
	\]
	We have shown that the hypotheses of Proposition \ref{prop: nonlinear-abstract-AD} are satisfied (with $3\e$ in place of $\e$). Note that $V_x\nu_{Y_x}$ is $(t\to D(t))$-adapted by definition. We conclude that there are $y\in Y$, $X_y\subset X$ with $\mu(X_y)>0$ and $\eta>0$ such that $F_y(\mu_{X_y})$ is $(2^{-m},D(t)-O(\e),2^{-\eta m})$-robust for all sufficiently large $m$. The last claim now follows from Lemma \ref{lem:robust-to-Frostman}. The case in which $\hdim(X)=\pdim(X)$ follows by letting $\e\to 0$.
\end{proof}

\section{Radial projections and distance sets in the plane}
\label{sec:plane}

\subsection{Introduction}

In this section we obtain improved estimates for radial projections and distance sets in the plane. The following is the planar case of Theorems \ref{thm:distance-conj-Ahlfors} and \ref{thm:radial-main-AD}.
\begin{theorem} \label{thm:planar-AD}
	Let $X\subset\R^2$ be a Borel set with $\hdim(X)=\pdim(X)=t\le 1$. Then for every $\eta>0$,
	\[
		\hdim  \{ y\in X: \hdim(\Delta^y(X)) \le \hdim(X)-\eta \} < t.
	\]
	In particular,
	\[
		\sup_{y\in X} \hdim(\Delta^y(X))  =\hdim(X),
	\]
	and if $X$ has positive $\mathcal{H}^{t}$-measure, then the supremum is realized at $\mathcal{H}^{t}$-almost all points.

	If $X$ is not contained in a line, then
	\[
		\sup_{y\in X} \hdim(\pi_y(X)) =\hdim(X).
	\]
\end{theorem}

For general sets, we have the following result, which combines Theorem \ref{thm:planar-distance} and the planar case of Theorem \ref{thm:radial-main-gral}. Recall the definition of the function  $\phi$ from \eqref{eq:def-phi}.
\begin{theorem} \label{thm:planar-general}
	Let $X\subset\R^2$ be a Borel set with $\hdim(X)= u\le 1$. Then
	\[
		\sup_{y\in X} \hdim(\Delta^y(X)) \ge \phi(u).
	\]
	If $X$ is not contained in a line, then
	\[
		\sup_{y\in X} \hdim(\pi_y(X)) \ge \phi(u).
	\]
\end{theorem}

\subsection{Thin tubes}

Recall that $|\mu|=\mu(X)$ denotes the total mass of the measure $\mu\in\cM_0(X)$, the family of finite compactly supported Borel measures on $\R^d$ (including the trivial measure).

\begin{definition} \label{def:thin-tubes}
	We say that a pair $(x,\nu)$, where $x\in\R^d$, $\nu\in\cP(\R^d)$, have \emph{$(t,K,c,\delta)$-thin tubes} if there exists a Borel function $\phi_x:\R^d\to\R$ such that $0\le \phi_x\le 1$, $\int\phi_x d\nu\ge c$, and
	\begin{equation} \label{eq:Frostman-tubes}
		\int_T \phi_x \,d\nu \le K\cdot r^t \quad\text{for all }r>\delta \text{ and all open $r$-tubes $T$ containing $x$}.
	\end{equation}

	Given $\mu\in\cP(\R^d)$, we say that $(\mu,\nu)$ has \emph{$(t,K,c,\delta)$-thin tubes} if
	\begin{equation} \label{eq:thin-tubes-mu-nu}
		\mu\big\{ x: (x,\nu) \text{ have $(t,K,c,\delta)$-thin tubes }\big\} \ge c.
	\end{equation}
	If $\delta=0$ (that is, \eqref{eq:Frostman-tubes} holds for all $r\in (0,1]$), then we drop it from the notation. If additionally the values of $K,c$ are not too important, we also drop them from the notation and simply say that $(x,\nu)$,  $(\mu,\nu)$ have $t$-thin tubes.

		If $\mu\in\cM(\R^d)$ and $\nu\in\cM(\R^d)$, we say that $(x,\nu)$, $(\mu,\nu)$ have $(t,K,c,\delta)$-thin tubes if the same holds for $\mu/|\mu|$, $\nu/|\nu|$ in place of $\mu,\nu$.
\end{definition}
We adopt the convention that if either of the measures $\mu$, $\nu$ is trivial, then $(\mu,\nu)$ have $t$-thin tubes. We will see in Lemma \ref{lem:Borel} that the set of $x$ appearing in \eqref{eq:thin-tubes-mu-nu} is compact, and so the concept of $(\mu,\nu)$-thin tubes is well defined.

Recall that $\pi_x:\R^d\setminus\{x\}\to S^{d-1}$ denotes radial projection, that is,
\[
	\pi_x(y) = \frac{y-x}{|y-x|}.
\]

\begin{remark} \label{rem:radial-to-tube}
	If $x$ and $\supp(\nu)$ are separated, then it is easy to see that $(x,\nu)$ having $t$-thin tubes is equivalent to the Frostman condition
	\[
		\pi_x(\phi_x d\nu)(B_r) \lesssim r^t,
	\]
	where the implicit constant depends on $K$ and the separation between $x$ and $\supp(\nu)$. Ultimately, we are interested in these radial projections, but it is slightly easier to work with tubes passing through $x$ as in Definition \ref{def:thin-tubes}.
\end{remark}

In the following lemmas we draw attention to some simple but important properties of tubes.
\begin{lemma} \label{lem:thin-projection}
	Let $x\in\R^d$, $\mu,\nu\in\cM(\R^d)$ and $V\in\mathbb{G}(\R^d,k)$. If $(P_V x,P_V\nu)$ (resp. $(P_V\mu,P_V\nu)$) has $(t,K,c,\delta)$-thin tubes, then $(x,\nu)$ (resp. $(\mu,\nu)$) have $(t,K,c,\delta)$-thin tubes.
\end{lemma}
\begin{proof}
	Immediate from the fact that if $T\subset\R^d$ is an $r$-tube passing through $x$, then $P_V T$ is contained in an $r$-tube passing through $P_V x$ (we simply lift the function $\phi_{P_V x}$ to $\phi_{P_V x}\circ P_V$).
\end{proof}

\begin{remark}\label{rmk: thintubesmeasure}
	Note that in the definition of thin tubes, Definition~\ref{def:thin-tubes}, by writing $\nu_x=\phi_x d\nu$, instead of considering functions $\phi_x\in L^\infty(\nu)$ we may equivalently consider measures $\nu_x$ which are absolutely continuous with respect to $\nu$ with total mass $\ge c$ and $\nu$-density bounded by $1$. We will often alternate between these two ways of looking at the definition; the original definition is more convenient when the measure $\nu$ is fixed, but in Section \ref{sec:high-dim-radial} we will need to understand how $\nu_x$ depends on $\nu$, and for this we need to consider a topology for the $\nu_x$ that is independent of $\nu$.
\end{remark}

From now on, on spaces of measures we always consider the weak$^\ast$ topology; see e.g \cite[p. 18]{Mattila95}. We recall (\cite[Theorem 1.23]{Mattila95}) that for any $K>0$, the space $\{ \mu\in\cM(\R^d): |\mu|\le K\}$ is sequentially compact.

We will write $\nu'\le \nu$ to denote that $\nu'\ll \nu$ with density bounded by $1$; equivalently, if $\nu'(Y)\le \nu(Y)$ for all Borel sets $Y$, or $\int fd\nu'\le \int f d\nu$ for all nonnegative $f\in C_0(\R^d)$. The latter characterization shows that $\{ (\nu,\nu'): \nu'\le \nu\}$ is a compact subset of $\cP_0(\R^d)\times\cP(\R^d)$. (Recall that $\cP_0(\R^d)$ denotes sub-probability measures.)

\begin{lemma} \label{lem:Borel}
	Let $B_0$ be a large closed ball, and fix $t,K>0$, $c\in(0,1)$, $\delta\ge 0$. Then
	\[
		\{ (x,\nu)\in B_0\times\cP(B_0): (x,\nu) \text{ have } (t,K,c,\delta)\text{-thin tubes}\}
	\]
	is a compact set $\mathcal{C}$. Moreover, it is possible to choose the measure $\nu_x$, see Remark~\ref{rmk: thintubesmeasure}, from the definition of thin tubes as a Borel function of $(x,\nu)$ on $\mathcal{C}$.
\end{lemma}
\begin{proof}
	Let
	\begin{align*}
		\mathcal{G} = \big\{ (x,\nu, & \nu')\in B_0\times\cP_0(B_0)\times\cP(B_0):                                                                   \\
		                             & |\nu'|\ge c, \nu'\le\nu, \nu'(T) \le K r^t \,\text{for all open $r$-tubes $T$ through $x$, $r>\delta$}\big\}.
	\end{align*}
	Note that, even though the tubes are assumed to be open, if $(x,\nu,\nu')\in\mathcal{G}$ then
	\[
		\nu'(\overline{T}) \le K r^t
	\]
	for all $r$-tubes containing $x$, since $\overline{T}$ is contained in an $(r+\e)$-tube for all $\e>0$. Hence if $(x_n,\nu_n,\nu'_n)$ is a sequence in $\mathcal{G}$ converging to $(x,\nu,\nu')$, and $T$ is an open $r$-tube containing $x$, then $T$ is an open $r$-tube containing $x_n$ for all large $n$, and hence
	\[
		\nu'(T)\le \nu'(\overline{T}) \le \liminf_{n\to\infty} \nu'_n(\overline{T}) \le K r^t.
	\]
	Together with our previous remark on the compactness of $\nu'\le\nu$, we see that $\mathcal{G}$ is a closed subset of $B_0\times\cP(B_0)\times \cP_0(B_0)$, hence compact.  Since $\mathcal{C}$ is the projection of $\mathcal{G}$ to the first two coordinates, it is compact. The assertion that $\nu_x$ can be chosen as a Borel function of $x$ and $\nu$ follows from standard Borel selection results, see e.g. \cite[Theorem 5.2.1]{Srivastava98}.
\end{proof}

We will often need to know that the function $\phi_x$ is actually an indicator function $\mathbf{1}_{Y_x}$ and, moreover, the corresponding set $\{ x\in X,y\in Y_x\}$ is compact. This motivates the following definition:
\begin{definition}
	Let $\mu,\nu\in\cP(\R^d)$. We say that $(\mu,\nu)$ have \emph{$(t,K,c,\delta)$-strong
		thin tubes} if there exists a set $X$ with $\mu(X)\ge c$ and for each $x\in X$ there is a set $Y_x$ such that $\nu(Y_x)\ge c$, the set $\{ (x,y):x\in X, y\in Y_x\}$ is compact, and
	\[
		\nu(Y_x\cap T) \le K\cdot r^t \quad\text{for all }r>\delta \text{ and all open $r$-tubes $T$ through $x$}.
	\]
	As before, we extend this to $\mu,\nu\in\cM(\R^d)$ by applying the definition to $(\mu/|\mu|,\nu/|\nu|)$, and declare that if either $\mu$ or $\nu$ are trivial, then $(\mu,\nu)$ have $(t,K,c,\delta)$. We also omit $\delta$ from the notation if $\delta=0$, and also $K,c$ when they are not important.
\end{definition}

Obviously, if $(\mu,\nu)$ have $t$-strong thing tubes, they have $t$-thin tubes. The converse is also true, up to a change in the parameters:
\begin{lemma} \label{lem:thin-tubes-to-strong-thin-tubes}
	Suppose $(\mu,\nu)$ have $(t,K,1-c,\delta)$-thin tubes. Then they have $(t,K',c',\delta)$-strong thin tubes, where $c'>  1-3c$, $K'=2K$  if $c<1/4$, and $c'> (1-\sqrt{c})/2$, $K'= (1-\sqrt{c})^{-1} K$ if $c\geq 1/4$.
\end{lemma}
\begin{proof}
	By Lemma \ref{lem:Borel}, we may assume that $x\mapsto \nu_x$ is Borel on
	\[
		X=\{x\in\supp(\mu):(x,\nu) \text{ has thin-tubes}\},
	\]
	and $\mu(X)\ge 1-c$ by assumption. On the other hand, for any $\mu,\nu\in\cP(\R^d)$ and $y\in\R^d$, $r>0$, the functions $\mu(B(y,r))$, $\nu(B(y,r))$ are Borel functions of $(\mu,\nu,y)$ as limits of continuous functions (approximating the characteristic function of balls by continuous functions). Hence $\lim_{r\to 0}\mu(B(y,r))/\nu(B(y,r))$ is a Borel function of $\mu,\nu,y$ wherever it exists. By the Lebesgue differentiation theorem in the form presented in \cite[Theorem 2.12]{Mattila95}, the limit does exist for $\nu$-almost all $y$ and equals the Radon-Nikodym density of $\mu$ with respect to $\nu$ when $\mu\ll \nu$. Writing $g_x$ for the Radon-Nikodym density $\tfrac{d\nu_x}{d\nu}$, we conclude from these considerations that $(x,y)\mapsto g_x(y)$ is a Borel function of $(x,y)$ (as a composition of Borel functions; we define the function to be $0$ where the density does not exist). Hence the set
	\[
		Z=\{ (x,y):x\in X, g_x(y) \ge a \}
	\]
	is Borel for any $a\ge 0$. We take $a=1/2$ if $c<1/4$, and $a= 1-\sqrt{c}$ if $c\geq 1/4$. Let
	\[
		Y_x= \{ y: g_x(y)\ge a\}=\{ y: (x,y)\in Z\}.
	\]
	By Markov's inequality applied to $1-g_x$ and $1-a$,
	\[
		\nu(Y_x) \ge 1- \frac{c}{1-a} \quad x\in X,
	\]
	while
	\[
		\nu(Y_x\cap T) \le a^{-1} \int_T g_xd\nu \le a^{-1}K\, r^t
	\]
	for any $r$-tube $T$ through $x$, $r>\delta$.

	We can take a compact subset $Z'\subset Z$ such that $\mu\times \nu (Z\setminus Z')\leq \e(c)$, where $\e(c)= c^2$ if $c < 1/4$ and $\e(c)= (1-\sqrt{c})^4$ if $c\geq 1/4$.  Set
	\[
		Y'_x=\{ y: (x,y)\in Z'\}\subset Y_x \quad\text{ and }\quad X'=\{ x\in X: \nu (Y_x\setminus Y_x') \geq \e(c)^{1/2} \}.
	\]
	By Fubini, we have $\mu(X') \leq \e(c)^{1/2}$. Let $X_0\subset X$ be a compact subset not intersecting $X'$ such that $\mu(X_0) > 1-3c$ if $c<1/4$, and $\mu(X_0)>  (1-\sqrt{c})/2$ if $c\geq 1/4$. Then
	the set $Z_0= \{(x,y)\in Z', x\in X_0\} $ is compact and for any $x\in X_0$, $\nu(Y_x') \geq 1-3c$ if $c<1/4$ and $\nu(Y_x') \geq (1-\sqrt{c})/2$ if $c\geq 1/4$, while we still have
	\[
		\nu(Y_x'\cap T) \leq a^{-1} K r^t, \quad T \text{ an $r$-tube through $x$}, r>\delta.
	\]
\end{proof}

We conclude the discussion of thin tubes by showing that, by restricting $\mu,\nu$ appropriately, we may improve $(t,K,c)$-thin tubes to $(t,K',1-\e)$-strong thin tubes for any $\e>0$.
\begin{lemma} \label{lem:thin-density}
	Suppose that $(\mu,\nu)$ have $t$-thin tubes. Then for every $\e>0$ there are $X$, $Y$ with $\mu(X)>0$, $\nu(Y)>0$ and $K>0$ such that $(\mu_X,\nu_Y)$ have $(t,K,1-\e)$-strong thin tubes.
\end{lemma}
\begin{proof}
	By Lemma \ref{lem:thin-tubes-to-strong-thin-tubes}, we may assume that $(\mu,\nu)$ have $t$-strong thin tubes. Let, then, $X_0$; $(Y_x)_{x\in X_0}$ be as in the definition of strong thin tubes. Let $\mathcal{D}$ be the collection of closed dyadic cubes. By the dyadic Lebesgue density theorem (or the martingale convergence theorem), for each $x\in X_0$, there is $Q(x)\in\mathcal{D}$ such that $\nu(Y_x\cap Q(x)) \ge (1-\e)\nu(Q(x))>0$. Pick $Y\in\mathcal{D}$ such that $Q(x)=Y$ for all $x$ in a compact set $X\subset X_0$ of positive $\mu$-measure. The claim holds since
	\[
		\nu_{Y\cap Y_x}(T) \le \nu(Y)^{-1}\nu_{Y_x}(T) \le K \nu(Y)^{-1} r^t
	\]
	for all $r$-tubes through $x$.
\end{proof}

\subsection{Orponen's radial projection theorem}

If two measures $\mu,\nu$ are supported on the same line, or if $\nu$ is supported on a zero-dimensional set, then $(\mu,\nu)$  cannot possibly have $t$-thin tubes for any $t>0$. The following result,  due to Orponen \cite{Orponen19}, shows that these are essentially the only possible obstructions. See \cite[Appendix B]{Shmerkin23} for the proof of this particular statement. It will play a key role as the base case in our bootstrapping arguments.
\begin{prop} \label{prop:Orponen-radial}
	Let $\mu,\nu\in\cP(\R^d)$ satisfy the Frostman condition $\mu(B_r), \nu(B_r)\lesssim r^s$ for some $s>0$. Assume also that $\nu(\ell)=0$ for all lines $\ell\in\mathbb{A}(\R^d,1)$. Then there is $t=t(s)>0$ such that for every $\e>0$ there is $K=K(\mu,\nu,\e)>0$ such that $(\mu,\nu)$ have $(t,K,1-\e)$-strong thin tubes.

	More precisely, for each $s>0$ there is a constant $c(s)>0$ such that if
	\[
		\mu(B_r),\nu(B_r) \le C\, r^s,\quad r\in [\delta,1],
	\]
	and
	\[
		\mu(\ell^{(\delta_0)}) \le c(s)\e,\quad\ell\in \mathbb{A}(\R^d,1),
	\]
	then $(\mu,\nu)$ have $(t,K,1-\e,\delta)$-strong thin tubes, where $K=K(C,s,\delta_0,\e)$.
\end{prop}

\subsection{Proof of Theorem \ref{thm:planar-AD}}

\label{subsec:proof-planar-AD}

We will prove the theorem by bootstrapping. Let $\mu$ be a Frostman measure on $X$. The starting point for the bootstrapping is provided by Proposition \ref{prop:Orponen-radial}. The bootstrapping step will be derived by combining the improved Kaufman's projection theorem (Theorem \ref{thm:kaufman-improvement}) and Proposition \ref{prop: nonlinear-abstract-AD}. Roughly speaking, the reason why bootstrapping works in this case is that the (direction of the) derivative of the radial projection $\pi_y(x)$ is a $\pi/2$-rotation of $\pi_y(x)$ itself. See Figure \ref{fig: radial projection}. Thus, in the linearization underlying Proposition \ref{prop: nonlinear-abstract-AD} (via Proposition \ref{prop:entropy-of-image-measure-robust}), the size of the radial projections $\pi_y\mu$ (measured by ``robustness'') is related back to linear projections of small pieces of $\mu$ in directions (orthogonal to) $\pi_y(x)$ for $\mu$-typical $x$. Suppose that we have already proved that these radial projections typically have ``dimension'' $s<t$. By the assumption of equal Hausdorff and packing dimension, the small pieces of $\mu$ typically also have ``dimension'' $t$ as well. Then Theorem \ref{thm:kaufman-improvement} ends up providing an improvement from $s$ to $s+\e$, where $\e=\e(s,t)>0$, and then one can iterate.

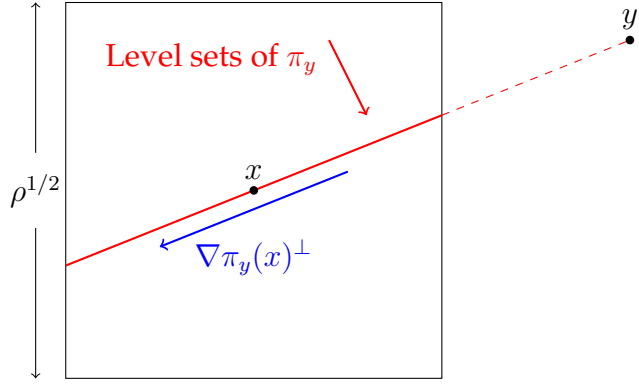
\begin{figure}
	\centering
	\begin{tikzpicture}[scale=0.5]
		\draw[black] (0,0)--(0,10)--(10,10)--(10,0)--cycle;
		\draw[red, thick] (0,3)--(10,7);
		\draw[red, dashed] (10, 7) -- (15, 9);
		\draw[black, fill=black] (5,5) circle[radius=0.1] node[above]{$x$};
		\draw[black, fill=black] (15,9) circle[radius=0.1] node[above]{$y$};
		\draw[<-, red, thick] (8,7)--(7,9);
		\draw[red, thick] (7,8.5) node[left]{Level sets of $\pi_{y}$};
		\draw[<-] (-0.8, 0)--(-0.8, 4);
		\draw[->] (-0.8, 6)--(-0.8, 10);
		\draw (-0.8, 5) node{$\rho^{1/2}$};
		\draw[<-, blue, thick] (2.5,3.5)--(7.5, 5.5);
		\draw[blue, thick] (5, 4) node[below]{$\nabla \pi_y(x)^{\perp}$};
	\end{tikzpicture}
	\caption{Linearization of radial projections}
	\label{fig: radial projection}
\end{figure}

The following lemma contains the details of the bootstrapping step.
\begin{lemma}\label{lem:kaufman-bootstrap}
	Fix $0<s<t\le 1$. There exist a  universal constant  $C>0$ and $\eta=\eta(s,t)>0$, such that the following holds for any $M\in [1, \infty]$.

	Let $\mu_1, \mu_2 \in \mathcal{P}([0,1]^2)$ satisfy  $\dist (\supp \mu_1 , \supp \mu_2) \gtrsim 1$ and
	\begin{equation}\label{eq: almost AD}
		K^{-1} 2^{-m(t+\frac{\eta}{C})} \leq \mu_i(Q) \leq K 2^{-m(t-\frac{\eta}{C})}
	\end{equation}
	for all $m \leq M$ and all $Q\in \mathcal{D}_m$ with $Q\cap \supp \mu_i \neq \emptyset$.

	If $(\mu_1, \mu_2)$ have $(s,K, 1-c, 2^{-M/2})$-strong thin tubes for some $c<1/100$, then
	\[
		(\mu_2, \mu_1) \text{ have } (s+\eta, K', 1-6 c^{1/2}, 2^{-M})\text{-strong thin tubes},
	\]
	for $K'=\max\{ O_{s,t}(1), 2K^{C^3 \eta^{-3}}, 2c^{-C^4 \eta^{-3}} \}$.
\end{lemma}
\begin{proof}
	Our goal is to apply Proposition \ref{prop: nonlinear-abstract-AD} to the family $(\pi_y)_{y\in \supp(\mu_2)}$. Note that in this case $V_x(y) = \left(\pi_x(y)\right)^\perp$.

	By the definition of strong thin tubes, there are a set $X'$ with $\mu_1(X')\ge 1-c$ and sets $(Y_x)_{x\in X'}$ with $\mu_2(Y_x) \ge 1-c$ such that
	\[
		\pi_x (\mu_2)_{Y_{x}}(B_r) \leq K r^s, \quad r\in [2^{-M/2},1],
	\]
	and $\{ (x,y):x\in X',y\in Y_x\}$ is compact. Since $V_x$ is a $\pi/2$ rotation of $\pi_x$, the same bound holds with $V_x$ instead of $\pi_x$. Now Theorem \ref{thm:kaufman-improvement} implies that  $(\mu_2)_{Y_{x}}$ is $(t\to s+\eta)$-adapted between scales $2^{-M/2}$ and $\min(K^{-1/\eta},\delta_0(s,t))$, where $\eta=\eta(s,t)>0$ . All the assumptions of Proposition \ref{prop: nonlinear-abstract-AD} are met with $M_1 = \eta^{-1} \log_2 K$. We deduce that there are $c'< 2 c^{1/2}$ and sets $Y'$, $(X_y)_{y\in Y'}$ with $\mu_2(Y')\ge 1-c'$, $\mu_1(X_y)\ge 1-c'$ and, taking $\e$ as a small multiple of $\eta$,
	\[
		\pi_y\left((\mu_1)_{X_y}\right) \text{ is } (\delta,s+\eta/2,\delta^{\frac{\eta^2}{4C^2}})\text{-robust},\quad\delta\in [2^{-M},2^{-M_0}],
	\]
	where $M_0 = \max\{ O_{s,t}(1),  C^3 \eta^{-3} \log_2 K, -C^4\eta^{-3} \log c\}$.

	The desired conclusion follows from Lemmas \ref{lem:robust-to-Frostman} and   \ref{lem:thin-tubes-to-strong-thin-tubes} and the choice of $K'$ in the statement (after an application of Lemmas \ref{lem:Borel}--\ref{lem:thin-tubes-to-strong-thin-tubes} to ensure that the resulting sets are compact).
\end{proof}

\begin{remark} \label{rem:eta-bounded-from-0}
	The value of $\eta$ in the lemma comes from Theorem \ref{thm:kaufman-improvement}, which in turn is obtained from that of \cite[Theorem 1.3]{OrponenShmerkin23}. In both cases, the loss is at most a multiplicative constant. As pointed out in \cite[Remark 1.4]{OrponenShmerkin23}, the value of $\eta$ can be taken to be bounded away from $0$ as long as $s$ is bounded away from $t$, and hence the same is true in Lemma \ref{lem:kaufman-bootstrap}.
\end{remark}

Starting with Orponen's radial projection theorem and iterating Lemma \ref{lem:kaufman-bootstrap}, we obtain the following corollary. It asserts that if $\mu_1,\mu_2$ are roughly $t$-Ahlfors regular and give zero mass to lines (or, more precisely, sufficiently small mass to tubes of a fixed size), then they have $(t-\zeta)$-strong thin tubes for any $\zeta>0$.

\begin{corollary} \label{cor:thin-tubes-AD}
	Given $t\in (0,1]$ and $\zeta>0$ there is a universal constant $C_0 >0$ and a small number  $\eta=\eta(t,\zeta)>0$ such that the following holds. For any $C>C_0$, let $\mu_1,\mu_2\in \mathcal{P}([0,1]^2)$ satisfy $\dist (\supp \mu_1, \supp \mu_2) \gtrsim 1$ and
	\[
		C^{-1} 2^{-m(t+\frac{\eta}{C} )} \leq \mu_i(Q) \leq C 2^{-m(t-\frac{\eta}{C} )},\quad Q\in\cD_m(\mu_i), m\le \tilde{M},
	\]
	for some $\tilde{M}\in\N\cup\{+\infty\}$. Moreover, fix a small $c>0$ and assume that
	\begin{equation} \label{eq:assumption-small-tubes}
		\mu_i(\ell^{(\delta_0)}) \le  c_0(c,t), \quad \ell\in \mathbb{A}(\R^2,1),
	\end{equation}
	where $c_0(c,t)$ is a small enough constant. Then there are $K = K(t, \zeta, C,c,\delta_0)$ and  $M_0=M_0(t, \zeta, C)$ such that if $\tilde{M}\ge M_0$, then $(\mu_1, \mu_2)$ have $(t-\zeta, K, 1-c, 2^{-\tilde{M}})$-strong thin tubes.
\end{corollary}
\begin{proof}
	Let $\eta=\min_{t/2\leq s\leq t-\zeta}  \eta(s,t)$ for the $\eta(s,t)$ in Lemma \ref{lem:kaufman-bootstrap}. Then $\eta>0$ by Remark \ref{rem:eta-bounded-from-0}.

	By Proposition \ref{prop:Orponen-radial} and the assumption \eqref{eq:assumption-small-tubes}, if $c_0$ is taken small enough then there is $t_0=t_0(t)>0$ such that both $(\mu_1, \mu_2)$ and $(\mu_2,\mu_1)$ have $(t_0, K_0, 1-c',\delta)$-strong thin tubes, where $c'\rightarrow 0$ as $K_0\rightarrow +\infty$. The claim follows by iterating Lemma \ref{lem:kaufman-bootstrap} $\lceil \eta^{-1}(t-t_0-\zeta)\rceil$ many times.
\end{proof}

We can now use Proposition \ref{prop: nonlinear-abstract-AD} to complete the proof of Theorem \ref{thm:planar-AD}.
\begin{proof}[Proof of Theorem \ref{thm:planar-AD}]
	Write $t=\hdim(X)$, fix $\zeta>0$, and let $\e=\eta(t, \zeta)<  \zeta$ be the number given by Corollary \ref{cor:thin-tubes-AD}. Let $\mu_1,\mu_2$ be $(t-\e/10)$-Frostman measures on $X$ with disjoint supports. If either measure $\mu_i$ gives positive mass to a line $\ell$, picking a point $y\in X\setminus\ell$ (for radial projections) or a point in $X\cap\ell$ (for distance sets) we have
	\[
		\hdim(\pi_y (X)), \hdim(\Delta^y (X)) \ge t-\e.
	\]
	Hence we assume from now on that $\mu_i(\ell)=0$ for all lines $\ell$. By a compactness argument, for any  $c_0(c,t)$, in particular the one  in \eqref{eq:assumption-small-tubes}, there exists $\delta_0$ such that \eqref{eq:assumption-small-tubes} holds for all $\ell \in \mathbb{A}(\R^2, 1)$.   By throwing away dyadic squares in $\cD_m$ of mass $\le 2^{-m(t+\e)}$ as in the proof of Theorem \ref{thm:abstract-proj-Hausdorff}, we may assume that
	\[
		2^{-m(t+\e)}\lesssim \mu_i(Q) \lesssim 2^{-m(t-\e)},\quad Q\in\cD_m.
	\]
	We can then apply Corollary \ref{cor:thin-tubes-AD} to $\mu_1,\mu_2$ to conclude that $(\mu_1,\mu_2)$ have $(t-\zeta)$-strong thin tubes. Let $Y=\supp(\mu_2)$ and note that for the families $(\Delta^y)_{y\in Y}$ and $(\pi_y)_{y\in Y}$, the corresponding functions $V_x$ are given by $\pi_x$ and $\pi_x^\perp$ respectively.

	Let $F_y$ denote either $\pi_y$ or $\Delta^y$ (the argument from now on is the same in both cases). Since  $(\mu_1,\mu_2)$ have $(t-\zeta)$-strong thin tubes by Corollary~\ref{cor:thin-tubes-AD}, Lemma \ref{lem:quantitative-Kaufman} implies that the assumptions of Proposition \ref{prop: nonlinear-abstract-AD} are satisfied for $(F_y)_{y\in Y}$, with $t-\zeta$ in place of $s$. The proposition provides a point $y\in\supp(\mu_2)\subset X$ and a set $Z\subset\supp(\mu_1)\subset X$ such that $F_y((\mu_1)_{Z})$ is $(\delta,t-\zeta-O(\e),\delta^{\e^2/4})$-robust for all $\delta<\delta(t, \zeta, \delta_0)$. We deduce from Lemma \ref{lem:robust-to-Frostman} that
	\[
		\hdim(F_y (X)) \ge \hdim(F_y (Z)) \ge t-\zeta-O(\e).
	\]
	The right-hand side can be made arbitrarily close to $t$. Given $\eta>0$, let $E_{\eta} =\{ y\in X: \hdim(F_y X)<t-\eta\}$. Then $\hdim(E_\eta)<\hdim(X)$, for otherwise we could apply our previous reasoning to $E_\eta$ instead of $X$ to obtain a contradiction. Letting $\eta\to 0$ along a sequence, we conclude that $\hdim(F_y X)=t$ for all $y$ outside of a countable union of sets of Hausdorff dimension $<t$, and in particular outside of a set of zero $\mathcal{H}^t$-measure.
\end{proof}

\subsection{Radial projections for general planar sets}
\label{subsec:radial-general}

The goal of this section is to prove Theorem \ref{thm:planar-general}. Later, in order to induct on the dimension of the ambient space we will also need a variant of the theorem for thin tubes:
\begin{prop} \label{prop:planar-thin-tubes}
	Fix $u\in (0,1]$. For every $\zeta>0$ there is $\e>0$ such that the following holds. Let $\mu,\nu\in\cP([0,1]^2)$ be measures with $\dist(\supp\mu,\supp\nu)\gtrsim 1$ such that
	\[
		\mu(B_r),\nu(B_r) \le C  r^{u-\e},\quad r\in (0,1].
	\]
	Fix a small constant $c>0$ and suppose there is $\delta_0>0$ such that $\mu(\ell^{(\delta_0)}), \nu(\ell^{(\delta_0)}) \le c_0(u, c)$ for all lines $\ell$, where $c_0(u, c)>0$ is a small constant depending only on $u$ and $c$.

	Then there is a constant $K=K(u, \zeta, C, c, \delta_0)$ such that  $(\mu,\nu)$ have $(\phi(u)-\zeta, K, 1-c, \delta)$-strong thin tubes for any $\delta <\delta(u, \zeta, C, c, \delta_0)$.
\end{prop}

We will use a bootstrapping scheme similar to that of Theorem \ref{thm:planar-AD}, but this time relying on Theorem \ref{thm:nonlinear-abstract}, and hence we need to estimate the combinatorial parameter $\Sigma_{\tau}(D;u)$ defined in \S\ref{subsec:def-comb}, for a suitable function $D$. In fact, in order to bootstrap we need to deal with the following family of functions:
Given $s\in (0,\phi(1)]$, let
\begin{equation} \label{eq:def-D-s}
	D_s(t) = \left\{\begin{array}{lll}
		\max\{ t/2, s+\eta(s,t)\} & \text{ if } & t\in (s,2] \\
		t                         & \text{ if } & t\in [0,s]
	\end{array}\right..
\end{equation}
Here $\eta(s,t)$ is a continuous function, non-decreasing in $t$ and strictly positive for $t>s$, and such that $\lim_{t\rightarrow s} \eta(s,t)=0$. It will be defined in Lemma~\ref{lem:planar-D-adapted} below.
Hence $D_s(t)$ is continuous, and satisfies $D_s(t)\ge t/2$ for all $t\in [0,2]$.
\begin{lemma} \label{lem:planar-D-adapted}
	Let $\rho\in\cP(S^{1})$ be a $(\delta,s,K)$-measure and fix $\e>0$. If $\delta< K^{-C'\e^{-1}}$ for some absolute constant $C'$, then  $\eta(s,t)$ can be chosen so that $\rho$ is $(D_s-\e)$-adapted at scale $\delta$ (recall Definitions \ref{def:adapted-1}, ~\ref{def: Dadapted}).
\end{lemma}
\begin{proof}
	This follows by combining Lemma \ref{lem:quantitative-Kaufman} (to handle the range $t\in [0,s]$), Theorem \ref{thm:kaufman-improvement} (for the lower bound $s+\eta(s,t)$), and Lemma \ref{lem:adapted-Bourgain} (for the lower bound $t/2$). By Remark \ref{rem:adapted-continuous}, we can indeed take the same $\delta_0$ for all values of $s$. Note that we are subtracting $\e$ from $D_s(t)$ to deal with the value $t=s$, since $\rho$ is not quite $(s\to s)$-adapted (but almost).
\end{proof}

The following is the key combinatorial estimate for the proof of Proposition  \ref{prop:planar-thin-tubes} and Theorem \ref{thm:planar-general}:
\begin{prop} \label{prop:planar-combinatorial}
	Fix $u\in (0,1]$. Given $\zeta>0$ there are $\xi>0$, $\tau>0$ (both depending on $u,\zeta$) such that if $s\in (0,\phi(u)-\zeta]$ then
	\[
		\Sigma_{\tau}(D_s;u)\ge s+\xi,
	\]
	where $\Sigma_{\tau}(D_s;u)$ here is for $d=2$.
\end{prop}
We continue with the proof of Proposition \ref{prop:planar-thin-tubes}, and defer the proof of Proposition \ref{prop:planar-combinatorial} to \S\ref{subsec:planar-combinatorial} below. We use an argument similar to that of Lemma \ref{lem:kaufman-bootstrap}, but based on Proposition \ref{prop:planar-combinatorial}.

\begin{lemma} \label{lem:bootstrapping-step-general}
	Fix $u\in (0,1]$ and $u_0\in (0,u/3]$. Given $\zeta>0$, $C>1$, there are $\eta=\eta(\zeta,u)>0$,  
	such that if $s\in [u_0,\phi(u)-\zeta]$ then the following holds. Let $\delta\ge 0$. Let $\mu_1,\mu_2$ be measures on $[0,1]^2$ with $\dist(\supp(\mu_1),\supp(\mu_2))\ge 1/C$, and
	\[
		\mu_i(B_r) \le C\,  r^{u-\eta},\quad r\in [\delta,1].
	\]
	If $(\mu_1, \mu_2)$ have $(s,K, 1-c, \delta^{1/2})$-strong thin tubes, then $(\mu_2, \mu_1)$ have $(s+\eta, K', 1-6 c^{1/2},\delta)$-strong thin tubes for $K'= 2^{ C'^8\eta^{-8}}K$, where $C'$  is the absolute constant in Lemma~\ref{lem:planar-D-adapted}.
\end{lemma}
\begin{proof}
	Our goal is to apply Theorem \ref{thm:nonlinear-abstract} to the family $(\pi_y)_{y\in \supp(\mu_2)}$. Since $V_x(y) = \left(\pi_x(y)\right)^\perp$, by the strong thin tubes assumption there are a set $X'$ with $\mu_1(X')\ge 1-c$ and sets $(Y_x)_{x\in X'}$ with $\mu_2(Y_x) \ge 1-c$ such that
	\[
		V_x (\mu_2)_{Y_x}(B_r) \leq K r^s, \quad r\in [\delta^{1/2},1],
	\]
	and $\{ (x,y):x\in X',y\in Y_x\}$ is compact. By Lemma \ref{lem:planar-D-adapted}, $V_x(\mu_2)_{Y_x}$ is $(D_s-\eta)$-adapted at scales in $[\delta^{1/2},  K^{-C'\eta^{-1}}]$. By Proposition \ref{prop:planar-combinatorial} and Lemma \ref{lem:Sigma-Lipschitz}, if $\eta,\tau$ are small enough in terms of $\zeta$ and $u$ only, then
	\[
		\Sigma_{\tau}(D_s-\eta, u -\eta) \ge s+2\eta.
	\]
	Theorem \ref{thm:nonlinear-abstract} then provides a number $c'<2 c^{1/2}$, and sets $Y'$, $(X_y)_{y\in Y'}$ with $\mu_2(Y')\ge 1-c'$, $\mu_1(X_y)\ge 1-c'$ such that
	\[
		\pi_y\left((\mu_1)_{X_y}\right) \text{ is } (\tilde{\delta},s+\eta,\tilde{\delta}^{\e})\text{-robust},\quad\tilde{\delta}\in (\delta,\delta_1),
	\]
	for some $\e=\e(\eta,\tau)>0$ and $\delta_1>0$ depending on all parameters. The conclusion follows from Lemma \ref{lem:robust-to-Frostman} combined with Lemma \ref{lem:thin-tubes-to-strong-thin-tubes}.
\end{proof}

\begin{proof}[Proof of Proposition \ref{prop:planar-thin-tubes}]
	By Proposition \ref{prop:Orponen-radial}, there is $t_0=t_0(t)$ such that if $c_0$ is taken small enough in terms of $c',t$ and the assumptions of the proposition hold, then $(\mu_1,\mu_2)$ and $(\mu_2,\mu_1)$ have $(t_0,K',c')$-strong thin tubes, where $c'\to 0$ as $K'\to\infty$. Then iterate Lemma \ref{lem:bootstrapping-step-general}   $\lceil \eta^{-1}(\phi(u)-t_0-\zeta) \rceil$ many times (where $\eta$ is given by Lemma \ref{lem:bootstrapping-step-general}).
\end{proof}

We record a consequence of Corollary \ref{cor:thin-tubes-AD} and Proposition \ref{prop:planar-thin-tubes} for later use. Obviously, if $\mu$ and $\nu$ give full mass to the same line, they cannot have thin tubes. This is the only obstruction:
\begin{corollary} \label{cor:planar-thin-tubes}
	Let $\mu,\nu\in\mathcal{P}(B_0)$, where $B_0$ is some closed ball. Suppose $\mu,\nu$ have disjoint supports and satisfy Frostman conditions with exponent $u\in (0,1]$. Assume furthermore that $\mu,\nu$ are not both concentrated on a single line. Then $(\mu,\nu)$ have $(\phi(u)-\zeta)$-strong thin tubes for all $\zeta>0$.

	If we additionally assume that $\mu,\nu$ satisfy
	\[
		C^{-1} 2^{-m(u+\e)} \leq \mu(Q),\nu(Q) \leq C 2^{-m(u-\e)},
	\]
	for all cubes $Q\in\cD_m(\mu),\cD_m(\nu)$ respectively, some $C\ge 1$ and some sufficiently small $\e=\e(\zeta)$, then $(\mu,\nu)$ have $(u-\zeta)$-strong thin tubes.
\end{corollary}
\begin{proof}
	We consider three cases:
	\begin{itemize}
		\item If both $\mu$ and $\nu$ give zero mass to all lines, then we apply Proposition \ref{prop:planar-thin-tubes} in the general case, and Corollary \ref{cor:thin-tubes-AD} in the ``almost Ahlfors-regular'' case.
		\item Suppose now that $\nu$ gives mass $>0$ to some line. Since $\mu$ and $\nu$ are not supported on the same line, we can find a line $\ell$ with $\nu(\ell)>0$ and a compact set $X$ disjoint from $\ell$ with $\mu(X)>0$. Then for $x\in X$, $\pi_x|_\ell$ is uniformly bi-Lipschitz (with constant depending on $\dist(\ell,X)$), and $(\mu,\nu)$ have $u$-strong thin tubes.
		\item Finally, suppose $\mu$ gives mass $>0$ to some line. As in the previous case, after suitable restriction we may assume that $\mu$ is supported on a line $\ell$ and $\nu$ is supported away from the $\delta$-neighborhood of $\ell$ for some $\delta=\delta(\mu,\nu)>0$. Without loss of generality, $\ell$ is the $x$-axis. In this case, a standard transversality estimate (see e.g. \cite[Theorem 18.3]{Mattila15}) applied to the parametrized family $\widetilde{\pi}_t(y)=(y_1-t)/y_2$ yields (writing $x=(t,0)$)
		      \[
			      \int \cE_{u-\zeta}(\pi_x \nu) \,d\mu(x) \lesssim_\delta \cE_{u-\zeta}(\nu) \lesssim_{\zeta} 1.
		      \]
		      It follows from Markov's inequality that $\cE_{u-\zeta}(\pi_x \nu)\le O_{\delta,\zeta}(1)$ for $x$ in a set $X$ with $\mu(X)\ge 1/2$. We conclude from Lemmas \ref{lem:energy-Frostman} and \ref{lem:thin-tubes-to-strong-thin-tubes} that $(\mu,\nu)$ have $(u-\zeta)$-strong thin tubes.
	\end{itemize}
\end{proof}

We can now easily conclude the proof of Theorem \ref{thm:planar-general}.
\begin{proof}[Proof of Theorem \ref{thm:planar-general}]
	Fix $u\in (0,1]$ and $\zeta>0$. Let $\eta>0$ be the value provided by Proposition \ref{prop:planar-thin-tubes}. Let $\mu_1,\mu_2$ be $(u-\eta)$-Frostman measures on $X$ with disjoint supports.
	Suppose first that some $\mu_i$ gives positive mass to a line $\ell$. Then, taking $y\in X\setminus\ell$ we get $\hdim(\pi_y (X))\ge u-\eta$, and taking $y\in\ell$ we get $\hdim(\Delta^y (X))\ge u-\eta$.

	We can therefore assume that $\mu_1,\mu_2$ both give zero mass to lines and (by compactness) the assumptions of Proposition \ref{prop:planar-thin-tubes} are met. Hence $(\mu_1,\mu_2)$ have $(\phi(u)-\zeta)$-strong thin tubes. This directly gives that there are $y\in \supp(\mu_2)$ and $X_y$ with $\mu_1(X_y)>0$ such that $\pi_y((\mu_1)_{X_y})(B_r)\lesssim r^{\phi(u)-\zeta}$, and so $\hdim(\pi_y X)\ge \phi(u)-\zeta$.

	For distance sets, we appeal to Theorem \ref{thm:abstract-proj-Hausdorff}. Since $V_x=\pi_x$ for the family $\{ \Delta^y: y\in\supp(\mu_2)\}$ and $(\mu_1,\mu_2)$ have $(\phi(u)-\zeta)$-strong thin tubes, we get from Lemma \ref{lem:planar-D-adapted} that the hypotheses of Theorem \ref{thm:abstract-proj-Hausdorff} are satisfied with $D_{\phi(u)-\zeta}-\zeta$ in place of $D$. By Lemma \ref{lem:Sigma-Lipschitz} and Proposition \ref{prop:planar-combinatorial}, there is $\tau>0$ such that (making $\eta$ smaller if needed)
	\[
		\Sigma_{\tau}(D_{\phi(u)-\zeta}-\zeta) \ge \phi(u) - 3\zeta.
	\]
	Since $\zeta>0$ is arbitrary, the conclusion now follows from Theorem \ref{thm:abstract-proj-Hausdorff}.
\end{proof}

\subsection{Combinatorial preliminaries}
\label{subsec:comb-preliminaries}

Before embarking on the proof of Proposition \ref{prop:planar-combinatorial}, we discuss some general lemmas that will also be useful later.

We start with a straightforward merging lemma. Recall Definition \ref{eq:def-superlinear}.
\begin{lemma}\label{lem:merge}
	If $(f, a, b)$ is $\sigma_1$--superlinear  and $(f, b, c)$ is $\sigma_2$--superlinear with $\sigma_1\geq \sigma_2$, then
	$(f, a, c)$ is $\sigma$--superlinear with
	\[ \sigma= \frac{\sigma_1(b-a)+\sigma_2(c-b)}{c-a}.\]
\end{lemma}
\begin{proof}
	Since $\sigma_1 \geq \sigma \geq \sigma_2$,
	\[
		f(x)\geq f(a) + \sigma (x-a), \quad x\in [a, b ].
	\]
	When $x\in [b, c]$,
	\begin{align*}
		f(x) & \geq f(b) +\sigma_{2} (x-b)                 \\
		     & \geq f(a) + \sigma_1 (b-a) + \sigma_2 (x-b) \\
		     & \geq  f(a)  +\sigma(x-a)
	\end{align*}
	by the definition of $\sigma$ and that $\sigma_1 \geq \sigma_2$.
\end{proof}

We can merge consecutive intervals so that the numbers $\sigma_k$ become increasing:
\begin{lemma}\label{lem:merge-increasing}
	Let $\{ [a_j, a_{j+1}]\}_{j=1}^{J}$ be a collection of intervals, let $f:[a_{1},a_{J+1}]\to\R$ be a function, and let $(\sigma_j)_{j=1}^{J}$ be numbers such that $(f,a_j,a_{j+1})$ is $\sigma_j$-superlinear. Then we can decompose the interval $[a_{1}, a_{J+1}]$ into consecutive intervals $\{[a_k', a_{k+1}']\}$  such that
	\begin{enumerate}
		\item $(f, a_k', a_{k+1}')$ is $\sigma_k'$--superlinear with
		      \[
			      \sum_{k} \sigma'_k(a'_{k+1}-a'_k) = \sum_{j=1}^{J} \sigma_j(a_{j+1}-a_j).
		      \]
		\item $\sigma_k' < \sigma_{k+1}'$.
	\end{enumerate}
	If $f$ is linear on each $[a_j,a_{j+1}]$ with slope $\sigma_j$, then we also get $f(a'_{k+1})-f(a'_k)=\sigma'_k(a'_{k+1}-a'_k)$ for all $k$.
\end{lemma}
\begin{proof}
	Using Lemma~\ref{lem:merge}, merge $[a_j, a_{j+1}]$ and $[a_{j+1}, a_{j+2}]$ if $\sigma_j \geq \sigma_{j+1}$.  Repeat the merging process until no consecutive intervals satisfy the merging condition. The last claim follows in the same way, by inspecting the proof of Lemma \ref{lem:merge}.
\end{proof}

The following is a modified version of \cite[Corollary 4.5]{Shmerkin23}. It shows that it is always possible to find consecutive allowable intervals $[a_j,a_{j+1}]$ on which $f$ is $\sigma_j$-superlinear, in such a way that $\sum_j (a_{j+1}-a_j)\sigma_j$ nearly exhausts the growth of $f$ on $[0,1]$.
\begin{lemma}\label{lem:superlinear-decomposition}
	Given $\rho,d>0$ and $\e\in (0,1/4)$, there is $\tau=\tau(\e,\rho,d)>0$ such that the following holds: for any non-decreasing, piecewise linear $d$-Lipschitz function $f:[a,b]\to \R$, where $a, b-a\ge \rho$, there exists a family of consecutive intervals $([a_j,a_{j+1}])_{j=1}^J$ with $a_1=a$, $a_{J+1}=b$, such that:
	\begin{enumerate}[(i)]
		\item \label{it:i:allowable} For each $j$,
		      \[
			      \tau \le a_{j+1}-a_j \le \rho \le a_j.
		      \]
		\item \label{it:ii:exhausts} $(f,a_j,a_{j+1})$ is $\sigma_j$-superlinear, where
		      \[
			      \sum_{j=1}^J (a_{j+1}-a_j)\sigma_j \ge f(b)-f(a)-\e(b-a).
		      \]
	\end{enumerate}
\end{lemma}
\begin{proof}
	We will obtain the claim with $(d+2)\e^{1/2}$ in place of $\e$.

	Split $[a, b]$ into intervals of length between $\rho/4$ and $\rho/2$. In our notation, \cite[Corollary 4.5]{Shmerkin23} (applied to $f/d$ which is $1$-Lipschitz) reads as follows: there is a number $\tau=\tau(\rho,\e,d)>0$ and a collection of nonoverlapping intervals $\{[a_j,b_j]\}_{j=1}^J$ contained in $[a,b]$, such that
	\begin{enumerate}[(a)]
		\item \label{it:i':allowable} Each $[a_j,b_j]$ is $\tau$-allowable and has length $\le \rho/2$; also, $b_{j+1}-a_j\le \rho$.
		\item \label{it:ii':superlinear}
		      \[
			      f(x) \ge f(a_j)+  \tfrac{f(b_j)-f(a_j)}{b_j-a_j}(x-a_j)-\e(b_j-a_j),\quad x\in [a_j,b_j].
		      \]
		\item \label{it:iii':exhausts} $\sum_{j=1}^J b_j-a_j \ge (1-\e/d)(b-a)$.
	\end{enumerate}

	To begin, note that since $f$ is $d$-Lipschitz, it follows from \eqref{it:iii':exhausts} that the total growth of $f$ on $[a,b]\setminus \cup_j[a_j,b_j]$ is at most $\e(b-a)$, and therefore
	\begin{equation} \label{eq:exhausts}
		\sum_{j=1}^J f(b_j)-f(a_j) \ge f(b)-f(a)-\e(b-a).
	\end{equation}

	For each $j$, use Lemma \ref{lem:merge-increasing} (which is applicable since $f$ is piecewise linear) to obtain intervals $([c_{j,i},c_{j,i+1}])_{i=1}^{I_j}$ with $c_{j,1}=a_j$, $c_{j,I_j+1}=b_j$ and numbers $\sigma'_{j,i}$, which increase with $i$, such that $(f,c_{j,i},c_{j,i+1})$ is $\sigma'_{j,i}$-superlinear and
	\[
		\sum_{i=1}^{I_j} \sigma'_{j,i}(c_{j,i+1}-c_{j,i}) = f(b_j)-f(a_j).
	\]
	Let $i(j)$ be the largest $i$ such that $c_{j,i}\le a_j +\e^{1/2}(b_j-a_j)$. Let $a'_j=c_{j,i(j)}$, $a''_j=c_{j,i(j)+1}$ and $\sigma'_j=\sigma'_{j,i(j)}$. Since the $\sigma'_{j,i}$ are increasing, $f(a''_j) \le f(a_j)+\sigma'_j(a''_j-a_j)$. Comparing this with \eqref{it:ii':superlinear} applied with $x=a''_j$, we deduce that
	\[
		f(a_j)+\sigma'_j(a''_j-a_j) \ge f(a_j)+  \tfrac{f(b_j)-f(a_j)}{b_j-a_j}(a''_j-a_j) - \e(b_j-a_j).
	\]
	Since $a''_j-a_j \ge \e^{1/2}(b_j-a_j)$, we deduce that
	\begin{equation} \label{eq:bound-sigma'}
		\sigma'_j \ge \tfrac{f(b_j)-f(a_j)}{b_j-a_j} - \e^{1/2}.
	\end{equation}
	On the other hand, again using that $\sigma'_{j,i}$ is increasing, we see that $(f,a'_j,b_j)$ is $\sigma'_j$-superlinear.

	Set $a'_{J+1}=b$. Since $f$ is non-decreasing, $(f,b_j,a'_{j+1})$ is $0$-superlinear. Now merge $[a'_j,b_j]$ with $[b_j,a'_{j+1}]$, using Lemma \ref{lem:merge}. Then $(f,a'_j,a'_{j+1})$ is $\sigma_j$-superlinear for $\sigma_j = (b_j-a'_j)\sigma'_j/(a'_{j+1}-a'_j)$, so that
	\begin{align*}
		\sum_{j=1}^J (a'_{j+1}-a'_j)\sigma_j & = \sum_{j=1}^J (b_j-a'_j)\sigma'_j \ge  \sum_{j=1}^J (b_j-a_j)(\sigma'_j - d\e^{1/2})       \\
		                                     & \overset{\eqref{eq:bound-sigma'}}{\ge} \sum_{j=1}^J f(b_j)-f(a_j) -  (d+1)\e^{1/2}(b_j-a_j) \\
		                                     & \overset{\eqref{eq:exhausts}}{\ge} f(b)-f(a)-(d+2)\e^{1/2}(b-a).
	\end{align*}
	Finally, $a'_{j+1}-a'_j \ge (b_j-a_j)/2\ge \tau/2$ and  $a'_{j+1}-a'_j\le \rho\le a'_j$, whence the intervals $[a'_j,a'_{j+1}]$ are $\tau/2$-allowable. This yields the claim with $ (d+2)\e^{1/2}$ in place of $\eps$, $\tau/2$ in place of $\tau$, $a'_j$ in place of $a_j$, and $J-1$ in place of $J$.
\end{proof}

\subsection{Proof of Proposition \ref{prop:planar-combinatorial}}

\label{subsec:planar-combinatorial}

\begin{proof}[Proof of Proposition \ref{prop:planar-combinatorial}]
	Fix $u\in (0,1)$ and $\zeta>0$. We will prove the claim by induction in $s$. The base case is $s\in (0,u/3]$ and it follows by applying Lemma \ref{lem:superlinear-decomposition} with $\e=u/100$ to the given $f\in\mathcal{L}_{2,u}$, and using the bound $D_s(t)\ge t/2$ for all $s>0$ and all $t\in [0,2]$ (we can take $\xi=u/100$ here).

	The rest of the proof concerns the inductive step: if $\xi$ is taken sufficiently small in terms of $u$ and $\zeta$, and the claim of the proposition holds for some $s\in [u/3,\phi(u)-\zeta]$, then it holds for $s+\xi$. The value of $\tau$ may depend on $s$ in addition to $u$ and $\zeta$; however, as the inductive step is applied only finitely many steps, eventually one can take a uniform $\tau=\tau(u,\zeta)>0$. We emphasize that $\xi$ is independent of $s$.

	Fix, then, $s\in [u/3,\phi(u)-\zeta]$, a small $\xi>0$ (to be determined along the proof), and $f\in\mathcal{L}_{2,u}$.

	Let
	\begin{equation}\label{eq: defa}
		a = \sup\{ a' \in [0,1/2]: (f, a', 2a') \text{ is } s\text{--superlinear}\}.
	\end{equation}
	We claim that the set in question is nonempty, and moreover
	\begin{equation} \label{eq:bound-a}
		a\geq \frac{1}{2}\frac{u-\phi(u)}{2-\phi(u)} =: a_{\min}.
	\end{equation}
	Assume for the sake of contradiction that \eqref{eq:bound-a} fails (or the set is empty). Pick $a<a_0<a_{\min}<1/2$ (or $a_0=a_{\min}/2$ if the set in \eqref{eq: defa} is empty).

	By definition, $(f, a_0, 2a_0)$ is not $s$-superlinear. Recalling the Definition \eqref{eq:def-superlinear}, this means that there is $a_1>a_0$ such that $f(a_1) < f(a_0) + s(a_1-a_0)$. By replacing $a_1$ with one of the two endpoints of the interval containing $a_1$ on which $f$ is linear, we can assume that $a_1$ is an endpoint of such an interval.

	If $a_1\ge 1/2$, we stop. Otherwise, we repeat the argument with $a_1$ in place of $a_0$. This process must stop after finitely many steps, as $a_j$ is strictly increasing along endpoints. Let $a_k$ be the last point in the process.

	Then, by the assumptions on $f$ and telescoping, we get
	\[
		u a_k \le f(a_k) < f(a_0) + s(a_k-a_0) \le 2 a_0 + s(a_k-a_0).
	\]
	Since $a_k\ge 1/2$ and $s\le \phi(u)$, this contradicts $a_0<a_{\min}$, and therefore \eqref{eq:bound-a} must hold.

	Up to replacing $f(x)$ by $\min\{f(x), u\}$, one can assume that $f(1)=u$. We will split the analysis into various sub-cases, each requiring a different argument. We note that the inductive hypothesis is only used in one sub-case.

	\medskip

	\noindent\textbf{Case I.} Suppose first that $a<1/2$.

	Let
	\begin{equation}\label{eq: defb}
		b=\max\{x\in (a, 2a]: f(x)-f(a) = s(x-a) \}.
	\end{equation}
	Such $b$ must exist by continuity, the definition of $a$, and the assumption $a<1/2$ (otherwise, we would have $f(x)>f(a)+s(x-a)$ for all $x\in (a,2a]$, which, since $f$ is piecewise-linear, would imply that $(f,a',2a')$ is $s$-superlinear for some $a<a'<1/2$).
	\smallskip

	\noindent \textbf{Case I.A.}
	Suppose that the set $C$ defined by
	\begin{equation}\label{eq: defC}
		C=\{ c \in (b, \min\{2b, 1\}]: f(c) \leq f(b)+(c-b)s \}
	\end{equation}
	is nonempty. (The case $C=\emptyset$ will be treated in Case I.B.)

	If $b=2a$, set $c=b=2a$. Otherwise, if $b<2a$, note that $C$ must have a minimum, and let $c=\min C$. By the definition of $b$, we have that $c\ge 2a$ also in this case (in fact, $c>2a$). Moreover, $(f,b,c)$ is $s$-superlinear, and $f(c)= f(b)+ s(c-b)$.

	\smallskip

	\noindent \textbf{Case I.A.1.} $c<1/2$. Consider the decomposition of $[c, 1]$ into consecutive intervals $\{[a_j, a_{j+1}]\}_{j=1}^J$ provided by Lemma~\ref{lem:superlinear-decomposition}, applied with $\xi,\xi^2$ in place of $\e,\rho$.  Let $\tau=\tau(\xi)>0$ and $\{\sigma_j\}_{j=1}^J$ be the numbers given by the lemma. We will merge the intervals using Lemma~\ref{lem:merge-increasing} in a way that the resulting intervals $[a_k,a_{k+1}]$ are allowed, and the corresponding $\sigma_k$ are $\le u$.

	Start with $a_1=c$. Pick the $K$ such that $a_K\le 2a_1$ and $a_{K+1}> 2a_1$. Merge the intervals $[a_1,a_2],\ldots, [a_{K-1},a_K]$ using Lemma \ref{lem:merge-increasing}; rename the resulting intervals and numbers $\sigma_j$ to keep the same notation.

	We claim that $\sigma_1\le s+\xi\le u$. Indeed, if this is not the case then, since the $\sigma_j$ are increasing, $f(x) \ge f(a_1)+(s+\xi)(x-a_1)$ for all $x\in [a_1,a_K]$. Since $2a_1\le a_{K+1}\le a_K+ \xi^2$ and $f$ is non-decreasing, for $x\in [a_K,2a_1]$ we have
	\[
		f(x) \ge f(a_K) \ge f(a_1)+(s+\xi)(a_K-a_1) \ge f(a_1)+s(x-a_1),
	\]
	assuming $\xi$ was taken sufficiently small (recall that $a_1\ge a_{\min}$). Since $1/2>a_1>a$, this contradicts the definition of $a$ from \eqref{eq: defa}.

	If $a_2\ge 1/2$ we stop; otherwise, we repeat the argument with $a_2$ in place of $a_1$. Continuing inductively, we obtain a sequence of allowable intervals $([a_j,a_{j+1}])_{j=1}^K$ such that $\sigma_j\le s+\xi$ and $a_{K+1}\ge 1/2$.

	Now we merge the remaining intervals $[a_j,a_{j+1}]$ with $j\ge K+1$ using Lemma \ref{lem:merge-increasing}, and keep the same notation. Since $f(x)\ge u x$ and we are assuming $f(1)=u$, we must have $\sigma_j\le u$ for $j=J$ (the last one) and therefore for all $j$.

	By Lemma \ref{lem:superlinear-decomposition}\eqref{it:ii:exhausts} and the fact that the intervals were obtained from the original ones by repeated applications of Lemma \ref{lem:merge-increasing}, we see that
	\[
		\sum_{j=1}^J (a_{j+1}-a_j)\sigma_j \ge f(1) - f(c) - \xi(1-c) \ge u - [f(a)+s(c-a)]-\xi.
	\]
	Note from \eqref{eq:def-D-s} that $D_s(t)\ge st/u$ for all $t\in [0,u]$. Since $\sigma_j\le u$ for all $j$, we deduce that
	\[
		\Sigma_{\tau}(D_s;f;[c,1]) \ge s\left(1-\frac{f(a)}{u}-\frac{s}{u}(c-a)\right) - \xi.
	\]

	On the interval $[0,a]$, we apply Lemma~\ref{lem:superlinear-decomposition} and use the bound $D_s(t)\ge t/2$ valid for all $t\in [0,2]$ to obtain
	\[
		\Sigma_{\tau}(D_s;f;[0,a]) \ge \frac{f(a)-f(0)}{2}-\xi = \frac{f(a)}{2}-\xi.
	\]

	Finally, on the interval $[a,c]$ we use the intervals $[a,b]$ and $[b,c]$, unless one of them has length $\le\tau\le\xi$, in which case we skip it. This includes in particular the case $c=b=2a$. Since $D_s(s)=s$ by \eqref{eq:def-D-s}, we get
	\[
		\Sigma_{\tau}(D_s;f;[a,c]) \ge f(c)-f(a)-\xi = s(c-a)-\xi.
	\]

	Combining all three estimates, we obtain
	\begin{align*}
		\Sigma_{\tau}(D_s; f) & \geq   \frac{f(a)}{2} + s(c-a) + s(c-a) + s\left(1-\frac{f(a)}{u}-\frac{s}{u}(c-a)\right)  -3\xi \\
		                      & \geq s + (1-s/u)s(c-a) -(s/u-1/2)f(a)-3\xi.
	\end{align*}
	Since $f(a) \leq 2a$ and $c\geq 2a$,
	\begin{align*}
		\Sigma_{\tau}(D_s; f) & \geq s+a(1-s/u)s-(s/u-1/2)2a -3\xi                       \\
		                      & \ge s+\frac{a_{\min}}{u}(u-(2-u)s-s^2)-3\eta \geq s+\xi,
	\end{align*}
	using \eqref{eq:bound-a}, and choosing $\xi$ sufficiently small depending on $u,\phi(u)-\zeta$ only.

	\smallskip

	\noindent \textbf{Case I.A.2.} $c\geq 1/2$.
	This case is similar to Case I.A.1 but simpler. Let $\{ [a_j,a_{j+1}]\}_{j=1}^J$ be obtained by first applying Lemma \ref{lem:superlinear-decomposition} to $f$ on $[c,1]$, and then applying Lemma \ref{lem:merge-increasing} to the resulting intervals. Then the corresponding $\sigma_j$ are $\le u$, and one continues as in Case I.A.1.

	\medskip

	\noindent \textbf{Case I.B.} The set $C$ defined in \eqref{eq: defC} is empty. Recall that $a, b$ were defined as in \eqref{eq: defa} and \eqref{eq: defb} and that $b< 2a$.  Since $C$ is empty, $(f,2a,1)$ is $s$-superlinear. Also, for any $x\in [b, 2a]$,
	\[
		f(x)-f(b)  = (f(x)-f(a))-(f(b)-f(a))  \ge  s(x-a)-s(b-a) = s(x-b),
	\]
	and so $(f,b,2a)$ is $s$-superlinear. Hence $(f,b,1)$ is $s$-superlinear. This implies that $b\ge 1/2$, for otherwise this would contradict \eqref{eq: defa} and $b>a$.

	We further split into two subcases.

	\smallskip

	\noindent \textbf{Case I.B.1.} $f(1)-f(a)\leq \frac{16}{15}(1-a)s$. As in case I.A.1,
	\[
		\Sigma_{\tau}(D_s; f; [0,a]) \geq f(a)/2 - \xi.
	\]
	On $[a,1]$ we use the intervals $[a,b]$ and $[b,1]$ (unless one has length $\le\tau$, in which case we ignore it). Since $f$ is $s$-superlinear on both these intervals and $D_s(s)=s$,
	\[
		\Sigma_{\tau}(D_s; f; [a, 1]) \geq s(1-a) -\eta \geq \frac{15}{16} (f(1)-f(a))-\xi.
	\]
	Since $u=f(1) \geq f(a)+ s(1-a) \geq f(a)+s/2$ and $1/2\le s\le \phi(u)\le (5/8)u$, we get
	\begin{align*}
		\Sigma_{\tau}(D_s; f) & \geq \frac{15}{16}u - f(a)(\frac{15}{16}-\frac{1}{2})-2\xi            \\
		                      & \geq \frac{15}{16}u -(u-\frac{s}{2})(\frac{15}{16}-\frac{1}{2}) -2\xi \\
		                      & \geq  \frac{u}{2}+\frac{7s}{32} \ge  1.01 s,
	\end{align*}
	so, since $s\ge u/3$, we can take $\xi=u/300$.

	\smallskip

	\noindent \textbf{Case I.B.2.}  $f(1)-f(a) \geq \frac{16}{15}(1-a)s$.

	Apply Lemma~\ref{lem:merge-increasing} to the subinterval $[b, 2a]$ union the  set of subintervals in $[2a,1]$ provided by Lemma \ref{lem:superlinear-decomposition} (applied with $\e=\xi$). We obtain a decomposition of $[b, 1]$ into $\tau$-allowable subintervals $[a_j, a_{j+1}]$ with $\sigma_j \leq \sigma_{j+1}$ and
	\begin{equation}\label{eq: A2}
		\sum_{j} \sigma_j(a_{j+1}-a_j) \geq f(1)-f(b) -\xi (1-b).
	\end{equation}

	Since $f(x)\geq u x$ and $f(1)=u$, we know that $\sigma_{k} \leq u$.  Let $\sigma_k$ be the smallest number such that $\sigma_k'\geq \frac{31}{30}s$ and write $c=a_k$.  Then
	\[
		\tfrac{31}{30}s(c-a) + u(1-c) \ge f(1)-f(a)\ge \tfrac{16}{15}(1-a)s,
	\]
	which after some algebra, and using $a<1/2$, $s\ge u/3$  yields
	\[
		c \leq \frac{u-\frac{63s}{60}}{u-\frac{31s}{30}} \leq 1 - \frac{1}{1000}.
	\]
	Let $\eta$ be the function from Lemma \ref{lem:planar-D-adapted}, and define
	\[
		\tilde{\eta}=\underset{ u/3 \leq s \leq \phi(u)}{\min} \eta(s,\tfrac{31}{30}s)>0.
	\]
	From $D_s(t)\ge D_s(\frac{31}{30}s)\ge s+\tilde{\eta}$ for $t\ge \tfrac{31}{30}s$, we get
	\begin{equation} \label{eq:Sigma[c,1]}
		\Sigma_{\tau}(D_s; f;[c, 1]) \geq (s+\tilde{\eta})(1-c) - \xi.
	\end{equation}

	To deal with the interval $[b,c]$, we note that $\sigma_k(a_{k+1}-a_k)\le f(a_{k+1})-f(a_k)$ and therefore from \eqref{eq: A2} we get that
	\[
		\sum_{j=1}^{\ell} \sigma_j(a_{j+1}-a_j) \geq f(a_{\ell+1})-f(b)-\xi.
	\]
	Now let $\ell$ be largest such that $\sigma_\ell\le s-\xi^{1/2}$. Then, since the $\sigma_j$ are increasing and $(f,b,1)$ is $s$-superlinear,
	\[
		(s-\xi^{1/2})(a_{\ell+1}-b) \geq f(a_{\ell+1})-f(b)-\eta \ge s(a_{\ell+1}-b)-\xi,
	\]
	from where we get that $a_{\ell+1}\le b+\xi^{1/2}$. Since $D_s(t)\ge s-\xi^{1/2}$ for $t\ge s-\xi^{1/2}$, and $(f,a,b)$ was $s$-superlinear, we conclude that
	\begin{equation}\label{eq:Sigma[a,c]}
		\Sigma_{\tau}(D_s; f;[a, c]) \geq (s-\xi^{1/2})(c-b-\eta^{1/2}) \ge s(c-a)-3\xi^{1/2}.
	\end{equation}
	In order to handle the remaining term $\Sigma_{\tau}(D_s; f; [0,a])$ we apply the inductive hypothesis to $s-\xi/2$ and the rescaled function $\wt{f}(x)=a f(x/a)$. Let $\tau'$ be the value of $\tau$ corresponding to $s-\xi$. By scaling the intervals obtained from applying the inductive hypothesis by a factor of $a\ge a_{\min}$, and since $D_s\ge D_{s-\xi}$, we get
	\begin{equation} \label{eq:Sigma[0,a]}
		\Sigma_{a_{\min}\tau'}(D_s; f; [0,a]) \geq (s-\xi) a.
	\end{equation}

	Combining all three estimates \eqref{eq:Sigma[c,1]}, \eqref{eq:Sigma[a,c]} and \eqref{eq:Sigma[0,a]}, and recalling that $1-c>1/1000$, we conclude that if $\xi$ is small enough in terms of $\tilde{\eta}$ (hence only in terms of $u$), then
	\[
		\Sigma_{a_{\min} \tau}(D_s; f) \geq s+\tilde{\eta}/2000.
	\]

	\medskip

	\noindent\textbf{Case II}. $a=1/2$.

	This case can be handled exactly like the case I.B.2 above (with $a$ playing the role of $b$).

\end{proof}

\subsection{Discretized versions}
\label{subsec:discretized}

We conclude this section by pointing out that the proofs of Theorems~\ref{thm:planar-AD} and \ref{thm:planar-general} also yield as a by-product discretized statements at a fixed small scale $\delta$:

\begin{corollary} \label{cor:discretized}
	Given $t\in (0,1]$, $\delta_0>0$, $\zeta>0$, $C\ge 1$, there are $\delta_1=\delta_1(t,\delta_0,C)>0$, $\e=\e(t,\zeta)>0$, and $c=c(t,\zeta)>0$, such that the following holds for all $0<\delta<\delta_1$. Let $\mu\in\cP([0,1]^2)$ be a measure such that
	\[
		\mu(B_r) \le C\, r^t, \quad r\in [\delta_2,1]
	\]
	for some $\delta_2<\delta_1$, and
	\[
		\mu(\ell^{(\delta_0)})\leq c \quad\text{for all lines }\ell\in\mathbb{A}(\R^2,1).
	\]
	Then there are a point $y\in X$ and a set $Z\subset X$ with $\mu(Z)>1/2$ such that
	\[
		\Delta^y( \mu_Z)(B_r) \le r^{\phi(t)-\zeta},\quad \delta\in [\delta_2,\delta_1],
	\]
	and in particular
	\[
		|\Delta^y(\supp(\mu))|_{\delta} \gtrsim \delta^{-\phi(t)+\zeta}, \quad \delta\in [\delta_2,\delta_1].
	\]

	Furthermore, there exists $\eta=\eta(t,\zeta)>0$, such that if $\mu$ additionally satisfies
	\[
		\mu(Q) \ge 2^{-m(t+\eta)}, \quad\text{whenever } Q\in\cD_m(\mu), 2^{-m}\ge \delta,
	\]
	then we can improve the above to
	\[
		\Delta^y( \mu_Z)(B_r) \le r^{t-\zeta},  |\Delta^y((\supp(\mu))|_{\delta} \gtrsim \delta^{-t+\zeta} \quad\delta\in [\delta_2,\delta_1].
	\]

	Finally, the same claims hold with $\Delta^y$ replaced by $\pi_y$.
\end{corollary}
\begin{proof}
	All the claims are shown in the course of the proofs of Theorems~\ref{thm:planar-AD} and \ref{thm:planar-general}. The key point is that in the starting point for the bootstrapping, Proposition \ref{prop:Orponen-radial}, the value of $K$ depends only on the value of $\delta_0$ such that the mass of $\delta_0$-tubes becomes sufficiently small, the Frostman constants of $\mu$ and $\nu$ (which in the context of this corollary are merely restrictions of the given measure $\mu$ to separated sets), and the $\e$ in the proposition, which eventually needs to be taken small enough in terms of our $\zeta$, that controls the number of bootstrapping steps. Each step in the bootstrapping (Lemmas~\ref{lem:kaufman-bootstrap} and \ref{lem:bootstrapping-step-general}) are then also quantitative in terms of the parameters of the previous steps. If (unlike the settings of Theorems~\ref{thm:planar-AD} and \ref{thm:planar-general}) we only have the Frostman (or lower Frostman) condition down to some scale $\delta_2$, then the conclusions still follow for scales $\ge \delta_2$. As a final remark, what the proofs actually show is that $\Delta^y(\mu_Z)$ is $(\delta,\phi(t)-\zeta,\delta^{\e'})$-robust for some $\e'=\e'(t,\zeta)>0$ (with $t$ in place of $\phi(t)$ in the roughly Ahlfors-regular case), but thanks to Lemma \ref{lem:robust-to-Frostman} this implies the claimed Frostman condition after reducing the set $Z$ slightly (and making $\delta_1$ even smaller in terms of $t,\zeta$ only).
\end{proof}


\section{Radial projections in higher dimensions}
\label{sec:high-dim-radial}

\subsection{Sliced measures for orthogonal projections}

We recall some facts about projections and slices of measures; see \cite[Chapters 9 and 10]{Mattila95} or \cite[Chapters 5 and 6]{Mattila15} for details. We denote by $\gamma_{d,k}$ the unique probability measure on $\mathbb{G}(\R^d,k)$ invariant under the action of the orthogonal group. Each affine plane in $\mathbb{A}(\R^d,k)$ can be written in a unique way as $V+a$ for $V\in\mathbb{G}(\R^d,k)$ and $a\in V^\perp$.  We denote by $\lambda_{d,k}$ the measure on $\mathbb{A}(\R^d,k)$ given by
\[
	\lambda_{d,k}(B) = (\gamma_{d,k}\times\mathcal{H}^{d-k})\{ (V,a): a\in V^\perp, V+a\in B  \}.
\]
See \cite[Chapter 3]{Mattila95} for further details on these measures.

The following is the Marstrand-Mattila projection theorem, see \cite[Theorem 9.7]{Mattila95}
\begin{theorem} \label{thm:projection}
	If $\mu\in\cM(\R^d)$ satisfies $\cE_s(\mu)<\infty$, then $P_V\mu\ll \mathcal{H}^k$ for $\gamma_{d,k}$-almost all $V$ if $s\ge k$, and $\cE_s(P_V\mu)<\infty$ for $\gamma_{d,k}$-almost all $V$ if $s<k$.
\end{theorem}

We recall some facts on sliced measures on affine planes; see \cite[\S 10.1]{Mattila95} or \cite[\S 6.1]{Mattila15} for proofs and further details. Let $\mu\in\cM(\R^d)$. For each fixed $V\in\mathbb{G}(\R^d,k)$ and $a\in V^\perp$ we can define the sliced measures
\begin{equation}\label{eq: sliced}
	\mu_{V,a}  =\lim_{r\to 0} \frac{\mu|_{P_{V^\perp}^{-1}(B(a,r))}}{\mathcal{H}^{d-k}(B(a,r))} = \lim_{r\to 0} c_{d-k}^{-1} r^{k-d}\mu|_{V_a^{(r)}} ,
\end{equation}
where $c_m$ is the measure of the unit ball in $\R^m$ and the limit is in the weak topology.  The sliced measures are well defined, finite Radon measures (and possibly trivial) for $\cH^{d-k}$ almost all $a$. Moreover, the map $(V,a)\mapsto \mu_{V,a}$ is Borel. Since $\mu$ is compactly supported, so are the $\mu_{V,a}$. Since $\nu\mapsto |\nu|$ is continuous on the set of compactly supported measures, the map defined by $(V,a)\mapsto \mu_{V,a}/|\mu_{V,a}|$ if $|\mu_{V,a}|>0$, and $(V,a)\mapsto 0$ otherwise, is also Borel.  We will use these facts repeatedly without further reference.

If $P_{V^\perp}\mu\ll \mathcal{H}^{d-k}$ (which, by Theorem \ref{thm:projection}, is the case for $\gamma_{d,k}$-almost all $V$ if $s>d-k$), then
\begin{equation} \label{eq:cond-measure-decomposition}
	\mu = \int_{V^\perp} \mu_{V,a}\, d\mathcal{H}^{d-k}(a),
\end{equation}
This decomposition also characterizes the sliced measures up to a set of $\mathcal{H}^{d-k}$-zero measure. Note that $P_{V^\perp}\mu(a)=|\mu_{V,a}|$ for $\mathcal{H}^{d-k}$-almost all $a$, which can be seen by evaluating \eqref{eq:cond-measure-decomposition} at $P_{V^\perp}^{-1}(B(a_0,r))$ and letting $r\to 0$. For the affine plane $W=V+a$, we sometimes write $\mu_W$ instead of $\mu_{V,a}$. This is a slight abuse of notation since $\mu_W$ also denotes the normalized restriction to $W$; in this section, whenever $W$ is a plane, we interpret $\mu_W$ as the sliced measure.

We recall the Marstrand-Mattila slicing theorem, relating the energy of a measure to that of its slices; see \cite[Theorem 10.7]{Mattila95}.
\begin{theorem} \label{thm:slicing}
	If $\mu\in\cM(\R^d)$ satisfies $\cE_s(\mu)<\infty$ for some $s> d-k$, then
	\[
		\int_{\mathbb{G}(\R^d,k)} \int_{V^\perp} \cE_{s-(d-k)}(\mu_{V,a}) \, d\mathcal{H}^{d-k}(a)d\gamma_{d,k}(V) \lesssim_d \cE_s(\mu).
	\]
	In particular, $\cE_{s-(d-k)}(\mu_W)<\infty$ for $\lambda_{d,k}$-almost all $W$.
\end{theorem}

Given $x\in\R^d$, we often write $\mu_{V,x}=\mu_{V,P_{V^\perp}x}$. The following is a simple consequence of Theorem \ref{thm:projection} and the definition of $\lambda_{d,k}$.
\begin{lemma} \label{lem:abs-cont-on-cond-meas}
	Let $1\le k\le d-1$, and let $\mathcal{B}$ be a Borel subset of $\mathbb{A}(\R^d,k)$.

	\begin{enumerate}
		\item If $V\in\mathbb{G}(\R^d,k)$ is such that $P_{V^\perp}\mu\ll \mathcal{H}^{d-k}$, then
		      \[
			      \mathcal{H}^{d-k}\{a\in V^\perp: V+a\in\mathcal{B} \}=0 \,\Longrightarrow\, \mu\{x: V+x\in\mathcal{B}\}=0.
		      \]
		\item If $\cE_{d-k}(\mu)<\infty$, then
		      \[
			      \lambda_{d,k}(\mathcal{B})=0\,\Longrightarrow\, (\gamma_{d,k}\times\mu)\{(V,x):V+x\in\mathcal{B}\}=0.
		      \]
	\end{enumerate}
\end{lemma}
\begin{proof}
	For the first claim, note that
	\[
		\mu\{ x: V+x\in\mathcal{B}\} = P_{V^\perp}\mu\{ a\in V^\perp : V+ a\in\mathcal{B}\}.
	\]
	The second claim then follows from Theorem \ref{thm:projection}, the first claim, and the definition of $\lambda_{d,k}$.
\end{proof}

\subsection{Sliced measures for cylindrical projections}

We will also require a decomposition into sliced measures for cylindrical projections. Let $2\le k\le d$ and fix $V\in\mathbb{G}(\R^d,k)$, $x\in V$. We define the associated cylindrical projection as
\[
	\Pi_{V,x}(y) = \pi_x P_V y.
\]
Here $\pi_x$ is the radial projection on $V$. See Figure \ref{fig: cylindrical}. This is well defined whenever $x\neq P_V y$. We denote the unit sphere in a linear space $V\in\mathbb{G}(\R^d,k)$ by $S^{k-1}_V$, and the corresponding Hausdorff measure by $\sigma^{k-1}_V$. The range of $\Pi_{V,x}$ is $S^{k-1}_V$. If $k=d$, this is just the usual radial projection. For general $x\in\R^d$ such that $P_Vx\neq P_V y$ we also write
\[
	\Pi_{V,x}(y) = \pi_{P_V x} P_V y = \Pi_{V,P_Vx}(y).
\]

\begin{figure}
	\begin{tikzpicture}[ scale=0.5]
		\draw[black,dashed] (-6.75,-4)--(3.25,1);
		\draw[red] (-12, -4) -- (4, -4) -- (12, 1)--(-4, 1)--cycle;
		\draw[red] (5, -4) node[right]{$V$};
		\draw[black] (3.25,1)--(3.25,6)--(-6.75,1)--(-6.75,-9)--(3.25,-4)--cycle;
		\draw[blue] (3.25, 6.2) node[above]{$\Pi_{V,x}^{-1}(\theta)$};
		\draw[red, fill=red] (-1.75,-1.5) circle[radius=0.1] node[above left]{$x$};
		\draw[black, fill=black] (0.3,-0.45) circle[radius=0.1] node[above right]{$\theta$};
		\draw[blue,dashed] (-1.75, -1.5) ellipse(3cm and 1.5cm);
	\end{tikzpicture}
	\caption{Cylindrical projections}
	\label{fig: cylindrical}
\end{figure}
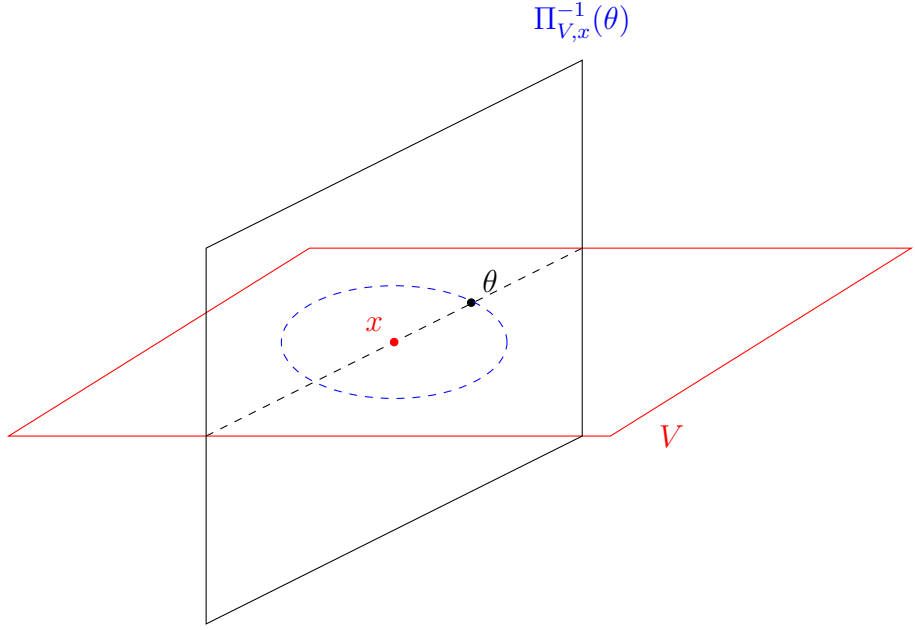

Just as for orthogonal projections, one can decompose measures along fibers of cylindrical projections. The next proposition is proved in the same way as for orthogonal projections, we sketch the argument for completeness.
\begin{prop} \label{prop:cylindrical-slice}
	Let $2\le k\le d$ and fix $V\in\mathbb{G}(\R^d,k)$, $x\in V$. Let $\nu\in\cM(\R^d)$ with $x\notin\supp(P_V\nu)$. Then for $\sigma_V^{k-1}$-almost all $\theta\in S_V^{k-1}$ the sliced measure
	\begin{equation}\label{eq: cylindrical}
		\nu_{V,x,\theta}^{\text{cyl}} = \lim_{r\to 0} \frac{\nu|_{\Pi_{V,x}^{-1}B(\theta,r)}}{\sigma_V^{k-1}(B(\theta,r))}
	\end{equation}
	is well defined. If $\Pi_{V,x}\nu\ll \sigma_V^{k-1}$, then also
	\begin{equation} \label{eq:integral-repr-cylindrical-cond}
		\nu = \int \nu_{V,x,\theta}^{\text{cyl}}\, d\sigma_V^{k-1}(\theta).
	\end{equation}
\end{prop}
\begin{proof}
	Let $U$ be a neighborhood of $\supp(\nu)$ separated from $x$, and write $C^+(\nu)$ for the space of non-negative continuous functions on $\supp(\nu)$. Fix $\varphi\in C^+(\nu)$. Applying \cite[Theorem 2.12(1)]{Mattila95} to the measures $\Pi_{V,x}(\varphi d\nu)$ and $\sigma^{k-1}_V$, we get that the limit
	\[
		\lim_{r\to 0} \frac{\int_{\Pi_{V,x}^{-1}(B(\theta,r))} \varphi \, d\nu}{\sigma_V^{k-1}(B(\theta,r))}
	\]
	exists (and is finite) for $\sigma_V^{k-1}$-almost all $\theta$. By the separability of $C^+(\nu)$, one can eliminate the dependence on $\varphi$.
	The existence of $\nu_{V,x,\theta}^{\text{cyl}}$ for $\sigma_V^{k-1}$-almost all $\theta$ follows from the Riesz representation theorem. Under the assumption $\Pi_{V,x}\nu\ll \sigma_V^{k-1}$, the integral representation \eqref{eq:integral-repr-cylindrical-cond} follows by applying \cite[Theorem 2.12(2)]{Mattila95} to $\Pi_{V,x}(\varphi d\nu)$ and $\sigma_V^{k-1}$ for $\varphi\in C^+(\nu)$.
\end{proof}
Note that $\nu_{V,x,\theta}^{\text{cyl}}$ is supported on
\[
	\Pi_{V,x}^{-1}(\theta) = P_V^{-1}(x+\langle \theta\rangle)\setminus P_V^{-1}(x)
\]
where $\langle \theta \rangle $ is the linear subspace generated by $\theta$. We will slightly abuse notation and consider $\Pi_{V,x}^{-1}(\theta)$ as a plane, by adjoining $P_V^{-1}(x)$ to it. Under this interpretation, $x\in \Pi_{V,x}^{-1}(\theta)$, and
\[
	\Pi_{V,x}^{-1}(\theta) = \Pi_{V,0}^{-1}(\theta)+x = \langle\theta\rangle + V^{\perp} + x.
\]
See Figure \ref{fig: cylindrical}. The following lemma is a simple geometric observation from the definitions, but will be very useful to us.
\begin{lemma} \label{lem:slices-mut-abs-cont}
	Let $2\le k\le d$. Fix $V\in \mathbb{G}(\R^d,k)$ and $x\in V$.

	Suppose $\nu\in\cM(B^d)$ with $x\notin P_V(\supp\nu)$. If both sliced measures $\nu_{ \Pi_{V,x}^{-1}(\theta)}$  and $\nu_{V,x,\theta}^{\text{cyl}}$ are defined, then they are mutually absolutely continuous with the densities bounded by a constant depending on $\dist(x,P_V(\supp\nu))$ only.
\end{lemma}
\begin{proof}
	Note that $\Pi_{V,x}^{-1}(B(\theta,r)) = P_V^{-1}(C(x,\theta,r))$, where $C(x,\theta,r)$ denotes the (spherical) cone in $V$ centered at $x$ with central direction $\theta$ and opening $r$. Since $\supp(\nu)$ is supported on the unit ball,  $C(x,\theta,r)\cap B^k\subset (x+\langle\theta\rangle)^{(r)}$, and therefore
	\[
		\Pi_{V,x}^{-1}(B(\theta,r)) \cap \supp(\nu) \subset (\Pi_{V,x}^{-1}(\theta))^{(r)}.
	\]
	Likewise, since $x$ is separated from $P_V(\supp\nu)$, there is a constant $C$ such that
	\[
		(\Pi_{V,x}^{-1}(\theta))^{(r)} \cap \supp(\nu) \subset \Pi_{V,x}^{-1}(B(\theta,Cr)).
	\]
	The claim follows from the definitions  \eqref{eq: sliced}, \eqref{eq: cylindrical}.
\end{proof}

We conclude this section by combining the previous results with Orponen's radial projection theorem from \cite{Orponen19}. 
\begin{theorem} \label{thm:equivalent-measures}
	Fix $2\le k\le d$, and let $k-1<s<k$. Let $\mu,\nu\in\cP(B^d)$ be such that $\cE_s(\mu)<\infty$, $\cE_s(\nu)<\infty$. Let $\mathcal{U}$ be an open subset of $\mathbb{G}(\R^d,k)$ such that the supports of $P_V\mu$ and $P_V\nu$ are disjoint for all $V\in\overline{\mathcal{U}}$. For $x\in\supp(\mu)$ and $V\in\mathcal{U}$, define
	\[
		\nu^{V,x} = \int_{S_V^{k-1}} \nu_{\Pi_{V,x}^{-1}(\theta)} \, d\sigma_V^{k-1}(\theta).
	\]
	Then for $\gamma_{d,k}$-almost all $V\in\mathcal{U}$ the following holds: for $\mu$-almost all $x$, the measure $\nu^{V,x}$ is well defined  and mutually absolutely continuous to $\nu$. The densities are bounded, with the bound depending only on distance between the supports of $P_V\mu$ and $P_V\nu$. Moreover,
	\begin{equation} \label{eq:Lp-mass-cond-meas}
		\iint_{S_V^{k-1}} |\nu_{\Pi_{V,x}^{-1}\theta}|^p \,d\sigma_V^{k-1}(\theta) d\mu(x)<\infty,
	\end{equation}
	where $p=p(k,s)\in (1,2)$.
\end{theorem}
\begin{proof}
	Applying the second part of Lemma \ref{lem:abs-cont-on-cond-meas} with $d-k+1$ in place of $k$ (which is permissible, since $s>k-1$) and $\mathcal{B}$ equal to the set of planes $W_a$ such that $\nu_{W_a}$ is defined, we get that for $(\gamma_{d,d-k+1}\times\mu)$-almost all $(W,x)$ the sliced measure $\nu_{W+x}$ is defined. If $V$ is selected according to $\gamma_{d,k}$ and $\theta$ is then selected according to $\sigma_V^{k-1}$, then the distribution of $\Pi_{V,0}^{-1}(\theta)=V^{\perp}+\langle\theta\rangle \in \mathbb{G}(\R^d,d-k+1)$ is invariant under the action of the orthogonal group, and therefore equals $\gamma_{d,d-k+1}$. By Fubini, for $\gamma_{d,k}$-almost all $V\in\mathcal{U}$, the sliced measures $\nu_{\Pi_{V,x}^{-1}(\theta)}$ are well defined for $\mu$-almost all $x$ and $\sigma_V^{k-1}$-almost all $\theta$.

	By Theorem \ref{thm:projection}, $\cE_s(P_V\mu)<\infty$, $\cE_s(P_V\nu)<\infty$ for $\gamma_{d,k}$-almost all $V\in\mathcal{U}$.

	We fix $V\in\mathcal{U}$ such that the above almost sure properties hold.

	Using the assumption $s>k-1$, it follows from Orponen's radial projection theorem (\cite[Eq. (3.6)]{Orponen19}) applied to $P_V\mu,P_V\nu$ that $\Pi_{V, x}\nu\ll \sigma_V^{k-1}$ for $\mu$ almost all $x$ and, moreover, there is $p=p(s,k)\in (1,2)$ such that
	\[
		\int \|\Pi_{V,x}\nu\|_{L^p(\sigma_V^{k-1})}^p \,d\mu(x) <\infty.
	\]
	The fact that $\nu^{V,x}\sim\nu$ now follows by combining \eqref{eq:integral-repr-cylindrical-cond} from Proposition \ref{prop:cylindrical-slice} and Lemma \ref{lem:slices-mut-abs-cont}. As an immediate consequence of this,
	\begin{equation} \label{eq:application-Orponen}
		\int \|\Pi_{V,x}\nu^{V,x}\|_{L^p(\sigma_V^{k-1})}^p\, d\mu(x)<\infty.
	\end{equation}
	Unwrapping the definitions, for $\mu$-almost all $x$,
	\[
		\Pi_{V,x}\nu^{V,x}(B(\theta_0,r)) = \int_{B(\theta_0,r)} |\nu_{\Pi_{V,x}^{-1}(\theta)}| \,d\sigma_V^{k-1}(\theta),\quad \theta_0\in S_V^{k-1},r>0,
	\]
	and therefore the density $\Pi_{V, x}\nu^{V,x}(\theta_0)$ is given by $ |\nu_{\Pi_{V,x}^{-1}(\theta_0)}|$ for $\sigma_V^{k-1}$ almost all $\theta_0$. Together with \eqref{eq:application-Orponen}, this establishes \eqref{eq:Lp-mass-cond-meas}.
\end{proof}

\subsection{Radial projections from slices and projections}

The next theorem provides the basic mechanism to obtain radial projection estimates by inducting on the ambient dimension.
\begin{theorem} \label{thm:radial-inductive}
	Let $\mu,\nu\in\cP(\R^d)$ be measures with $\cE_s(\mu)<\infty$, $\cE_s(\nu)<\infty$ with $s\in (k,k+1]$ for some $k\in \{1,\ldots,d-2\}$.

	Suppose that there is $t>0$ such that the pair of sliced measures $(\mu_{W+z},\nu_{W+z})$ has $t$-thin tubes for $(\gamma_{d,d-k}\times\mu)$-almost all $(W,z)$ (in particular, for a Borel set of $(W,z)$).

	Then  $(\mu,\nu)$ has $(k+t)$-thin tubes.
\end{theorem}

We can always take $t=0$ in the above theorem to obtain that $(\mu,\nu)$ have $k$-thin tubes. The idea in this case (bounding the measure of tubes by looking at a typical projection) goes back to \cite{DIOWZ21}. A new insight in this paper is that by combining estimates for the measures of projected and of sliced tubes it is possible to obtain stronger estimates.

\begin{proof}[Proof of Theorem \ref{thm:radial-inductive}]
	After a change of coordinates, rescaling, and restricting, we may assume that $\supp(\mu)\subset B(0,1/100)$ and $\supp(\nu)\subset B((1,0,\ldots,0),1/100)$. Hence, if $P$ is the projection onto the plane
	\[
		V=\{ x\in \R^d:x_{k+2}=\cdots=x_d=0\},
	\]
	then $P(\supp(\mu))$ and $P(\supp(\nu))$ are well separated.

	By replacing $s$ with a smaller number, we may assume that $s\in (k,k+1)$. Replacing $V$ by $UV$ for a random perturbation $U\in\mathbb{O}_d$ of the identity, we may assume by Theorem \ref{thm:projection} that $\cE_s(P\mu)<\infty$, $\cE_s(P\nu)<\infty$. We may further assume that the conclusion of Theorem \ref{thm:equivalent-measures} holds for $V$. If $V$ is chosen according to $\gamma_{d,k+1}$ and $\theta$ is then chosen according to $\sigma_V^k$, then the law of $\theta^\perp \cap V\in\mathbb{G}(\R^d,k)$ is invariant under the action of $\mathbb{O}_d$, and thus equals $\gamma_{d,k}$. Since $s>k$, in light of Theorem \ref{thm:projection} by further perturbing $V$ we may also assume that
	\begin{equation} \label{eq:abs-cont-proj-generic-theta}
		P_{\theta^\perp\cap V}\mu \ll \mathcal{H}^k\text{ for $\sigma_V^k$-almost all } \theta.
	\end{equation}
	By the thin-tubes assumption, after a final perturbation of $V$ we may assume that
	\begin{equation} \label{eq:G-full-measure}
		(\sigma^k\times\mu)(\mathcal{G}) = 1,
	\end{equation}
	where
	\begin{equation} \label{eq:G-def}
		\mathcal{G} = \{(\theta,x): (\mu_{\theta,x},\nu_{\theta,x}) \text{ have } t\text{-thin tubes} \}
	\end{equation}
	is a Borel set,  and $\mu_{\theta,x}$, $\nu_{\theta,x}$ denote the sliced measures of $\mu$, $\nu$ on $W_{\theta,x}:=\Pi_{V,x}^{-1}(\theta)= V^\perp + \langle\theta\rangle+x$. (These objects depend on $V$, but this plane is fixed for the rest of the proof). Implicit in the definition of $\mathcal{G}$ is that $\mu_{\theta,x}$ and $\nu_{\theta,x}$ are well defined.

	We have arranged things so that for $\mu$-typical $x$, for $\sigma^k$-almost all $\theta$ the sliced measures $(\mu_{\theta,x},\nu_{\theta,x})$ on the planes $W_{\theta,x}$ have $t$-thin tubes. These planes fiber $\R^d\cap\supp(\nu)$ (for a fixed $x$) so intuitively one can expect that $r$-tubes through $x$ will have $\nu$-measure $\lesssim r^{k+t}$ for $\mu$-almost all $x$. See Figure \ref{fig: slicing}. In the rest of the proof we will verify this rigorously.

	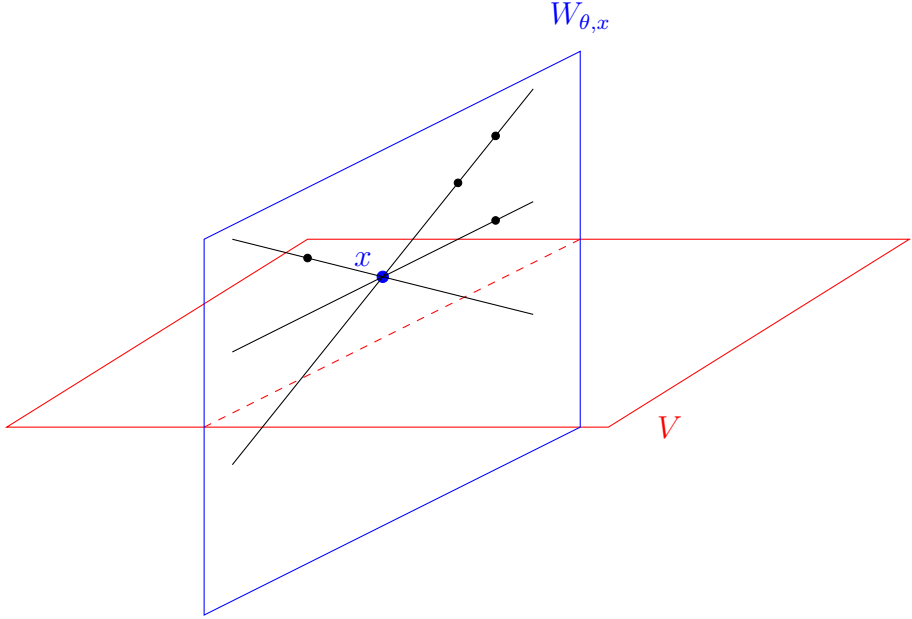
\begin{figure}
		\begin{tikzpicture}[ scale=0.5]
			\draw[red,dashed] (-6.75,-4)--(3.25,1);
			\draw[red] (-12, -4) -- (4, -4) -- (12, 1)--(-4, 1)--cycle;
			\draw[red] (5, -4) node[right]{$V$};
			\draw[blue] (3.25,1)--(3.25,6)--(-6.75,1)--(-6.75,-9)--(3.25,-4)--cycle;
			\draw[blue] (3.25, 6.2) node[above]{$W_{\theta,x}$};
			\draw[blue, fill=blue] (-2,0) circle[radius=0.15] node[above left]{$x$};
			\draw  (2, 2)--(-6, -2);
			\draw[black, fill=black] (1,1.5) circle[radius=0.1];
			\draw  (2, -1)--(-6, 1);
			\draw[black, fill=black] (-4,0.5) circle[radius=0.1];
			\draw[black, fill=black] (1,3.75) circle[radius=0.1];
			\draw[black, fill=black] (0,2.5) circle[radius=0.1];
			\draw   (2, 5) -- (-6, -5);
		\end{tikzpicture}
		\caption{Large radial projections on many planar slices through $x$ imply large radial projections overall}
		\label{fig: slicing}
	\end{figure}

	Since the maps $(x,\theta)\mapsto W_{\theta,x}$ and  $W\mapsto \nu_W$ are Borel, then so is $F: (\theta,x)\mapsto (x,\nu_{\theta,x})$ defined on $S^k\times\supp(\mu)$. Let $B_0$ be a closed ball containing the supports of $\mu$ and $\nu$. By Lemma \ref{lem:Borel}, the set
	\[
		\mathcal{C} = \{ (x,\rho)\in B_0\times \mathcal{P}(B_0): (x,\rho) \text{ have $t$-thin tubes}\}
	\]
	is $\sigma$-compact, in particular Borel. Hence
	\[
		F^{-1}(\mathcal{C}) =\{ (\theta,x)\in S^k\times \supp(\mu): (x,\nu_{\theta,x})\text{ have $t$-thin tubes}\}
	\]
	is a Borel set. We deduce that if we write
	\[
		X_{\theta} = \{ x: (x,\nu_{\theta,x}) \text{ have $t$-thin tubes}\},
	\]
	then $\{ (\theta,x)\in\mathcal{G}: x\in X_\theta\}$ is Borel. Furthermore, if $\nu'_{\theta,x}\le \nu_{\theta,x}$ stands for the corresponding measure in the definition of thin tubes, then $(x,\theta)\mapsto \nu'_{\theta,x}$ defined on $F^{-1}(\mathcal{C})$ can be chosen to be a Borel map, given by $G\circ F$, where $G:(x,\rho)\mapsto \rho_x$ is the Borel selector provided by Lemma \ref{lem:Borel}.

	Since $\mu_{\theta,x}=\mu_{\theta,P_{\theta^\perp\cap V}(x)}$, it follows from \eqref{eq:cond-measure-decomposition} and \eqref{eq:abs-cont-proj-generic-theta} that, for $\sigma^k$-almost all $\theta$,
	\begin{equation} \label{eq:integral-decomp}
		\mu_\theta := \int \mu_{\theta,x} \, d\mu(x)  = \int_{\theta^\perp \cap V} \mu_{\theta,a} \,d P_{\theta^\perp\cap V}\mu(a)  \ll \int_{\theta^\perp \cap V} \mu_{\theta,a}\,d \mathcal{H}^k(a) = \mu.
	\end{equation}

	By the the first part of Theorem \ref{thm:equivalent-measures}, it follows that
	\[
		\int_{S^k}\int |\nu_{\theta,x}| \,d\mu(x)d\sigma^k(\theta) \gtrsim 1.
	\]
	On the other hand, we deduce from \eqref{eq:G-def}, \eqref{eq:G-full-measure} and the measurability assumption that if $c,K$ are taken respectively sufficiently small and large, then there is a compact set $\mathcal{G}'\subset S^k\times \R^d$  such that for all $(\theta,x)\in\mathcal{G}'$ it holds that
	\begin{equation} \label{eq:slice-thin-tubes}
		\mu_{\theta,x}\{ y: (y,\nu_{\theta,x})  \text{ has } (t,K,c)\text{-thin tubes}\} \ge c,
	\end{equation}
	and
	\[
		\int_{S^k}\int \mathbf{1}_{\mathcal{G}'}(\theta,x)|\nu_{\theta,x}| \,d\mu(x)d\sigma^k(\theta) \gtrsim 1.
	\]

	Also, it follows from \eqref{eq:Lp-mass-cond-meas} that
	\[
		\int_{S^k}\int \mathbf{1}_{|\nu_{\theta,x}|\ge M}|\nu_{\theta,x}| \,d\mu(x)d\sigma^k(\theta) \lesssim M^{1-p}.
	\]
	Thus, there is $M\lesssim_p 1$ such that
	\[
		(\sigma^k\times \mu)\{ (\theta,x)\in\mathcal{G}': |\nu_{\theta,x}|\in [1/M,M] \}\gtrsim 1/M \gtrsim 1.
	\]
	Since $\nu_{\theta,x}=\nu_{\theta,y}$ for $y\in\supp(\mu_{\theta,x})$, we deduce that there is a compact set $\Theta$ with $\sigma^k(\Theta)\gtrsim 1$, such that if $\theta\in \Theta$, then
	\[
		\mu\big\{ x: \mu_{\theta,x}\{ y: (y,\nu_{\theta,y})  \text{ has } (t,K,c)\text{-thin tubes}, |\nu_{\theta,y}|\in [ 1/M, M] \} \ge c \big\}\gtrsim 1.
	\]
	In light of \eqref{eq:integral-decomp}, after replacing $\Theta$ by a suitable compact subset,
	\[
		\mu\{ y:  (y,\nu_{\theta,y})  \text{ has } (t,K,c)\text{-thin tubes}, |\nu_{\theta,y}|\in [ 1/M, M] \} \gtrsim_\theta 1, \quad\theta\in\Theta.
	\]
	The set
	\[
		\mathcal{G}'' =\big\{ (\theta,x)\in\mathcal{G'}: \theta\in\Theta,   (x,\nu_{\theta,x})  \text{ has } (t,K,c)\text{-thin tubes}, |\nu_{\theta,x}|\in [ 1/M, M]  \big\}
	\]
	is checked to be Borel and, as we have seen, $(\sigma^{k}\times\mu)(\mathcal{G}'')\gtrsim 1$. By replacing $\mathcal{G}''$ with a subset, we assume that $\mathcal{G}''$ is compact.
	Let
	\[
		\Theta_x=\{\theta\in S^k: (\theta,x)\in\mathcal{G}''\}.
	\]
	Then $\Theta_x$ is compact for all $x$, and a new application of Fubini yields a compact set $X\subset\R^d$ and a number $c_0>0$ such that $\mu(X)\ge c_0$ and $\sigma^k(\Theta_x) \ge c_0$ for each $x\in X$. Let
	\[
		\nu^x = \int_{S^k} \nu_{\theta,x}\, d\sigma^k(\theta).
	\]
	We know from Theorem \ref{thm:equivalent-measures} that $\nu^x\sim_C \nu$ for some $C>0$ and $\mu$-almost all $x$ (here $\sim_C$ denotes mutual absolute continuity with densities bounded by $C$). Hence by passing to a further compact subset of $X$, we may assume that $\nu^x\sim_C \nu$ for all $x\in X$.

	Define
	\[
		\nu'_x = C^{-1}\int_{\Theta_x} \nu'_{\theta,x}\, d\sigma^{k}(\theta).
	\]
	Since $(x,\theta)\mapsto \nu'_{\theta,x}$ is a Borel map, this is a well defined Borel measure. By construction, $\nu'_x\le C^{-1}\nu^x\le \nu$. Also,
	\[
		|\nu'_x| \ge C^{-1}\sigma^{k}(\Theta_x) \inf_{\theta\in\Theta_x}|\nu'_{\theta,x}| \ge C^{-1} c c_0 \gtrsim 1.
	\]
	Finally, fix $x\in X$ and a small ball $B(\theta_0,r)\subset S^{d-1}$. Since $P(\supp\mu)$ and $P(\supp\nu)$ are well separated, if $\pi_x^{-1}(B(\theta_0,r))\cap\supp\nu\neq\varnothing$, then $P\theta_0$ (view $\theta_0$ as a unit vector in $\mathbb{R}^d$) is bounded away from $0$, and therefore
	\[
		\{ \theta\in S^k :W_{x,\theta} \cap \pi_x^{-1}(B(\theta_0,r)) \neq \emptyset \}
	\]
	is contained in a ball of radius $\lesssim r$. We conclude that
	\begin{align*}
		\nu'_x(\pi_x^{-1}(B_r)) & \le \int_{B_{O(r)}\cap \Theta_x} \nu'_{\theta,x} (\pi_x^{-1}(B_r))\, d\sigma^k(\theta) \\
		                        & \lesssim r^t \int_{B_{O(r)}\cap \Theta_x} |\nu_{\theta,x}|\, d\sigma^k(\theta)         \\
		                        & \lesssim r^{k+t}.
	\end{align*}
\end{proof}

\subsection{Ruling out sliced measures concentrating on proper subspaces}

In order to be able to use Theorem \ref{thm:radial-inductive} effectively, we need to show that slices have $t$-thin tubes for some reasonably large (in particular, non-zero) $t$. We are especially interested in the case where we are slicing $2$-planes. Then we cannot possibly have $t$-thin tubes, $t>0$ if almost all (pairs of) sliced measures are supported on a line. On the other hand, if we know the slices are \emph{not} supported on a line, then we can call upon the results of \cite{Orponen19} and Proposition \ref{prop:planar-thin-tubes} to obtain an improved thin tubes estimate. We suspect that if $\nu\in\cP(\R^d)$ is $s$-Frostman with $s>d-2$ and $\nu$ is not supported on a hyperplane, then $\nu_W$ cannot be supported on a line for $\lambda_{d,2}$-positively many $W$, but are only able to prove it for $s$ close to $d-1$:

\begin{prop} \label{prop:no-lines-slices}
	Let $\mu,\nu\in\cP(\R^d)$ be measures with disjoint supports such that  $\cE_s(\mu), \cE_s(\nu)<\infty$ for some
	\[
		s\in (d-1-1/(d-1)-\eta,d-1],
	\]
	where $\eta=\eta(d,s)>0$ is a sufficiently small constant. Assume that $\mu(H)=0$, $\nu(H)=0$ for all hyperplanes $H\in\mathbb{A}(\R^d,d-1)$. Assume also that $\hdim(\supp(\nu))< s+\eta$. Then, after restricting $\mu$ and $\nu$ to subsets of positive measure if needed,
	\[
		(\gamma_{d,2}\times \mu)\{(V,x): \mu_{V,x}(\ell)\nu_{V,x}(\ell)=|\mu_{V,x}||\nu_{V,x}|>0 \text{ for some }\ell\in\mathbb{A}(V+x,1) \} = 0.
	\]
\end{prop}

Note that, since $s>d-2$, the conditional measures $\mu_{V,x}$ are indeed well defined for $(\gamma_{d,2}\times \mu)$-almost all $(V,x)$ by Lemma \ref{lem:abs-cont-on-cond-meas}. The proposition depends on a radial projection estimate that can be seen as a more refined version of \cite[Theorem 6.15]{Shmerkin23}.
\begin{prop} \label{prop:orponen-refined}
	Given $s\in (d-2,d)$ there is $\eta>0$ such that the following holds.

	Let $\mu,\nu\in\cP(\R^d)$ be measures with $\cE_s(\mu)<\infty$, $\cE_s(\nu)<\infty$ and separated supports. Suppose that $\mu(H)=0$, $\nu(H)=0$ for each hyperplane $H\in\mathbb{A}(\R^d,d-1)$.  Then for $\mu$-almost all $x$, for all sets $Y$ of positive $\nu$-measure,
	\begin{equation} \label{eq:hausdorff-dim-meas-radial}
		\hdim(\pi_x Y) \ge \tfrac{d-1}{d}s+\eta.
	\end{equation}
\end{prop}

For the proof of the proposition, we require the following higher rank version of \cite[Theorem 5.2]{Shmerkin23}. For simplicity, we state it only for the case of projections of co-rank $1$, which is enough for our application.
\begin{theorem} \label{thm:nonlinear-Bourgain}
	Fix $d\ge 2$. Given $\kappa>0, 0<\alpha<d$ there is $\eta=\eta_d(\kappa,\alpha)>0$ (that can be taken continuous in $\kappa,\alpha$) such that the following holds.

	Let $\mathcal{F}=(Y,\nu,U,F)$ be a regular family of projections $F_y:\R^d\supset U\to\R^{d+1}$ (recall Definition \ref{def:regular-family}).
	Suppose $\mu\in\cP(U)$ satisfies
	\[
		\mu(B_r) \le C\, r^\alpha,
	\]
	and for  $x$ in a set $X$ with $\mu(X)\ge 1-\e$, there are a set $Y_x$, depending in a Borel manner on $x$, with $\nu(Y_x)\ge 1-\e$, and a number $C>0$ such that
	\begin{equation} \label{eq:assumption-nonlinear-Bourgain}
		V_x\nu_{Y_x}\left(\{ H\in \mathbb{G}(\R^d,d-1): |\langle n_H,\theta\rangle|<r\}\right)   \le C \, r^\kappa \quad\text{for all }\theta\in S^{d-1}, r\in (0,1],
	\end{equation}
	where $n_H$ is a unit normal to $H$ and $V_x$ is the function from \eqref{eq:def-Vx}. Assume also that $\{ (x,y):x\in X,y\in Y_x\}$ is compact.

	Then there are $\e'=o_{\e\to 0}(1)$ and  a set $Y_0$ with $\nu(Y_0)\ge 1-\e'$, such that the following holds: for all $y\in Y_0$ there is a set $X_y$ with $\mu(X_y)\ge 1-\e'$ such that, for all sufficiently large $m$, the measure $F_y (\mu_{X_y})$ is $(2^{-m}, \frac{(d-1)\alpha}{d} +\eta, 2^{-\eta m})$-robust.
\end{theorem}

The proof of this theorem is nearly identical to that of \cite[Theorem 5.2]{Shmerkin23}, with the only difference coming in the application of the discretized projection theorem at each scale. While \cite[Theorem 5.1]{Shmerkin23} relies on Bourgain's rank-$1$ discretized projection theorem, for Theorem \ref{thm:nonlinear-Bourgain} we need to appeal to the higher rank version developed by W.~He \cite{He20}; see \cite[{\S 6.4}]{Shmerkin23} for further details on the minor differences in the argument when passing from the rank $1$ to the higher rank situation. The theorem can also be deduced from Theorem \ref{thm:nonlinear-abstract}. The required adaptedness follows from the already mentioned theorem of He, while the estimation of the combinatorial parameter is a consequence of \cite[Proposition 4.7]{Shmerkin23}. 

\begin{proof}[Proof of Proposition \ref{prop:orponen-refined}]
	Thanks to Lemma \ref{lem:energy-Frostman}, by passing to subsets of nearly full measure we may assume that $\mu(B_r)\lesssim r^s$ and $\nu(B_r)\lesssim r^s$ for all $r\in (0,1]$. Furthermore, it is enough to show that if $\e>0$, then there is a set $X$ with $\mu(X)>1-\e$ such that if $x\in X$ then
	\[
		\hdim(\pi_x Y) \ge  \tfrac{d-1}{d}s+\eta \quad\text{for all $Y$ with $\nu(Y)\ge \e$}.
	\]
	In turn, by Lemma \ref{lem:robust-to-Frostman}, in order to establish this it is enough to show that there is a set $X$ with $\mu(X)>1-\e$ and for each $x\in X$ there is a set $Y_x$ with $\nu(Y_x)>1-\e/2$, such that $\pi_x \mu_{Y_x}$ is $(\delta, \tfrac{d-1}{d}s+\eta, \delta^\eta)$-robust for all sufficiently small $\delta\le\delta_0(\mu,\nu,\e)$. Note that this is only slightly weaker than saying that $(\mu,\nu)$ have $(\tfrac{d-1}{d}s+\eta,K,1-\e/2)$-thin tubes; in fact by Lemma \ref{lem:robust-to-Frostman} both statements are equivalent, but we don't need to use this.



	By \cite[Theorem B.1]{Shmerkin23} applied with $k=d-1$ and the assumption $s>d-2$, there are $\eta_1=\eta_1(s)>0$ and $r_0=r_0(\mu,\nu,\e)>0$ such that for all $y$ in a set $Y$ of $\nu$-measure $\ge 1-\e$, there is a set $X_y$ with $\mu(X_y) \ge 1-\e$, and such that
	\begin{equation} \label{eq:radial-eta-estimate}
		\pi_y(\nu_{X_y})(H^{(r)}) \le r^{\eta_1} \quad r\in (0,r_0], H\in \mathbb{G}(\R^d,d-1).
	\end{equation}
	By considering finitely many charts (and suitable rotations), we may replace $\pi_y$ by  $\wt{\pi}_y(x)=(y_i-x_i/(y_d-x_d))_{i=1}^{d-1}$ in the statement of the theorem. A direct calculation shows that for the family of projections $\{ \pi_x\}_{x\in\supp(\mu)}$, the function $V_y$ is given by $V_y(x)=\pi_y(x)^\perp$, so that $\pi_y(x)$ is a unit normal to $V_y(x)$. Since
	\[
		|\langle \pi_y(x), \theta\rangle | < r \Longleftrightarrow \pi_y(x)\in (\theta^{\perp})^{(r)},
	\]
	we see that \eqref{eq:assumption-nonlinear-Bourgain} holds thanks  to \eqref{eq:radial-eta-estimate}. Applying Theorem \ref{thm:nonlinear-Bourgain}, we conclude that there exists $\eta=\eta(s)>0$ such that for $x$ in a set $X_0$ of $\mu$-measure $\ge 1-\e'$ there is a set $Y_x$ with $\mu(Y_x)\ge 1-\e'$ such that the measure $\pi_y\mu_{X_y}$ is $(2^{-m},\tfrac{d-1}{d}s+\eta,2^{-\eta m})$-robust for all sufficiently large $m$. Since $\e'\to 0$ as $\e\to 0$, this completes the proof.
\end{proof}

We can now complete the proof of Proposition \ref{prop:no-lines-slices}.
\begin{proof}[Proof of Proposition \ref{prop:no-lines-slices}]
	Suppose the conclusion does not hold. Let $\mathcal{B}$ be the set of affine planes $W\in\mathbb{A}(d,2)$ such that $\mu_W$ and $\nu_W$ are non-trivial and supported on a single line $\ell_W$. Then
	\begin{equation} \label{eq:counter-assumption-lines-slices}
		(\gamma_{d,2}\times\mu)\{(V,x):V+x\in \mathcal{B} \}>0.
	\end{equation}
	We omit the verification that $\mathcal{B}$ and the map $W\mapsto \ell_W$ on $\mathcal{B}$ are Borel. By passing to a subset, we may assume that in fact $\mathcal{B}$ is compact and $W\mapsto \ell_{W}$ is continuous on $\mathcal{B}$.

	By Theorem \ref{thm:slicing} and the second part of Lemma \ref{lem:abs-cont-on-cond-meas}, we may also assume that $I_{s-d+2}(\nu_{V,x})<\infty$ whenever $V+x\in\mathcal{B}$ and therefore, since $\nu_{V,x}$ is supported inside $\supp(\nu)\cap \ell_{V+x}$,
	\begin{equation} \label{eq:large-dim-on-lines}
		\hdim(\supp(\nu)\cap \ell_{V+x}) \ge s-d+2.
	\end{equation}

	Arguing as in the proof of Theorem \ref{thm:radial-inductive}, if $H$ is distributed according to $\gamma_{d,d-1}$ and $\theta$ is distributed according to $\sigma_H^{d-2}$, then $H\cap\langle\theta\rangle^\perp$ has distribution $\gamma_{d,d-2}$, and $H^\perp+\langle\theta\rangle$ has distribution $\gamma_{d,2}$. In particular, by the counter-assumption \eqref{eq:counter-assumption-lines-slices} and Fubini,
	\[
		(\gamma_{d,d-1}\times\mu)(\mathcal{C})>0,
	\]
	where
	\[
		\mathcal{C}=\big\{ (H,x): \sigma_H^{d-2}\{ \theta: H^\perp+\langle \theta\rangle+x\in\mathcal{B} \}>0\big\}.
	\]
	Let $(H_0,x_0)$ be a density point of $\mathcal{C}$. The measure $\nu$ cannot be supported on the line perpendicular to $H_0$ through $x_0$, since it has finite $s$-energy for some $s>d-2\ge 1$. Hence we can pick a density point $y_0$ of $\nu$ such that $P_{H_0}(x_0)\neq P_{H_0}(y_0)$. Hence if $r$ is small enough, the projections of $B(x_0,r)$ and $B(y_0,r)$ to $H$ are disjoint for any $H\in B(H_0,r)$. We restrict $\mu,\nu$ to these balls without changing notation. Since $(H_0,x_0)$ is a density point of $\mathcal{C}$, by Fubini there is a set of $\gamma_{d,d-1}$-positive measure of planes $H\in B(H_0,r)$ such that
	\begin{equation} \label{eq:counterassumption}
		(\sigma_H^{d-2}\times\mu)\{ (\theta,x): H^\perp+\langle\theta\rangle+x \in\mathcal{B} \} > 0.
	\end{equation}
	For $\gamma_{d,d-1}$-almost all $H$, Theorem \ref{thm:projection} also yields
	\begin{equation} \label{eq:abs-cont-proj-generic-theta-2}
		P_{\theta^\perp\cap H}\mu \ll \mathcal{H}^{d-2}\text{ for $\sigma_H^{d-2}$-almost all } \theta.
	\end{equation}
	Finally, by our choice of neighborhoods, Theorem \ref{thm:equivalent-measures} yields that for $\gamma_{d,d-1}$-almost all $H\in B(H_0,r)$,
	\begin{equation} \label{eq:equivalent-measures}
		\text{For $\mu$-almost all $x$}, \quad \nu \sim \nu^{H,x} := \int_{S_H^{d-2}} \nu_{H^\perp+\langle\theta\rangle+x} \, d\sigma_H^{d-2}(\theta),
	\end{equation}
	where $\sim$ denotes mutual absolute continuity.

	Fix any $H$ such that \eqref{eq:counterassumption}--\eqref{eq:equivalent-measures} hold for the rest of the proof. Now that $H$ is fixed, we supress it from the notation, and write $V_{\theta,x}=H^\perp+\langle\theta\rangle+x$, let $\mu_{\theta,x},\nu_{\theta,x}$ be the sliced measures on $V_{\theta,x}$, and also denote $\ell_{\theta,x}=\ell_{V_{\theta,x}}$.

	Let $\mu_\theta$ be as in \eqref{eq:integral-decomp}; then, thanks to \eqref{eq:abs-cont-proj-generic-theta-2} as in the proof of Theorem \ref{thm:radial-inductive}, $\mu_\theta\ll \mu$ for $\sigma_H^{d-2}$-almost all $\theta$. Write
	\[
		L_\theta=\bigcup\{ \ell_{\theta,x}: V_{\theta,x}\in\mathcal{B}\}.
	\]
	These sets (intersected with a large closed ball) are compact because $W\mapsto \ell_W$ is continuous on the compact set $\mathcal{B}$. By \eqref{eq:counterassumption} and Fubini, $\mu_\theta(L_\theta)>0$ for $\theta$ in a set $\Theta\subset S^{d-2}$ of positive $\sigma^{d-2}$-measure. Since $\mu_\theta\ll \mu$, we have $\mu(L_\theta)>0$ for $\theta\in \Theta$. Note that if $x\in L_\theta$, then there is $x'$ such that $x\in\ell_{\theta,x'}\subset V_{\theta,x'}$, and hence $V_{\theta,x}=V_{\theta,x'}$ (since these planes are parallel), and so in fact $x\in\ell_{\theta,x}=\ell_{\theta,x'}$. By Fubini, we can therefore find  a compact set $X$ with $\mu(X)>0$, and compact sets $(\Theta_x)_{x\in X}$ with $\sigma^{d-2}(\Theta_x)>0$, such that
	\[
		x\in X, \theta\in\Theta_x \Longrightarrow  V_{\theta,x}\in\mathcal{B} \text{ and } x\in \ell_{\theta,x}.
	\]
	Fix $x\in X$, and let
	\[
		Z_x = \bigcup_{\theta\in\Theta_x} \ell_{\theta,x}.
	\]
	By \eqref{eq:equivalent-measures}, making $X$ slightly smaller we can ensure that for each $x\in X$ the measures
	\[
		\nu^x = \int \nu_{\theta,x}\, d\sigma^{d-2}(\theta)
	\]
	are mutually absolutely continuous to $\nu$. Fix $x\in X$. Since $\nu_{\theta,x}(\ell_{\theta,x})=|\nu_{\theta,x}|>0$ for all $\theta\in\Theta_x$, we conclude that $\nu^x(Z_x)>0$ and therefore $\nu(Z_x)>0$.

	By construction, $Z_x$ is a union of lines passing through $x$, each of them intersecting $\supp(\nu)$ in dimension $\ge s-d+2$ by \eqref{eq:large-dim-on-lines}. Let $S_x\subset S^{d-1}$ denote the set of directions of these lines; this is a compact set by construction. Then $\pi_x(Z_x)=S_x$.

	Let $\eta$ be a small constant to be chosen shortly, and consider a dichotomy. If $\hdim(S_x) \le d-2+\eta$ for all $x\in X$, then $\hdim(\pi_x(Z_x))\le d-2+\eta$ even though $\nu(Z_x)>0$, and this happens for all $x$ in a set of positive $\mu$-measure. If $\eta$ is sufficiently small, this contradicts Proposition \ref{prop:orponen-refined}. Suppose now that $\hdim(S_x) > d-2+\eta$ for some $x\in X$. It follows from \cite[Theorem 2.10.25]{Federer69}, applied to the Lipschitz map $y\mapsto \pi_x(y)$ defined on $Z_x\cap\supp(\nu)$ with range $S_x$, that
	\[
		\hdim(\supp(\nu)) \ge \hdim(Z_x\cap\supp(\nu)) >s-d+2+d-2+\eta=s+\eta.
	\]
	This is also a contradiction to our assumption $\hdim(\supp(\nu))\le s+\eta$. We conclude that the counter-assumption \eqref{eq:counter-assumption-lines-slices} cannot hold.
\end{proof}

\subsection{Radial projections in higher dimensions}

By putting together all the machinery developed so far, we can deduce the following thin tubes estimate in higher dimensions:
\begin{theorem} \label{thm:radial-projection}
	Fix $d\geq 3$ and $k\in\{2,\ldots,d-1\}$. Let $\mu,\nu\in\cP(\R^d)$ be measures with disjoint supports such that $\cE_s(\mu)<\infty$, $\cE_s(\nu)<\infty$ and $\hdim(\supp(\nu))<s+\eta$ for some $s\in (k-1/k-\eta,k]$, where $\eta=\eta(d)>0$ is a sufficiently small constant. Assume that there is no $k$-plane $W\in\mathbb{A}(\R^d,k)$ with $\mu(W)\nu(W)>0$. Then $(\mu,\nu)$ has $(k-1+\phi(s-k+1)-\e)$-thin tubes for all $\e>0$, where $\phi$ is the function from \eqref{eq:def-phi}.
\end{theorem}
\begin{proof}
	Let $W$ be a minimal plane such that $\mu(W)\nu(W)>0$. By assumption, $\dim(W)\ge k+1$. Note that if $T$ is a $\delta$-tube in $\R^d$, then $T\cap W$ is contained in a $\delta$-tube in $W$. By restricting $\mu,\nu$ to $W$ and replacing $d$ by $\dim(W)$, we may thus assume that $\mu(H)\nu(H)=0$ for all hyperplanes $H\in\mathbb{A}(\R^d,d-1)$.

	Note that the hypotheses are preserved for the projections of $\mu$, $\nu$ to a generic subspace $V$ of dimension $k+1$. Indeed, the finiteness of the $s$-energy of $P_V\mu$, $P_V\nu$ is guaranteed by Theorem \ref{thm:projection}, the dimension of $\supp(\nu)$ cannot increase, and a (relative) hyperplane $H\subset V$ of positive measure for $P_V\mu,P_V\nu$ would lift to a hyperplane $P_V^{-1}(H)$ of positive measure for $\mu,\nu$. Hence, thanks to Lemma \ref{lem:thin-projection}, we may assume that $k=d-1$ for the rest of the proof.

	In light of Theorem \ref{thm:radial-inductive}, we only need to show that, for every $\e>0$, $(\mu_{V,x},\nu_{V,x})$ has $(\phi(s-d+2)-\e)$-thin tubes for $(\gamma_{d,2}\times\mu)$-almost all $(V,x)$. Let $(V,x)$ be a $(\gamma_{d,2}\times \mu)$-typical pair.
	We know from Theorem \ref{thm:slicing} and Lemma  \ref{lem:abs-cont-on-cond-meas} that $\cE_{s-d+2}(\mu_{V,x})<\infty$, $\cE_{s-d+2}(\nu_{V,x})<\infty$ and therefore, by Lemma \ref{lem:energy-Frostman}, suitable restrictions $\mu'_{V,x},\nu'_{V,x}$ satisfy a Frostman condition with exponent $s-d+2$. By Proposition \ref{prop:no-lines-slices}  $\mu_{V,x}$ and $\nu_{V,x}$ are not supported on a common line (and hence neither are $\mu'_{V,x},\nu'_{V,x}$). The conclusion then follows from Corollary \ref{cor:planar-thin-tubes}.
\end{proof}

We can now complete the proof of our radial projection estimate for general sets:
\begin{proof}[Proof of Theorem \ref{thm:radial-main-gral}]
	Let $\mu,\nu$ be Frostman measures on $X$ of exponent $\hdim(X)-\e$ with disjoint supports. If either measure gives positive mass to a $k$-plane $V$, then taking $y\in X\setminus V$ and noticing that $\pi_y|_V$ is a smooth function, we have
	\[
		\hdim(\pi_y X) \ge \hdim(\pi_y(X\cap V)) \ge \hdim(X\cap V) \ge \hdim(X)-\e,
	\]
	which is better than needed. We may then suppose that $\mu(W)\nu(W)=0$ for all $W\in\mathbb{A}(\R^d,k)$.

	If $k\ge 2$, we may then apply Theorem \ref{thm:radial-projection} to conclude the proof. If $k=1$, then by projecting to a generic $2$-plane we may assume that $d=2$. This case was already considered in Theorem \ref{thm:planar-general}.
\end{proof}

In odd dimension $d\ge 5$, the range $(k-1/k-\eta,k]$ with $k=(d+1)/2$ does not include $d/2$, which prevents a direct application of Theorem \ref{thm:radial-projection} to study distance sets of sets of dimension $d/2$. The next corollary of Theorem \ref{thm:radial-projection} provides a dichotomy that allows us to bypass this issue:
\begin{corollary} \label{cor:radial-high-dim}
	Fix $d\geq 3$,  $k\in \{2,\ldots,d-2\}$, and $s\in (k+1/2-\eta,k+1]$, where $\eta>0$ is a small universal constant.

	Let $\mu,\nu\in\cP(\R^d)$ with disjoint supports and $\cE_s(\mu)<\infty$, $\cE_s(\nu)<\infty$. Then at least one of the following holds:
	\begin{enumerate}
		\item There is $x\in\supp(\mu)$ such that $|\Delta^x(\supp(\nu))|>0$.
		\item $(\mu,\nu)$ has $(k+\phi(s-k)-\e)$-thin tubes for all $\e>0$.
	\end{enumerate}
\end{corollary}
\begin{proof}
	We know from Theorem \ref{thm:slicing} that
	\begin{equation} \label{eq:finite-s-k+1-energy}
		\cE_{s-k+1}(\mu_W)<\infty, \quad \cE_{s-k+1}(\nu_W)<\infty
	\end{equation}
	for $\lambda_{d,d-k+1}$-almost all $W$. Note that $s-k+1\in (3/2-\eta,2]$.  Let $W$ be a $\lambda_{d,d-k+1}$-typical plane.

	Suppose first that there is $V\in\mathbb{A}(W,2)$ such that $\mu_W(V)\nu_W(V)>0$. Then the measures $\mu_V,\nu_V$ have finite $(s-k+1)$-energy and their supports are subsets of $\supp(\mu),\supp(\nu)$. Applying the results on the Falconer distance set problem in the plane \cite{Liu19, GIOW20} we get that there is $x\in\supp(\mu_V)$ such that $|\Delta^x(\supp(\nu_V))|>0$, so the first alternative holds.

	We may thus assume that $\mu_W(V)\nu_W(V)=0$ for all $V\in\mathbb{A}(W,2)$. By \eqref{eq:finite-s-k+1-energy}, we may apply Theorem \ref{thm:radial-projection} to $\mu_W,\nu_W$  with $k=2$ and $s-k+1$ in place of $s$, and conclude that $(\mu_W,\nu_W)$ has $(1+\phi(s-k)-\e)$-thin tubes. Since this holds for $\lambda_{d,d-k+1}$-almost all $W$, we may apply Theorem \ref{thm:radial-inductive} (with $k-1$ in place of $k$) to finish the proof.
\end{proof}

\subsection{Radial projections for roughly Ahlfors regular sets}

We will use the following slicing lemma for box dimension; see \cite[Proposition 3]{FalconerJarvenpaa99}.
\begin{lemma} \label{lem:box-dim-slices}
	Let $E\subset\R^d$ be a bounded Borel set, and fix $k\in\{1,\ldots,d-1\}$. Then
	\[
		\ubdim(E\cap V) \le \min(\ubdim(E)-k,0) \quad\text{for $\lambda_{d,d-k}$-almost all $V$}.
	\]
\end{lemma}

Our next result is an analog of Theorem  \ref{thm:radial-projection} for sets of (roughly) equal Hausdorff and packing dimension.
\begin{theorem} \label{thm:radial-Ahlfors}
	Fix $d\geq 3$ and  $k\in\{2,\ldots,d-1\}$. Let $\mu,\nu\in\cP(\R^d)$ be measures with disjoint support and $\cE_s(\mu)<\infty$, $\cE_s(\nu)<\infty$ for some $s\in (k-1/k-\eta,k]$, where $\eta=\eta(d)>0$ is a sufficiently small constant. Assume that there is no $k$-plane $W\in\mathbb{A}(\R^d,k)$ with $\mu(W)\nu(W)>0$. Further, suppose that $\ubdim(\supp(\mu))\le s+\e$, $\ubdim(\supp(\nu))\le s+\e$. Then $(\mu,\nu)$ have $(s-\zeta)$-thin tubes, where $\zeta=\zeta(\e)\to 0$ as $\e\to 0$.
\end{theorem}
\begin{proof}
	The proof is very similar to the proof of Theorem  \ref{thm:radial-projection}.  By the same reduction as in that proof, we may assume that $d=k+1$ (note that the assumptions on the upper box dimensions of the supports are preserved under all projections). By Theorem \ref{thm:radial-inductive}, it is enough to show that $(\mu_{V,x},\nu_{V,x})$ have $(s-(d-2)-\zeta)$-thin tubes for $(\gamma_{d,2}\times\mu)$-almost all $(V,x)$, where $\zeta$ can be made arbitrarily small (depending on $\e$). We know from Theorem \ref{thm:slicing} and Lemma \ref{lem:box-dim-slices} that
	\begin{align*}
		\cE_{s-d+2}(\mu_W) & <\infty,      \\
		\cE_{s-d+2}(\nu_W) & <\infty,      \\
		\ubdim(\supp\mu_W) & \le s-d+2+\e, \\
		\ubdim(\supp\nu_W) & \le s-d+2+\e,
	\end{align*}
	for $\lambda_{d,2}$-almost all $W$ and therefore, by Lemma \ref{lem:abs-cont-on-cond-meas}, also for $W=V+x$ for $(\gamma_{d,2}\times\mu)$-almost all $(V,x)$. We also know from Proposition \ref{prop:no-lines-slices}  that $\mu_{V+x}$ and $\nu_{V+x}$ are not supported on a common line for $(\gamma_{d,2}\times\mu)$-almost all $(V,x)$.

	Fix a pair $(V,x)$ satisfying the above conclusions. First using Lemma \ref{lem:energy-Frostman} and then removing dyadic squares in $\cD_m$ of mass $\le 2^{-m(s-d+2+O(\e))}$ as in the proof of Theorem \ref{thm:abstract-proj-Hausdorff}, we obtain suitable restrictions $\mu'_{V,x},\nu'_{V,x}$ of $\mu_{V,x},\nu_{V,x}$ satisfying the assumptions of Corollary \ref{cor:planar-thin-tubes} (with $u=s-d+2$ and $O(\e)$ in place of $\e$). Then $(\mu'_{V,x},\nu'_{V,x})$ have $s-d+2-\zeta$-thin tubes, where $\zeta\to 0$ as $\e\to 0$, as we wanted to see.
\end{proof}

\begin{corollary} \label{cor:radial-high-dim-Ahlfors}
	Fix $d\geq 3$,  $k\in \{2,\ldots,d-2\}$, and $s\in (k+1/2-\eta,k+1]$, where $\eta>0$ is a small universal constant.

	Let $\mu,\nu\in\cP(\R^d)$ with disjoint supports and $\cE_s(\mu)<\infty$, $\cE_s(\nu)<\infty$. Suppose $\mu(W)\nu(W)=0$ for all affine planes $V\in\mathbb{A}(\R^d,k+1)$. Assume further that $\ubdim(\supp\mu)\le s+\e$, $\ubdim(\supp\nu)\le s+\e$. Then at least one of the following holds:
	\begin{enumerate}
		\item There is $x\in\supp(\nu)$ such that $|\Delta^x(\supp(\nu))|>0$.
		\item $(\mu,\nu)$ have $(s-\zeta(\eps))$-thin tubes, where $\zeta(\eps)\to 0$ as $\e\to 0$.
	\end{enumerate}
\end{corollary}
\begin{proof}
	The argument is nearly identical to that of Corollary \ref{cor:radial-high-dim}. The only change in the proof is that instead of applying Theorem \ref{thm:radial-projection} to $\mu_W,\nu_W$ for a $\lambda_{d-k+1}$-typical plane $W$, we apply Theorem \ref{thm:radial-Ahlfors} instead. This is justified since we know from Lemma \ref{lem:box-dim-slices} that
	\[
		\ubdim(\supp\mu_W),\ubdim(\supp\nu_W)\le s-k+1+\e
	\]
	for $\lambda_{d,d-k+1}$-almost all $W$.
\end{proof}

We conclude this section by completing the proof of Theorem \ref{thm:radial-main-AD}.
\begin{proof}[Proof of Theorem \ref{thm:radial-main-AD}]
	The case $k=1$, $d=2$ was already established in Theorem \ref{thm:planar-AD}, and the case $k=1$, $d>2$ can be reduced to that one via a generic projection to a $2$-plane.

	Fix, then, $k\in\{2,\ldots,d-1\}$. Write $s=\hdim(X)$, and let $\mu,\nu$ be $(s-\e)$-Frostman measures on $X$ with disjoint supports. By the exact same reasoning as in the proof of Theorem \ref{thm:radial-main-gral}, we may assume that $\mu(V)=\nu(V)=0$ for all $V\in\mathbb{A}(\R^d,k)$. Using that
	\[
		\pdim(X) = \inf \left\{ \sup_i\ubdim(X_i): X\subset \cup_i X_i\right\},
	\]
	we may find subsets $Y,Z$ with $\mu(Y)>0$, $\nu(Z)>0$ and $\ubdim(Y)<s+\e$, $\ubdim(Z)<s+\e$. By assumption, $\mu_Y(V)=\nu_Z(V)=0$ for all $k$-planes $V$. We can then apply Theorem \ref{thm:radial-Ahlfors} to conclude that $(\mu_Y,\nu_Z)$ have $(s-\zeta)$-thin tubes, for $\zeta=o_{\e\to 0}(1)$. In particular, there are $x\in\supp(\mu_Y)\subset X$ and a non-negative bounded Borel function $g$ with $\int g\, d\nu>0$ such that $\pi_x(gd\nu)$ is Frostman with exponent $s-o_{\e\to 0}(1)$.  Since $gd\nu$ is supported on $Z\subset X$, $\hdim(\pi_x X)$ can be made arbitrarily close to $s$, as desired.
\end{proof}

\section{Distance set estimates in dimension $\ge 3$}
\label{sec:high-dim-dist}

\subsection{Sets of equal Hausdorff and packing dimension}
\label{subsec:high-dim-dist-Ahlfors}

\begin{proof}[Proof of Theorem \ref{thm:distance-conj-Ahlfors}]
	It is enough to show that
	\[
		\sup_{y\in X}\hdim(\Delta^y X)= 1.
	\]
	Indeed, if this is true then consider the set
	\[
		Y_\eta=\{ y\in X: \hdim(\Delta^y X) < 1-\eta\}.
	\]
	We know that $\hdim(Y_\eta)<\hdim(X)$, for otherwise we would get a contradiction by applying the claim to $Y_\eta$ (which would satisfy the same assumptions as $X$). Considering a sequence $\eta_n\to 0$ we see that if $\mathcal{H}^{d/2}(X)>0$, then $\hdim(\Delta^y X)=1$ for $y\in X\setminus\cup_n Y_{\eta_n}$, which is a set of full $\mathcal{H}^{d/2}|_X$-measure.

	The planar case was already established in Theorem \ref{thm:planar-AD}, so we may assume that $d\ge 3$.  Let $\mu,\nu$ be Frostman measures on $X$ with separated supports and $\cE_{d/2-\e}(\mu)<\infty$, $\cE_{d/2-\e}(\nu)<\infty$ with $\e$ sufficiently small.

	Let $k=d/2-1$ if $d$ is even and $k=(d-1)/2$ if $d$ is odd. Suppose first that there is a plane $W\in\mathbb{A}(\R^d,k+1)$ such that $\mu(W)>0$. Then $\hdim(X\cap W)\ge d/2-\e$ and it follows from the known results on the Falconer distance set problem (see e.g. \cite[Corollary 8.4]{PeresSchlag00}) that there is even a $y\in Y$ with $|\Delta^y X|>0$.

	We may then assume that $\mu(W)=0$ for every $W\in \mathbb{A}(\R^d,k+1)$, and  can thus apply Corollary \ref{cor:radial-high-dim-Ahlfors} (with $s=d/2-\e$). If the first option in the dichotomy holds, we are done. Otherwise, we get that $(\mu,\nu)$ have $(d/2-\zeta)$-thin tubes, where $\zeta\to 0$ as $\e\to 0$. After applying Lemma \ref{lem:thin-tubes-to-strong-thin-tubes}, we may assume that $(\mu,\nu)$ have $(d/2-\zeta)$-strong thin tubes and the corresponding set $\{ (x,y):x\in X',y\in Y_x\}$ is Borel.  We can then apply  Theorem \ref{thm:abstract-proj-Hausdorff} to $X'$ and the family of projections $\{ \Delta^y\}_{y\in\supp\nu}$. The required adaptedness follows from Lemma \ref{lem:adapted-Falconer}. The conclusion is that there is $y\in \supp\nu\subset X$ such that $\hdim(\Delta^y X)\ge 1 - O(\zeta)$.
\end{proof}

\subsection{Combinatorial estimates}

The following combinatorial proposition will be used to establish Theorem \ref{thm:dim3}.
\begin{prop} \label{prop:highdim-comb}
	Fix $d\ge 3$. Let
	\[
		\frac{d-1}{2} < s < \frac{d}{2},
	\]
	and define
	\begin{equation} \label{eq:def-Ds-high-dim}
		D_s(t) = \max(\min(1,s+t-(d-1)),t/d), \quad \forall t\in [0,d].
	\end{equation}
	Then
	\[
		\sup_{\tau>0}\Sigma_{\tau}(D_s;d/2) \ge \frac{s+1}{d+1}.
	\]
\end{prop}

\begin{proof}
	The proof is similar to that of Proposition \ref{prop:planar-combinatorial} but slightly simpler. Without loss of generality, we can assume that $f(1)=d/2$ for otherwise we can replace $f(t)$ by $\min\{ f(t), d/2\}$.
	We will only use the following properties of $D_s$:
	\begin{enumerate}[(i)]
		\item \label{it:1} $D_s(t)\ge t/d$ for all $t$.
		\item \label{it:2}  $D_s(d-s)=1$.
		\item \label{it:3} $D_s(t) \ge s+t-(d-1)$ if $t\le d-s$.
	\end{enumerate}

	Fix a small $\e>0$ for the rest of the proof.

	Let \begin{equation}\label{eq: a}
		A = \{a\in (0,1/2]: (f, a, 2a) \text{ is } (d-s)\text{--superlinear}\}.
	\end{equation}

	\noindent \textbf{Case A}. The set $A$ is empty.

	This case is similar to case I.A.1 from the proof of Proposition \ref{prop:planar-combinatorial}. Apply Lemma~\ref{lem:superlinear-decomposition} on the interval $[d\e,1]$, with $\rho=\e^3$, to obtain $\tau=\tau(\e,d)>0$ and a collection  $\{[a_j, a_{j+1}]\}_{j=1}^J$ with $a_1=d\e$, $a_{J+1}=1$.  We claim that in this case we can merge the intervals $[a_j,a_{j+1}]$ using Lemma~\ref{lem:merge-increasing} in a way such that the resulting $\sigma_j \leq d-s+\e$ and $a_{j+1}\leq 2a_j$.

	Start with $a_1=d\e$. Pick the $K$ such that $a_K\le 2a_1$ and $a_{K+1}> 2a_1$. Merge the intervals $[a_1,a_2],\ldots, [a_{K-1},a_K]$ using Lemma \ref{lem:merge-increasing}; rename the resulting intervals and numbers $\sigma_j$ to keep the same notation.

	We claim that $\sigma_1\le d-s+\e$. Indeed, if this is not the case then, since the $\sigma_j$ are increasing, $f(x) \ge f(a_1)+(d-s+\e)(x-a_1)$ for all $x\in [a_1,a_K]$. Since $2a_1\le a_{K+1}\le a_K+ \e^3$ (by the choice $\rho=\e^3$) and $f$ is non-decreasing, for $x\in [a_K,2a_1]$ we have
	\[
		f(x) \ge f(a_K) \ge f(a_1)+(d-s+\e)(a_K-a_1) \ge f(a_1)+(d-s)(x-a_1),
	\]
	assuming $\e$ was taken sufficiently small in terms of $d-s$. Then $a_1\in A$, contradicting the fact that $A$ is empty.

	If $a_2\ge 1/2$ we stop; otherwise, we repeat the argument with $a_2$ in place of $a_1$. Continuing inductively, we obtain a sequence of $\tau$-allowable intervals $([a_j,a_{j+1}])_{j=1}^K$ such that $\sigma_j\le d-s+\e$ and $a_{K+1}\ge 1/2$.

	Now we merge the remaining intervals $[a_j,a_{j+1}]$ with $j\ge K+1$ using Lemma \ref{lem:merge-increasing}, and keep the same notation. Since $f(x)\ge (d/2) x$ and we are assuming $f(1)=d/2$, we must have $\sigma_j\le d/2\le d-s$ for $j=J$ (the last one) and therefore for all $j$.

	By Lemma \ref{lem:superlinear-decomposition}\eqref{it:ii:exhausts} and the fact that the intervals were obtained from the original ones by repeated applications of Lemma \ref{lem:merge-increasing}, we see that
	\[
		\sum_{j=1}^J (a_{j+1}-a_j)\sigma_j \ge f(1) - f(d\e) - \e(1-d\e) \ge d/2- (d^2+1)\e.
	\]
	Applying property~\eqref{it:3} of $D_s$ to this set of intervals, and  using that $s>\frac{d-1}{2}$, we conclude that
	\[
		\Sigma_{\tau}(D; f)  \ge(s-(d-1))+ d/2 - O_d(\e) \ge \frac{s+1}{d+1},
	\]
	provided $\e$ is small enough.

	\smallskip

	\noindent \textbf{Case B}. The set $A$ defined in \eqref{eq: a} is nonempty. Let $a=\max A$. Then $a<1/2$ since $d-s>d/2$ and $f(t)\geq dt/2, f(1)=d/2$.

	If
	\begin{equation}\label{eq: 2a}
		f(2a) \leq f(a)+(d-s)a,
	\end{equation}
	then write $b=c=2a$. Otherwise, let
	\begin{equation}\label{eq: b}
		b=\sup\{ x\in (a, 2a): f(x)-f(a) \leq (d-s)(x-a) \}.
	\end{equation}
	Such $b$ must exist because  $f$ is piecewise linear and $a=\max A$.

	Let
	\begin{equation}\label{eq: 2c}
		c=\min\{ x \in [2a, 1]: f(x) \leq f(b)+ (x-b)(d-s)\}.
	\end{equation}
	The number $c$ is well defined since $1$ is a valid value of $x$ (since $d-s>d/2$, $f(x)\ge (d/2)x$ and $f(1)=d/2$). Note that $(f,b,2a)$ is $(d-s)$-superlinear by the definitions of $a$ and $b$, and $(f,2a,c)$ is $(d-s)$-superlinear by the definition of $c$. Hence $(f, b, c)$ is $(d-s)$--superlinear, which implies that $c\le 2b$ (for otherwise $a<b\in A$). For later reference we record that
	\begin{equation} \label{eq:f(c)}
		f(c) = f(a)+(d-s)(c-a),
	\end{equation}
	which follows from  \eqref{eq: b} and \eqref{eq: 2c}.
	\smallskip

	\noindent \textbf{Case B.1.} $a\geq \frac{d(2s-d+1)}{2s(d+1)}.$
	Apply Lemma~\ref{lem:superlinear-decomposition} to the intervals $[0,a]$ and $[c, 1]$. On each of the obtain intervals we apply the bound \eqref{it:1}. On the intervals $[a,b]$ and $[b,c]$ we apply \eqref{it:2}. In all cases we skip the corresponding interval if the length is $\le \e$. Recalling Lemma \ref{lem:superlinear-decomposition}\eqref{it:ii:exhausts}, we deduce that
	\begin{align*}
		\Sigma_{\tau}(D; f)+O(\e) & \geq (c-a) +\frac{d/2-(c-a)(d-s)}{d}                                                  \\
		                          & \geq \frac{1}{2}+\frac{s(c-a)}{d} \geq \frac{1}{2}+\frac{as}{d} \geq \frac{s+1}{d+1}.
	\end{align*}

	\smallskip

	\noindent \textbf{Case B.2.} $a\leq \frac{d(2s-d+1)}{2s(d+1)}.$
	Apply the argument in Case A to the interval $[c,1]$, which works since $c>\max A$, to obtain a union of consecutive intervals $\{[a_j, a_{j+1}]\}$ covering $[c,1]$ and satisfying
	\begin{enumerate}
		\item $(f, a_j, a_{j+1})$ is $\sigma_j$--superlinear with $\sigma_j\leq d-s + \e$,
		\item $\tau \le  a_{j+1}-a_j \le a_j$,
		\item $\sum_j \sigma_j (a_{j+1} -a_j) \geq d/2-f(c)-\e. $
	\end{enumerate}
	On the interval $[0,a]$ we apply Lemma \ref{lem:superlinear-decomposition} to obtain a suitable collection of $\tau$-allowable intervals , while on $[a,b]$, $[b,c]$ we use that $f$ is $(d-s)$-superlinear as above. Note that $\tau=\tau(\e,d)>0$. As usual, we ignore intervals of length $\le\e$. Using \eqref{it:1} on $[0,a]$, \eqref{it:2} on $[a,c]$ and \eqref{it:3} on $[c,1]$, we conclude
	\begin{align*}
		\Sigma_{\tau}(D; f) +O(\e)
		 & \geq \frac{f(a)}{d} + (c-a) +(1-c)(s-d+1) + \frac{d}{2}-f(c)                                            \\
		 & \overset{\eqref{eq:f(c)}}{\geq} \frac{f(a)}{d} +(c-a) + (1-c)(s-d+1)  + \frac{d}{2}-(f(a) + (d-s)(c-a)) \\
		 & = \frac{d}{2} - \big(1-\frac{1}{d}\big)f(a) + (1-a)(s-d+1)                                              \\
		 & \overset{f(a)\le da}{\geq} s+1-\frac{d}{2}-as \geq \frac{s+1}{d+1}.
	\end{align*}
\end{proof}

\subsection{Hausdorff dimension estimates in $\R^3$}

In this section we prove Theorem \ref{thm:dim3}. In the proof we will make use of the following higher dimensional variant of Proposition \ref{prop:Orponen-radial}.

\begin{prop} \label{prop:Orponen-radial-highdim}
	Let $\mu,\nu\in\cP(\R^d)$ satisfy
	\[
		\mu(V^{(r)}), \nu(V^{(r)})\lesssim r^s,\quad V\in\mathbb{A}(\R^d,k),r>0.
	\]
	for some $k\in\{0,\ldots, d-2 \}$ and $s>0$. Assume also that $\nu(W)=0$ for all $W\in\mathbb{A}(\R^d,k+1)$.

	Then there is $t=t(s)>0$ such that for every $\e>0$ there is $K=K(\mu,\nu,\e)>0$ such that for all $x$ in a set $X$ of $\mu$-measure $\ge 1-\e$ there is a set $Y_x$ with $\nu(Y_x)>1-\e$ and
	\[
		\nu(Y_x\cap W^{(r)})\le K\, r^t,\quad r>0,
	\]
	for all $W\in\mathbb{A}(\R^d,k+1)$ passing through $x$. Moreover, $\{ (x,y):x\in X,y\in Y_x\}$ is compact.
\end{prop}
See \cite[Theorem B.1]{Shmerkin23} for the proof (which is an adaptation of the method in \cite{Orponen19}).

\begin{proof}[Proof of Theorem \ref{thm:dim3}]
	Fix a small $\e>0$ and let $\mu,\nu$ be $(3/2-\e)$-Frostman measures on $X$ with disjoint supports. In particular,
	\[
		\mu(\ell^{(r)}),\nu(\ell^{(r)}) \lesssim r^{1/2-\e},\quad \ell  \in\mathbb{A}(\R^3,1), r>0.
	\]
	As usual, the case in which $\nu$ gives positive mass to a $2$-plane can be reduced to known estimates for the Falconer problem on the plane, so we may assume that this is not the case. Then $(\mu,\nu)$ satisfy the assumptions of Proposition \ref{prop:Orponen-radial-highdim} for $k=1$. We also know from Theorem \ref{thm:radial-projection} that $(\mu,\nu)$ have $(1+\phi(1/2-\e)-\e)$-thin tubes. By Lemma \ref{lem:thin-tubes-to-strong-thin-tubes}  $(\mu,\nu)$ have $(1+\phi(1/2-\e)-\e)$-strong thin tubes. Putting these two facts together, and applying Proposition \ref{prop:Orponen-radial-highdim}, we get that (provided $\e$ is small enough) there are $c>0$, a set $X'$ with $\mu(X')\ge c$, and sets $Y_x$ with $\nu(Y_x)\ge c$ such that $\{ (x,y):x\in X', y\in Y_x\}$ is compact, and if $x\in X'$ then
	\begin{align}
		\pi_x\nu_{Y_x}(B_r)     & \lesssim r^{1+\phi(1/2)-O(\e)},\nonumber                                         \\
		\pi_x\nu_{Y_x}(V^{(r)}) & \lesssim r^t \quad (V\in \mathbb{G}(\R^3,2)), \label{eq:plane-non-concentration}
	\end{align}
	where $t>0$ is the number provided by Proposition \ref{prop:Orponen-radial-highdim}. It follows from Lemmas \ref{lem:adapted-Falconer} and \ref{lem:adapted-Bourgain} that $\pi_x\nu_{Y_x}$ is $D_{1+\phi(1/2)-O(\e)}$-adapted, where $D_s$ is the function from Proposition \ref{prop:highdim-comb}.

	Consider the family of projections $\mathcal{F}=\{ \Delta^y\}_{y\in\supp{\nu}}$ and the set $X'$. Since $V_x=\pi_x$ as we have noted many times, $\mathcal{F}$ is $D_{1+\phi(1/2)-O(\e)}$-adapted. The conclusion follows from Theorem \ref{thm:abstract-proj-Hausdorff}, Proposition \ref{prop:highdim-comb}, and Lemma \ref{lem:Sigma-Lipschitz}.
\end{proof}

\subsection{Box dimension estimates: proof of Theorem \ref{thm:box-dim-high-dim}}

\subsubsection{Some remarks on the proof}
\label{subsubsec:remarks-box}

In this section we prove Theorem \ref{thm:box-dim-high-dim}. Before doing so, we explain why the proof in dimension $d=3$ does not extend to higher dimensions and how we overcome the issue to obtain a lower box dimension estimate. The central issue is that if $X$ has dimension $d/2$ in $\R^d$ for $d\ge 4$, the analog of the hyperplane non-concentration \eqref{eq:plane-non-concentration} does not necessarily hold (one would need $\hdim(X)>d-2$ for this to hold as a consequence of Proposition \ref{prop:Orponen-radial-highdim}). Therefore we cannot apply Bourgain's projection theorem in the form given by Lemma \ref{lem:adapted-Bourgain}, and this is a key tool in the proof of Theorem \ref{thm:box-dim-high-dim}. If \eqref{eq:plane-non-concentration} does hold for the Frostman measure $\nu$ (for hyperplanes $V$), which is a rather mild condition, then the proof of Theorem \ref{thm:dim3} goes through almost verbatim, and we get a Hausdorff dimension version of Theorem \ref{thm:box-dim-high-dim}.

To overcome this issue for general sets, we appeal to the (much more elementary) Lemma \ref{lem:entropy-proj-elementary}. Crucially, this lemma only requires a mild hyperplane non-concentration at a single scale. Even this mild condition may fail at a given scale, but we can bypass this issue by restricting one of the Frostman measures to a small neighborhood of a minimal plane with large mass, and removing a large neighborhood of this plane from the other Frostman measure (see \S\ref{subsub:dichotomy} for details). This works at a fixed small scale, but because the pieces we remove depend on the scale, and do not have decaying mass, we cannot conclude that the pinned distance measure is robust, which would be needed to establish a Hausdorff dimension estimate.

\subsubsection{Initial reductions}

We now start the proof of Theorem \ref{thm:box-dim-high-dim}. Fix a small parameter $\e$. Let $\mu,\nu$ be Frostman measures of exponent $d/2-\e$ supported on $X$ with separated supports. We may assume that $\mu$ gives zero mass to all hyperplanes, otherwise we even get from earlier results on the Falconer problem that there is a pinned distance set of positive measure, see \cite[Theorem 2.7]{DuZhang19}.

By Corollary \ref{cor:radial-high-dim} applied with $s=d/2-\e$ and $k=(d-1)/2$ or $k=d/2-1$ depending on whether $d$ is odd or even, we may assume that $(\mu,\nu)$ has $(s_d-O(\e))$-thin tubes for
\[
	s_d = \left\{ \begin{array}{ll}
		\frac{d-1}{2} + \phi(1/2) & \text{ if $d$ is odd}  \\
		\frac{d}{2} +\phi(1)-1    & \text{ if $d$ is even}
	\end{array}
	\right..
\]
By Lemma \ref{lem:thin-density}, after a further restriction we may assume that $(\mu,\nu)$ have $(s_d-O(\e),K,1-\e)$-strong thin tubes for some $K>0$. Let $X', (Y_x)_{x\in X'}$ be the sets corresponding to the definition of strong thin tubes, so that in particular $\{ (x,y):x\in X', y\in Y_x\}$ is compact.

Define
\[
	Y_0 = \big\{ y\in \supp(\nu): \mu\{ x: y\in Y_x\}>1-3\e\big\}.
\]
By Fubini, $(\mu\times\nu)\{(x,y):y\in Y_x\}>1-2\eps$. Splitting this set into fibers over $Y_0$ and $Y_0^c$, we see that
\[
	1 - 2\eps \leq \nu(Y_0) + (1-\nu(Y_0))(1-3\eps).
\]
This yields $\nu(Y_0)>1/3$.

It is enough to show that
\[
	\sum_m\nu\left\{y\in Y_0:|\Delta^y(\supp\mu)|_{2^{-m}} \le 2^{(1/2+c_d-O(\e^{1/2}))m}\right\}<\infty.
\]
In turn, it suffices to show that if $Y\subset Y_0$ is a Borel subset with $\nu(Y)>1/m^2$, then there is $y\in Y$ such that
\begin{equation} \label{eq:counting-number-pinned-dist-to-prove}
	\log |\Delta^y(\supp\mu)|_{2^{-m}} \ge (1/2+c_d-O(\e^{1/2}))m.
\end{equation}
In fact, it is enough to consider values of $m$ that are multiples of $T$, where $T$ is large enough in terms of $\e$. Write, then,  $m=T\ell$, where $\ell$ is large enough.

\subsubsection{A dichotomy for the set of vantage points}
\label{subsub:dichotomy}

Let $\e'\ll\e$ be another small parameter, to be chosen later. We fix $Y$ and $m$ as above for the rest of the proof. We consider two cases. If
\begin{equation} \label{eq:case-i-small-hyperplane}
	\nu_Y(H^{(2^{-\e' m})})\le m^{-2}\quad \text{for all }H\in \mathbb{A}(\R^d,d-1),
\end{equation}
we set $Z=Y$. Otherwise, let $k$ be the smallest dimension for which there exists $V\in \mathbb{A}(\R^d,k)$ with
\begin{equation} \label{eq:case-ii-large-plane}
	\nu_Y(V^{(2^{-\e' m})})> m^{-2(d-k+1)}.
\end{equation}
The Frostman condition on $\nu$ ensures that $k\ge 1$. In this case, we set $Z=Y\cap V^{(2^{-\e' m})}$. Note that, in either case, $Z\subset Y\subset Y_0$ and $\nu(Z)>m^{-2d-2}/2$.

Using compactness, let $c_\e$ be small enough that
\begin{equation} \label{eq:hyperplate-Ce}
	\mu(H^{(c_\e)})\le \e \quad\text{for all }H\in\mathbb{A}(\R^d,d-1).
\end{equation}
Suppose that we are in the case that \eqref{eq:case-ii-large-plane} holds and $H$ is a hyperplane that hits both $Z$ and $\supp(\mu)\setminus V^{(c_\e)}$. Then $H^{(2^{-\e' m})}\cap V^{(2^{-\e' m})}$ can be covered by $O_\e(1)$ plates $W_j^{(2^{-\e' m})}$ with $\dim W_j=k-1$. By the minimality of $V$,
\begin{equation} \label{eq:case-ii-small-hyperplane}
	\nu_Z(H^{(2^{-\e' m})}) \lesssim_\e \frac{m^{-2(d-k)}}{m^{-2(d-k+1)}} = m^{-2}.
\end{equation}
This will serve as a substitute of \eqref{eq:case-i-small-hyperplane}.

\subsubsection{Definition of sets of ``good points'' $G$ and $A$}

In the remaining of the proof, we will only care about tubes of width $\ge 2^{-m}$. If $x,x'$ are at distance $O(2^{-m})$ apart, then for $r\ge 2^{-m}$ an $r$-tube through $x'$ is contained in an $O(r)$-tube through $x$. We can therefore assume that the sets $Y_x$ are constant over each cube in $\cD_m$.

In case \eqref{eq:case-i-small-hyperplane}, let
\[
	G= \{ (x,y)\in \supp(\mu)\times Z: y\in Y_x \}.
\]
Otherwise, in case \eqref{eq:case-ii-large-plane}, we define
\[
	G = \{ (x,y)\in \supp(\mu)\times Z: y\in Y_x, x\notin V^{(c_\e)}\}.
\]
Since $Z\subset Y_0$, in light of \eqref{eq:hyperplate-Ce} in either case we have
\[
	(\mu\times\nu_Z)(G) \ge 1- \e.
\]

Let
\begin{equation} \label{eq:def-A}
	A=\big\{x: \nu_Z\{y:(x,y)\in G\} > 1-\sqrt{\e}\big\}.
\end{equation}
Arguing by Fubini as for $Y_0$ above, we see that $\mu(A)>1-\sqrt{\e}$.

\subsubsection{Regularity of radial projections of $\nu$ with center in $A$}

Fix $x\in A$. Since $\nu(Z)\gtrsim m^{-2d-2}$,
\[
	\nu(Y_x\cap Z) \gtrsim m^{-2d-2},
\]
and therefore, by the definition of thin tubes,
\[
	\nu_{Y_x\cap Z}(T) \le \frac{\nu(Y_x\cap T)}{\nu(Y_x\cap Z)} \lesssim m^{-2d-2} r^{s_d-O(\e)}
\]
for all $r\in (0,\delta_0]$ and all $r$-tubes $T$ through $x$.
We deduce that
\begin{equation} \label{eq:radial-projection-Frostman}
	\pi_x \nu_{Y_x\cap Z}(B_r) \lesssim  2^{O(\e)m} r^{s_d} \quad( r\in [2^{-m},1]).
\end{equation}
It also follows from \eqref{eq:case-i-small-hyperplane} and \eqref{eq:case-ii-small-hyperplane} that (always for a fixed $x\in A$),
\begin{equation} \label{eq:radial-projection-thin-hyperplates}
	\pi_x \nu_{Y_x\cap Z}(H^{(2^{-\e' m})}) \lesssim_\e m^{-2} \quad\text{for all } H\in\mathbb{A}(\R^d,d-1).
\end{equation}

\subsubsection{Regularization of the measure and application of the combinatorial estimate}

Apply Lemma \ref{lem:decomposition-uniform} to $\mu^{(m)}$  to obtain a set $X\subset \supp\mu(O(2^{-m}))$ with $\mu^{(m)}(X)\ge 2^{-\e m}$ (taking $T$ large enough) and such that $\mu':=\mu^{(m)}_X$ is $(\beta;T)$-uniform for some $\beta\in [0,d]^\ell$ and (using $\mu(A)>1-\sqrt{\e}$) satisfies
\begin{equation} \label{eq:mu-A-large}
	\mu'(A^{(m)})>1-O(\sqrt{\e}).
\end{equation}
Here $A^{(m)}$ is the union of the cubes in $\cD_m(A)$.

As in \S\ref{subsubsec:setup-abstract-proj-thm}, let $f:[0,1]\to [0,d]$ be the function such that
\[
	f(j/\ell) =\beta_1+\ldots+\beta_j,
\]
and interpolates linearly between $j/\ell$ and $(j+1)\ell$. By the Frostman condition on $\mu$,
\[
	\mu'(B_r) \lesssim 2^{\e m} r^{d/2-\e} \quad (r\in [2^{-m},1]).
\]
Together with \eqref{eq:reg-measure-cubes}, and arguing as in \S\ref{subsubsec:setup-abstract-proj-thm}, this implies that there is a piecewise linear, $d$-Lipschitz function $\tilde{f}\in\mathcal{L}_{d,d/2-\e^{1/2}}$ which agrees with $f$ on $[O(\sqrt{\e}),1]$.

Applying Lemma \ref{prop:highdim-comb} together with Lemma \ref{lem:Sigma-Lipschitz} to $\tilde{f}$, we obtain a number $\tau=\tau(\e)>0$, non-overlapping $\tau$-allowable intervals $\{[a_j,b_j]\}$ and numbers $\{ \sigma_j\}$ such that $(f,a_j,b_j)$ is $\sigma$-superlinear and
\begin{equation} \label{eq:highdim-comb-appl}
	\sum_j D_{s_d}(\sigma_j)(b_j-a_j) \ge  \frac{1}{2}+c_d-O(\e^{1/2}).
\end{equation}
Perturbing $\tau$ and the $\sigma_j$ without changing the above properties (which can be achieved by taking $\ell$ large enough), we may assume that $a_j=A_j/\ell$, $b_j=B_j/\ell$ for some integers $A_j,B_j$.

Fix $x\in X\cap A$, and write $m_j=T(B_j-A_j)$ and
\[
	\mu'_{x,j}=\left((\mu')^{\cD_{TA_j}(x)}\right)^{(m_j)}
\]
for simplicity. Then $\mu'_{x,j}$ is $(\beta_{[A_j,B_j)};T)$-uniform, and using \eqref{eq:reg-measure-cubes} the $\sigma_j$-superlinearity of $f$ on $[a_j,b_j]$ translates to a Frostman condition
\begin{equation} \label{eq:mu'-Frostman}
	\mu'_{x,j}(B_r) \lesssim_T r^{\sigma_j},\quad r\in [2^{-m_j},1].
\end{equation}

\subsubsection{Application of projection theorems and sets of bad vantage points}

Fix $j$. Suppose first that $s_d+\sigma_j-(d-1) \ge \sigma_j/d$. It follows from \eqref{eq:radial-projection-Frostman}, \eqref{eq:mu'-Frostman} and an inspection of the proof of Lemma \ref{lem:adapted-Falconer} (replacing $K$ by $\delta^{-O(\e)}$) that
\[
	\| P_{\pi_x y}\mu'_{x,j}\|_2^2 \le 2^{-m_j(1+O(\e)-\min\{s_d+\sigma_j-(d-1),1\})}
\]
for all $y$ outside of a set $\bad_j(x)$ with $\nu_{Y_x\cap Z}(\bad_j)\le 2^{-\e m_j}$. Combining Lemmas \ref{lem:robust-to-entropy} and \ref{lem:robust-l2-energy} (or simply applying Jensen's inequality) we get that if $y\notin \bad_j(x)$, then
\begin{equation} \label{eq:entropy-scale-j-1}
	H_{m_j}(P_{\pi_x y}\mu'_{x,j}) \ge m_j(\min\{s_d+\sigma_j-(d-1),1\}-O(\e)).
\end{equation}
(In fact, this holds even for robust entropy.)

Suppose now that $s_d+\sigma_j-(d-1) \le \sigma_j/d$. It follows from \eqref{eq:mu'-Frostman} that
\[
	H_{m_j}(\mu_{x,j}')\ge \sigma_j m_j - O_T(1).
\]
At this point we choose $\e'=\tau\e$ (recall that $\tau=\tau(\e)$), so that \eqref{eq:radial-projection-thin-hyperplates} yields
\begin{equation} \label{eq:entropy-scale-j-2}
	\pi_x\nu_{Y_x\cap Z}(H^{(2^{-\e m_j})}) \lesssim_\e m^{-2}\quad\text{for all } H\in\mathbb{A}(\R^d,d-1).
\end{equation}
Using this together with \eqref{eq:mu'-Frostman}  and Lemma \ref{lem:entropy-proj-elementary}, we get that
\[
	H_{m_j}(P_{\pi_x y}\mu'_{x,j}) \ge m_j\big(\frac{\sigma_j}{d}-O(\e)\big)
\]
for all $y$ outside of a set $\bad_j(x)$ with $\nu_{Y_x\cap Z}(\bad_j(x))\lesssim_\e m^{-2}$.

Combining \eqref{eq:entropy-scale-j-1} and \eqref{eq:entropy-scale-j-2} and recalling the definition of $D_s$ in \eqref{eq:def-Ds-high-dim}, we get
\begin{equation} \label{eq:entropy-scale-j}
	H_{m_j}(P_{\pi_x y}\mu'_{x,j}) \ge m_j(D_{s_d}(\sigma_j)-O(\e))
\end{equation}
for all $y$ outside of a set $\bad_j(x)$ with $\nu_{Y_x\cap Z}(\bad_j(x))\lesssim_\e m^{-2}$. This will play a role analogous to adaptedness in the proof of Theorem \ref{thm:nonlinear-abstract}.  Since, by \eqref{eq:def-A}, we have $\nu_Z(Y_x)\ge 1/2$ for $x\in A$, and $J\le \lceil 1/\tau\rceil\lesssim_\e 1$, we get
\begin{equation} \label{eq:bad-y-small}
	\nu_Z(Y_x\cap \bad(x))  \lesssim_\e m^{-2}\le \e,\quad\text{where }\bad(x)=\cup_{j=1}^J \bad_j(x).
\end{equation}
Since moving $x$ by a distance $O(2^{-m})$ only changes $\pi_y(x)$ by $O(2^{-m})$, we may assume that $\bad(x)$ is $\cD_m$-measurable and hence so is $Y_x\cap\bad(x)$. Hence, even though a priori \eqref{eq:bad-y-small} holds for $x\in A$, we may assume that it also holds for $x\in A^{(m)}$.

\subsubsection{Application of Fubini and conclusion of the proof}

Recall that, by \eqref{eq:def-A}, $\nu_Z(Y_x) \ge 1-\sqrt{\e}$ for all $x\in A$ (and hence for all $x\in A^{(m)}$). Coupled with \eqref{eq:mu-A-large}, \eqref{eq:bad-y-small} and Fubini, this yields a point $y\in Z$ such that
\begin{equation} \label{eq:mu'-A'}
	\mu'(A') \ge 1-O(\sqrt{\e}),\quad\text{where } A' =\{ x\in A^{(m)}: y\in Y_x\setminus\bad(x)\}.
\end{equation}
Applying Proposition \ref{prop:entropy-of-image-measure}, we conclude that
\begin{align*}
	H_m(\Delta^y\mu') & \ge - O(J)   + \int_{A'} \sum_j  H_{m_j}\left(P_{\pi_y(x)}\mu'_{x,j}\right) \,d\mu'(x)                    \\
	                  & \overset{\eqref{eq:entropy-scale-j}}{\ge} -O(\tau^{-1}) +  \mu'(A)\sum_j m_j(D_{s_d}(\sigma_j)-O(\e))     \\
	                  & \overset{\eqref{eq:highdim-comb-appl},\eqref{eq:mu'-A'}}{\ge} (1-O(\e^{1/2}))\big(\frac{1}{2}+c_d\big) m.
\end{align*}

Hence $\log |\Delta^y\supp(\mu')|_{2^{-m}}\ge (1-O(\e^{1/2}))(1/2+c_d) m$. Since $\supp(\mu')$ is contained in the $O(2^{-m})$ neighborhood of $\supp(\mu)$, this shows that \eqref{eq:counting-number-pinned-dist-to-prove} holds, finishing the proof.

\section{Extensions, remarks and applications}
\label{sec:extensions}

\subsection{Distances for other norms}
\label{subsec:othernorms}

All the results on pinned distance sets that we have discussed so far extend from the Euclidean norm to any $C^\infty$ norm $N:\R^d\to [0,\infty)$ whose unit sphere has everywhere positive Gaussian curvature. Indeed, if we let $N_y(x)=N(x-y)$, then
\[
	V_y(x):= \dir(\tfrac{d}{dx}N_y(x)) = \psi(\pi_x(y)),
\]
where $\psi:S^{d-1}\mapsto S^{d-1}$, $\theta\mapsto \dir N'(\theta)$ is a diffeomorphism of the (Euclidean) unit sphere by the assumptions on $N$; see \cite[Proof of Theorem 1.1]{Shmerkin23} for the rather simple proof. In fact, for this it is enough that $N$ is $C^2$. Hence any Frostman condition on $\pi_x\nu_Y$ (or equivalently any statement about thin tubes) translates immediately into a Frostman condition with the same exponent for $\psi\pi_x\nu_Y$ (the multiplicative constant gets worse by a $O_N(1)$ factor). Then all the arguments involving pinned distances in the plane sets carry over verbatim (in fact, in the planar case it is enough to assume that the norm is $C^2$).

Some of the proofs in higher dimensions require some additional remarks. Under the present assumptions on $N$, if $X\subset\R^2$ is a Borel set with $\hdim(X)>5/4$, then there are many $y\in X$ such that $\{ N(x-y):x\in X\}$ has positive Lebesgue measure, see \cite[Theorem 7.1]{GIOW20} - this fact is needed to extend Corollaries \ref{cor:radial-high-dim} and  \ref{cor:radial-high-dim-Ahlfors} to the setting of more general norms (this uses the $C^\infty$ assumption). Likewise, if $k\ge 3$ and a Borel set $X\subset\R^k$ has Hausdorff dimension $>k-1$ (in fact, dimension $>(k+1)/2$ is enough), then by \cite[Theorem 8.3]{PeresSchlag00} there are many $y\in X$ such that $\{ N(x-y):x\in X\}$ has positive Lebesgue measure - with this fact and the above remarks, the proof of Theorem \ref{thm:distance-conj-Ahlfors} goes through unchanged.

\subsection{The distance set problem for Assouad dimension}

There is interest in the Falconer distance set problem also for other notions of dimension. In particular, J.~Fraser \cite{Fraser18, Fraser23} has investigated the problem for Assouad dimension; see \cite{Fraser20} for the definition and further background on Assouad dimension, which we denote $\adim$.

It was observed by Fraser, see \cite[Proof of Theorem 11.1.3]{Fraser20}, that if $\hdim(\Delta(X))\ge s$ whenever $X\subset\R^d$ is a compact set with $\hdim(X)=\pdim(X)=t$, then one has the implication $\adim(X)\ge t\Rightarrow \adim(\Delta(X))\ge s$ for subsets $X\subset\R^d$. In fact, \cite[Theorem 11.1.3]{Fraser20} corresponds to the case $d=2$, $s=1$ and $t>1$, but the argument carries over without changes to the more general situation; see also \cite[Theorem 2.5]{Fraser20}. Putting this together with Theorem \ref{thm:distance-conj-Ahlfors}, we deduce:

\begin{corollary}
	Let $X\subset\R^d$ be a set with $\adim(X)\ge d/2$. Then $\adim(\Delta(X))=1$.
\end{corollary}

The planar case of this result was established by Fraser in \cite{Fraser23}. The corollary is new in dimensions $d\ge 3$ and settles the Falconer distance set conjecture for Assouad dimension (stated explicitly as part of \cite[Conjecture 11.1.2]{Fraser20}). It also solves \cite[Question 17.9.2]{Fraser20} affirmatively and makes progress on \cite[Question 17.9.1]{Fraser20}.

\subsection{Connections with the discretized sum-product problem}

We conclude by remarking that our results have connections with other well known problem in discretized additive combinatorics, such as the discretized sum-product problem. For example, taking $X=A\times A$ for some Borel set $A\subset\R$, then Theorems \ref{thm:planar-AD} and \ref{thm:planar-general} provide lower bounds for the Hausdorff dimension of $(A-A)/(A-A)$ in terms of the Hausdorff dimension of $A$, both for general sets $A$ and for sets of equal Hausdorff and packing dimension. A single-scale discretized version can also be deduced from Corollary \ref{cor:discretized}. We hope to explore this connection in more detail in future work.


\begin{thebibliography}{10}

	\bibitem{Bourgain03}
	Jean Bourgain.
	\newblock On the {E}rd{\H o}s-{V}olkmann and {K}atz-{T}ao ring conjectures.
	\newblock {\em Geom. Funct. Anal.}, 13(2):334--365, 2003.

	\bibitem{Bourgain10}
	Jean Bourgain.
	\newblock The discretized sum-product and projection theorems.
	\newblock {\em J. Anal. Math.}, 112:193--236, 2010.

	\bibitem{DIOWZ21}
	Xiumin Du, Alex Iosevich, Yumeng Ou, Hong Wang, and Ruixiang Zhang.
	\newblock An improved result for {F}alconer's distance set problem in even
	dimensions.
	\newblock {\em Math. Ann.}, 380(3-4):1215--1231, 2021.

	\bibitem{DuZhang19}
	Xiumin Du and Ruixiang Zhang.
	\newblock Sharp {$L^2$} estimates of the {S}chr\"{o}dinger maximal function in
	higher dimensions.
	\newblock {\em Ann. of Math. (2)}, 189(3):837--861, 2019.

	\bibitem{Falconer85}
	Kenneth~J. Falconer.
	\newblock On the {H}ausdorff dimensions of distance sets.
	\newblock {\em Mathematika}, 32(2):206--212 (1986), 1985.

	\bibitem{FalconerJarvenpaa99}
	Kenneth~J. Falconer and Maarit J\"{a}rvenp\"{a}\"{a}.
	\newblock Packing dimensions of sections of sets.
	\newblock {\em Math. Proc. Cambridge Philos. Soc.}, 125(1):89--104, 1999.

	\bibitem{FasslerOrponen14}
	Katrin F\"{a}ssler and Tuomas Orponen.
	\newblock On restricted families of projections in {$\Bbb R^3$}.
	\newblock {\em Proc. Lond. Math. Soc. (3)}, 109(2):353--381, 2014.

	\bibitem{Federer69}
	Herbert Federer.
	\newblock {\em Geometric measure theory}.
	\newblock Die Grundlehren der mathematischen Wissenschaften, Band 153.
	Springer-Verlag New York Inc., New York, 1969.

	\bibitem{Fraser18}
	Jonathan~M. Fraser.
	\newblock Distance sets, orthogonal projections and passing to weak tangents.
	\newblock {\em Israel J. Math.}, 226(2):851--875, 2018.

	\bibitem{Fraser20}
	Jonathan~M. Fraser.
	\newblock {\em Assouad Dimension and Fractal Geometry}.
	\newblock Cambridge Tracts in Mathematics. Cambridge University Press, 2020.

	\bibitem{Fraser23}
	Jonathan~M. Fraser.
	\newblock A nonlinear projection theorem for {A}ssouad dimension and
	applications.
	\newblock {\em J. Lond. Math. Soc. (2)}, 107(2):777--797, 2023.

	\bibitem{GIOW20}
	Larry Guth, Alex Iosevich, Yumeng Ou, and Hong Wang.
	\newblock On {F}alconer's distance set problem in the plane.
	\newblock {\em Invent. Math.}, 219(3):779--830, 2020.

	\bibitem{GuthKatz15}
	Larry Guth and Nets~Hawk Katz.
	\newblock On the {E}rd{\H{o}}s distinct distances problem in the plane.
	\newblock {\em Ann. of Math. (2)}, 181(1):155--190, 2015.

	\bibitem{GSW19}
	Larry Guth, Noam Solomon, and Hong Wang.
	\newblock Incidence estimates for well spaced tubes.
	\newblock {\em Geom. Funct. Anal.}, 29(6):1844--1863, 2019.

	\bibitem{He20}
	Weikun He.
	\newblock Orthogonal projections of discretized sets.
	\newblock {\em J. Fractal Geom.}, 7(3):271--317, 2020.

	\bibitem{HochmanShmerkin12}
	Michael Hochman and Pablo Shmerkin.
	\newblock Local entropy averages and projections of fractal measures.
	\newblock {\em Ann. of Math. (2)}, 175(3):1001--1059, 2012.

	\bibitem{KatzTao01}
	Nets~Hawk Katz and Terence Tao.
	\newblock Some connections between {F}alconer's distance set conjecture and
	sets of {F}urstenburg type.
	\newblock {\em New York J. Math.}, 7:149--187, 2001.

	\bibitem{Kaufman68}
	Robert Kaufman.
	\newblock On {H}ausdorff dimension of projections.
	\newblock {\em Mathematika}, 15:153--155, 1968.

	\bibitem{KeletiShmerkin19}
	Tam\'{a}s Keleti and Pablo Shmerkin.
	\newblock New bounds on the dimensions of planar distance sets.
	\newblock {\em Geom. Funct. Anal.}, 29(6):1886--1948, 2019.

	\bibitem{Liu19}
	Bochen Liu.
	\newblock An {$L^2$}-identity and pinned distance problem.
	\newblock {\em Geom. Funct. Anal.}, 29(1):283--294, 2019.

	\bibitem{Liu20}
	Bochen Liu.
	\newblock Hausdorff dimension of pinned distance sets and the {$L^2$}-method.
	\newblock {\em Proc. Amer. Math. Soc.}, 148(1):333--341, 2020.

	\bibitem{Mattila95}
	Pertti Mattila.
	\newblock {\em Geometry of sets and measures in {E}uclidean spaces}, volume~44
	of {\em Cambridge Studies in Advanced Mathematics}.
	\newblock Cambridge University Press, Cambridge, 1995.
	\newblock Fractals and rectifiability.

	\bibitem{Mattila15}
	Pertti Mattila.
	\newblock {\em Fourier analysis and {H}ausdorff dimension}, volume 150 of {\em
			Cambridge Studies in Advanced Mathematics}.
	\newblock Cambridge University Press, Cambridge, 2015.

	\bibitem{Orponen17}
	Tuomas Orponen.
	\newblock On the distance sets of {A}hlfors-{D}avid regular sets.
	\newblock {\em Adv. Math.}, 307:1029--1045, 2017.

	\bibitem{Orponen19}
	Tuomas Orponen.
	\newblock On the dimension and smoothness of radial projections.
	\newblock {\em Anal. PDE}, 12(5):1273--1294, 2019.

	\bibitem{OrponenShmerkin23}
	Tuomas Orponen and Pablo Shmerkin.
	\newblock On the {H}ausdorff dimension of {F}urstenberg sets and orthogonal
	projections in the plane.
	\newblock {\em Duke Math. J.}, 172(18):3559--3632, 2023.

	\bibitem{PeresSchlag00}
	Yuval Peres and Wilhelm Schlag.
	\newblock Smoothness of projections, {B}ernoulli convolutions, and the
	dimension of exceptions.
	\newblock {\em Duke Math. J.}, 102(2):193--251, 2000.

	\bibitem{Shmerkin19}
	Pablo Shmerkin.
	\newblock On the {H}ausdorff dimension of pinned distance sets.
	\newblock {\em Israel J. Math.}, 230(2):949--972, 2019.

	\bibitem{Shmerkin21}
	Pablo Shmerkin.
	\newblock Improved bounds for the dimensions of planar distance sets.
	\newblock {\em J. Fractal Geom.}, 8(1):27--51, 2021.

	\bibitem{Shmerkin23}
	Pablo Shmerkin.
	\newblock A non-linear version of {B}ourgain's projection theorem.
	\newblock {\em J. Eur. Math. Soc. (JEMS)}, 25(10):4155--4204, 2023.

	\bibitem{Srivastava98}
	Shashi~Mohan Srivastava.
	\newblock {\em A course on {B}orel sets}, volume 180 of {\em Graduate Texts in
			Mathematics}.
	\newblock Springer-Verlag, New York, 1998.

	\bibitem{Walters82}
	Peter Walters.
	\newblock {\em An introduction to ergodic theory}, volume~79 of {\em Graduate
			Texts in Mathematics}.
	\newblock Springer-Verlag, New York-Berlin, 1982.

\end{thebibliography}

\end{document}